\documentclass[11pt, reqno]{amsart} \textwidth=14.9cm \oddsidemargin=1cm \evensidemargin=1cm
\usepackage{amsmath,amssymb,amscd,mathrsfs,stmaryrd,wasysym}
\usepackage{amsmath}
\usepackage{amsxtra}
\usepackage{amscd}
\usepackage{amsthm}
\usepackage{amsfonts}
\usepackage{amssymb}
\usepackage{eucal}
\usepackage{color}
\usepackage{graphicx}%
\usepackage{bm}
\DeclareMathOperator\Gedom{\,{\underline{\kern-.1ex{\blacktriangleright}\kern-0.1ex}}\,}

\newtheorem{theorem}{Theorem}[section]
\newtheorem{lemma}[theorem]{Lemma}
\newtheorem{proposition}[theorem]{Proposition}
\newtheorem{corollary}[theorem]{Corollary}
\theoremstyle{definition}
\newtheorem{definition}[theorem]{Definition}
\newtheorem{remark}[theorem]{Remark}

\definecolor{A}{rgb}{.75,1,.75}

\numberwithin{equation}{section}

\begin{document}
\title[Yokonuma-Schur algebras]{Yokonuma-Schur algebras}
\author[Weideng Cui]{Weideng Cui}
\address{School of Mathematics, Shandong University, Jinan, Shandong 250100, P.R. China.}
\email{cwdeng@amss.ac.cn}

\begin{abstract}
In this paper, we define the Yokonuma-Schur algebra $\text{YS}_{q}(r,n)$ as the endomorphism algebra of a permutation module for the Yokonuma-Hecke algebra $\text{Y}_{r,n}(q).$ We prove that $\text{YS}_{q}(r,n)$ is cellular by constructing an explicit cellular basis following the approach in [DJM], and we further show that it is a quasi-hereditary cover of $\text{Y}_{r,n}(q)$ in the sense of Rouquier following [HM2]. We also introduce the tilting modules for $\text{YS}_{q}(r,n).$ In the appendix, we define and study the cyclotomic Yokonuma-Schur algebra in a similar way.
\end{abstract}



\maketitle
\medskip

\section{Introduction}
\subsection{}
The Yokonuma-Hecke algebra was first introduced by Yokonuma [Yo] as a centralizer algebra associated to the permutation representation of a Chevalley group $G$ with respect to a maximal unipotent subgroup of $G.$ Juyumaya [Ju1] gave a new presentation of the Yokonuma-Hecke algebra, which is commonly used for studying this algebra.

The Yokonuma-Hecke algebra $\text{Y}_{r,n}(q)$ is a quotient of the group algebra of the modular framed braid group $(\mathbb{Z}/r\mathbb{Z})\wr B_{n},$ where $B_{n}$ is the braid group of type $A$ on $n$ strands. It can also be regraded as a deformation of the group algebra of the complex reflection group $G(r,1,n),$ which is isomorphic to the wreath product $(\mathbb{Z}/r\mathbb{Z})\wr \mathfrak{S}_{n},$ where $\mathfrak{S}_{n}$ is the symmetric group on $n$ letters. It is well-known that there exists another deformation of the group algebra of $G(r,1,n),$ the Ariki-Koike algebra $H_{r,n}$ [AK]. The Yokonuma-Hecke algebra $\text{Y}_{r,n}(q)$ is quite different from $H_{r,n}.$ For example, the Iwahori-Hecke algebra of type $A$ is canonically a subalgebra of $H_{r,n},$ whereas it is an obvious quotient of $\text{Y}_{r,n}(q),$ but not an obvious subalgebra of it.

In the past few years, many people are largely motivated to study $\text{Y}_{r,n}(q)$ in order to construct its associated knot invariant; see the papers [Ju2], [JuL] and [ChL]. In particular, Juyumaya and Kannan [Ju2, JuK] found a basis of $\text{Y}_{r,n}(q),$ and then defined a Markov trace on it.

Some other people are particularly interested in the representation theory of $\text{Y}_{r,n}(q),$ and also its application to knot theory. Chlouveraki and Poulain d'Andecy [ChPA1] gave explicit formulas for all irreducible representations of $\text{Y}_{r,n}(q)$ over $\mathbb{C}(q)$, and obtained a semisimplicity criterion for it. In their subsequent paper [ChPA2], they defined and studied the affine Yokonuma-Hecke algebra $\widehat{Y}_{r,n}(q)$ and the cyclotomic Yokonuma-Hecke algebra $Y_{r,n}^{d}(q),$ and constructed several bases for them, and then showed how to define Markov traces on these algebras. Moreover, they gave the classification of irreducible representations of $Y_{r,n}^{d}(q)$ in the generic semisimple case, defined the canonical symmetrizing form on it and computed the associated Schur elements directly.

\subsection{}
Recently, Jacon and Poulain d'Andecy [JaPA] constructed an explicit algebraic isomorphism between the Yokonuma-Hecke algebra $\text{Y}_{r,n}(q)$ and a direct sum of matrix algebras over tensor products of Iwahori-Hecke algebras of type $A,$ which is in fact a special case of the results by G. Lusztig [Lu, Section 34].
This allows them to give a description of the modular representation theory of $\text{Y}_{r,n}(q)$ and a complete classification of all Markov traces for it. Chlouveraki and S\'{e}cherre [ChS, Theorem 4.3] proved that the affine Yokonuma-Hecke algebra is a particular case of the pro-$p$-Iwahori-Hecke algebra defined by Vign\'eras in [Vi].

Espinoza and Ryom-Hansen [ER] gave a new proof of Jacon and Poulain d'Andecy's isomorphism theorem by giving a concrete isomorphism between $\text{Y}_{r,n}(q)$ and Shoji's modified Ariki-Koike algebra $\mathcal{H}_{r,n}.$ Moreover, they showed that $\text{Y}_{r,n}(q)$ is a cellular algebra by giving an explicit cellular basis. Combining the results of [DJM] with those of [ER], we [C1] proved that the cyclotomic Yokonuma-Hecke algebra $Y_{r,n}^{d}(q)$ is cellular by constructing an explicit cellular basis, and showed that the Jucys-Murphy elements for $Y_{r,n}^{d}(q)$ are JM-elements in the abstract sense introduced by Mathas [Ma3].

We [CW] have established an equivalence between a module category of the affine (resp. cyclotomic) Yokonuma-Hecke algebra $\widehat{Y}_{r,n}(q)$ (resp. $Y_{r,n}^{d}(q)$) and its suitable counterpart for a direct sum of tensor products of affine Hecke algebras of type $A$ (resp. cyclotomic Hecke algebras), which allows us to give the classification of simple modules of affine Yokonuma-Hecke algebras and of the associated cyclotomic Yokonuma-Hecke algebras over an algebraically closed field of characteristic $p$ when $p$ does not divide $r,$ and also describe the classification of blocks for these algebras. In addition, the modular branching rules for cyclotomic (resp. affine) Yokonuma-Hecke algebras are obtained, and they are further identified with crystal graphs of integrable modules for affine lie algebras of type $A.$ In a subsequent paper, we [C2] have established an explicit algebra isomorphism between the affine Yokonuma-Hecke algebra $\widehat{Y}_{r,n}(q)$ and a direct sum of matrix algebras over tensor products of affine Hecke algebras of type $A$. As an application, we proved that $\widehat{Y}_{r,n}(q)$ is affine cellular in the sense of Koenig and Xi, and studied its homological properties.

\subsection{}
In [DJM], they constructed a cellular basis for the cyclotomic $q$-Schur algebra $\mathcal{S}(\Lambda)$ by firstly constructing a cellular basis for the Ariki-Koike algebra $H_{r,n}.$ They further obtained a complete set of non-isomorphic irreducible $\mathcal{S}(\Lambda)$-modules and showed that it is quasi-hereditary. Now, there exists a cellular basis on $\text{Y}_{r,n}(q)$ by [ER], it is natural to try to define and study the corresponding Schur algebra for the Yokonuma-Hecke algebra $\text{Y}_{r,n}(q)$ by using this cellular basis.

In this paper, we will define the Yokonuma-Schur algebra $\text{YS}_{q}(r,n)$ as the endomorphism algebra of a permutation module associated to the Yokonuma-Hecke algebra $\text{Y}_{r,n}(q).$ Combining the results of [DJM] with those of [SS], we prove that $\text{YS}_{q}(r,n)$ is cellular by constructing an explicit cellular basis, and further prove that it is quasi-hereditary. We also investigate the indecomposable tilting modules for $\text{YS}_{q}(r,n)$ and prove that they are self-dual.

This paper is organized as follows. In Section 2, we recall the definition of the Yokonuma-Hecke algebra $\text{Y}_{r,n}(q)$ and the construction of a cellular basis of $\text{Y}_{r,n}(q)$ following [ER]. In Section 3, we will define the Yokonuma-Schur algebra $\text{YS}_{q}(r,n)$ as the endomorphism algebra of a permutation module associated to the Yokonuma-Hecke algebra $\text{Y}_{r,n}(q).$ We prove that $\text{YS}_{q}(r,n)$ is cellular by constructing an explicit cellular basis, and further prove that it is quasi-hereditary by combining the results of [DJM] with those of [SS]. In Section 4, following the approach in [HM2], we will construct an exact functor from the category of $\text{YS}_{q}(r,n)$-modules to the category of $\text{Y}_{r,n}(q)$-modules. In Section 5, we introduce the tilting modules for $\text{YS}_{q}(r,n)$ and the closely related Young modules for $\text{Y}_{r,n}(q)$ following [Ma2]. In the appendix, we will generalize these results to define and study the cyclotomic Yokonuma-Schur algebra by using the cellular basis of $Y_{r,n}^{d}(q)$ constructed in [C1]. Since this approach is very similar, we only mention the main results and skip all the details.

Many ideas of this paper originate from the references [DJM, Ma2, SS], although it should be noted that the basic set-up here is different from theirs; anyhow, we expect that the Yokonuma-Schur algebra and its cyclotomic analog defined here deserve further study.

\section{Cellular Bases for Yokonuma-Hecke algebras}

In this section, we recall the definition of the Yokonuma-Hecke algebra $\text{Y}_{r,n}(q)$ and the construction of a cellular basis of $\text{Y}_{r,n}(q)$ presented in [ER, Section 4].

Let $r, n\in \mathbb{N},$ $r\geq1,$ and let $\zeta=e^{2\pi i/r}.$ Let $q$ be an indeterminate. Let $\mathfrak{S}_{n}$ be the symmetric group on $n$ letters, which acts on the set $\{1,2,\ldots,n\}$ on the right by convention.

Let $\mathcal{R}=\mathbb{Z}[\frac{1}{r}][q,q^{-1},\zeta].$ The Yokonuma-Hecke algebra $\text{Y}_{r,n}=\text{Y}_{r,n}(q)$ is an $\mathcal{R}$-associative algebra generated by the elements $t_{1},\ldots,t_{n},g_{1},\ldots,g_{n-1}$ satisfying the following relations:
\begin{equation}\label{rel-def-Y1}\begin{array}{rclcl}
g_ig_j\hspace*{-7pt}&=&\hspace*{-7pt}g_jg_i && \mbox{for all $i,j=1,\ldots,n-1$ such that $\vert i-j\vert \geq 2$;}\\[0.1em]
g_ig_{i+1}g_i\hspace*{-7pt}&=&\hspace*{-7pt}g_{i+1}g_ig_{i+1} && \mbox{for all $i=1,\ldots,n-2$;}\\[0.1em]
t_it_j\hspace*{-7pt}&=&\hspace*{-7pt}t_jt_i &&  \mbox{for all $i,j=1,\ldots,n$;}\\[0.1em]
g_it_j\hspace*{-7pt}&=&\hspace*{-7pt}t_{js_i}g_i && \mbox{for all $i=1,\ldots,n-1$ and $j=1,\ldots,n$;}\\[0.1em]
t_i^r\hspace*{-7pt}&=&\hspace*{-7pt}1 && \mbox{for all $i=1,\ldots,n$;}\\[0.2em]
g_{i}^{2}\hspace*{-7pt}&=&\hspace*{-7pt}1+(q-q^{-1})e_{i}g_{i} && \mbox{for all $i=1,\ldots,n-1$,}
\end{array}
\end{equation}
where $s_{i}$ is the transposition $(i,i+1)$, and for each $1\leq i\leq n-1$,
$$e_{i} :=\frac{1}{r}\sum\limits_{s=0}^{r-1}t_{i}^{s}t_{i+1}^{-s}.$$

Note that the elements $e_{i}$ are idempotents in $\text{Y}_{r,n}.$ The elements $g_{i}$ are invertible, with the inverse given by
\begin{equation}\label{inverse}
g_{i}^{-1}=g_{i}-(q-q^{-1})e_{i}\quad\mbox{for~all}~i=1,\ldots,n-1.
\end{equation}

Let $w\in \mathfrak{S}_{n},$ and let $w=s_{i_1}\cdots s_{i_{r}}$ be a reduced expression of $w.$ By Matsumoto's lemma, the element $g_{w} :=g_{i_1}g_{i_2}\cdots g_{i_{r}}$ does not depend on the choice of the reduced expression of $w,$ that is, it is well-defined. Let $l$ denote the length function on $\mathfrak{S}_{n}.$ Then we have
\begin{align}\label{multiplication-formula}
g_{i}g_{w}=\begin{cases}g_{s_{i}w}& \hbox {if } l(s_{i}w)>l(w); \\
g_{s_{i}w}+(q-q^{-1})e_{i}g_{w}& \hbox {if } l(s_{i}w)<l(w).
\end{cases}
\end{align}

Using the multiplication formulae in \eqref{multiplication-formula}, Juyumaya [Ju2] has proved that the following set is an $\mathcal{R}$-basis of $\text{Y}_{r,n}$:
\begin{equation}\label{basis-theorem-2}
\mathcal{B}_{r,n}=\{t_{1}^{k_1}\cdots t_{n}^{k_n}g_{w}\:|\: 0\leq k_1,\ldots,k_{n}\leq r-1 \text{ and }w\in \mathfrak{S}_{n}\}.
\end{equation}
Thus, $\text{Y}_{r,n}$ is a free $\mathcal{R}$-module of rank $r^{n}n!.$

Set $\mathbf{s} :=\{1,2,\ldots,n\}$. Let $i, k\in \mathbf{s}$ and set \begin{equation}\label{idempotents}e_{i,k} :=\frac{1}{r}\sum\limits_{s=0}^{r-1}t_{i}^{s}t_{k}^{-s}.\end{equation} Note that $e_{i,i}=1,$ $e_{i,k}=e_{k,i},$ and that $e_{i,i+1}=e_{i}.$ It can be easily checked that
\begin{equation}\label{relations}\begin{array}{rclcl}
e_{i,k}^{2}\hspace*{-7pt}&=&\hspace*{-7pt}e_{i,k} && \mbox{for all $i,k=1,\ldots,n$,}\\[0.1em]
t_{i}e_{j,k}\hspace*{-7pt}&=&\hspace*{-7pt}e_{j,k}t_{i} && \mbox{for all $i,j,k=1,\ldots,n$,}\\[0.1em]
e_{i,j}e_{k,l}\hspace*{-7pt}&=&\hspace*{-7pt}e_{k,l}e_{i,j} &&  \mbox{for all $i,j,k,l=1,\ldots,n$,}\\[0.1em]
e_ie_{k,l}\hspace*{-7pt}&=&\hspace*{-7pt}e_{s_i(k), s_i(l)}e_i && \mbox{for all $i=1,\ldots,n-1$ and $k,l=1,\ldots,n$,}\\[0.1em]
e_{j,k}g_{i}\hspace*{-7pt}&=&\hspace*{-7pt}g_{i}e_{js_{i},ks_{i}} && \mbox{for all $i=1,\ldots,n-1$ and $j,k=1,\ldots,n$.}
\end{array}
\end{equation}
In particular, we have $e_{i}g_{i}=g_{i}e_{i}$ for all $i=1,2,\ldots,n-1.$

For any nonempty subset $I\subseteq \mathbf{s}$ we define the following element $E_{I}$ by $$E_{I} :=\prod_{i, j\in I; i<j}e_{i,j},$$where by convention $E_{I}=1$ if $|I|=1.$

We also need a further generalization of this. We say that the set $A=\{I_{1}, I_{2},\ldots,I_{k}\}$ is a set partition of $\mathbf{s}$ if the $I_{j}$'s are nonempty and disjoint subsets of $\mathbf{s},$ and their union is $\mathbf{s}.$ We refer to them as the blocks of $A.$ We denote by $\mathcal{SP}_{n}$ the set of all set partitions of $\mathbf{s}.$ For $A=\{I_{1}, I_{2},\ldots,I_{k}\}\in \mathcal{SP}_{n}$ we then define $E_{A} :=\prod_{j}E_{I_{j}}.$

We extend the right action of $\mathfrak{S}_{n}$ on $\mathbf{s}$ to a right action on $\mathcal{SP}_{n}$ by defining $Aw :=\{I_{1}w,\ldots,I_{k}w\}\in \mathcal{SP}_{n}$ for $w\in \mathfrak{S}_{n}.$ Then we can easily get the following lemma.
\begin{lemma}\label{2-1-lemmaaa}
For $A\in \mathcal{SP}_{n}$ and $w\in \mathfrak{S}_{n},$ we have $$g_{w}E_{A}=E_{Aw^{-1}}g_{w}.$$
In particular, if $w$ leaves invariant every block of $A,$ or more generally permutes some of the blocks of $A,$ then $g_{w}$ commutes with $E_{A}.$
\end{lemma}

$\mu=(\mu_{1},\ldots,\mu_{k})$ is called a composition of $n$ if it is a finite sequence of nonnegative integers whose sum is $n.$ A composition $\mu$ is a partition of $n$ if its parts are non-increasing. We write $\mu \models n$ (resp. $\lambda\vdash n$) if $\mu$ is a composition (resp. partition) of $n,$ and we define $|\mu| :=n$ (resp. $|\lambda| :=n$).

We associate a Young diagram to a composition $\mu,$ which is the set $$[\mu] :=\{(i,j)\:|\:i\geq 1~\mathrm{and}~1\leq j\leq \mu_{i}\}.$$ We will regard $[\mu]$ as an array of boxes, or nodes, in the plane. For $\mu \models n,$ we define a $\mu$-tableau by replacing each node of $[\mu]$ by one of the integers $1,2,\ldots,n,$ allowing no repeats.

For $\mu \models n,$ we say that a $\mu$-tableau $\mathfrak{t}$ is row standard if the entries in each row of $\mathfrak{t}$ increase from left to right. A $\mu$-tableau $\mathfrak{t}$ is standard if $\mu$ is a partition, $\mathfrak{t}$ is row standard and the entries in each column increase from top to bottom. For a composition $\mu$ of $n,$ we denote by $\mathfrak{t}^{\mu}$ the $\mu$-tableau in which $1,2,\ldots,n$ appear in increasing order from left to right along the rows of $[\mu].$

The symmetric group $\mathfrak{S}_{n}$ acts from the right on the set of $\mu$-tableaux by permuting the entries in each tableau. For any composition $\mu=(\mu_{1},\ldots,\mu_{k})$ of $n$ we define the Young subgroup $\mathfrak{S}_{\mu} :=\mathfrak{S}_{\mu_{1}}\times\cdots\times\mathfrak{S}_{\mu_{k}},$ which is the row stabilizer of $\mathfrak{t}^{\mu}.$

Let $\lambda=(\lambda_{1},\ldots,\lambda_{k})$ and $\mu=(\mu_{1},\ldots,\mu_{l})$ be two compositions of $n.$ We say that $\lambda\unrhd\mu$ if $$\sum_{i=1}^{j}\lambda_{i}\geq \sum_{i=1}^{j}\mu_{i}~~~~~~\mathrm{for~all~}j\geq 1.$$ If $\lambda\unrhd\mu$ and $\lambda\neq\mu,$ we write $\lambda\rhd\mu.$

We extend the partial order above to tableaux as follows. If $\mathfrak{v}$ is a row standard $\lambda$-tableau and $1\leq k\leq n,$ then the entries $1,2,\ldots,k$ in $\mathfrak{v}$ occupy the diagram of a composition; let $\mathfrak{v}_{\downarrow k}$ denote this composition. Let $\lambda$ and $\mu$ be two compositions of $n.$ Suppose that $\mathfrak{s}$ is a row standard $\lambda$-tableau and that $\mathfrak{t}$ is a row standard $\mu$-tableau. We say that $\mathfrak{s}$ dominates $\mathfrak{t},$ and we write $\mathfrak{s}\unrhd\mathfrak{t}$ if $\mathfrak{s}_{\downarrow k}\unrhd\mathfrak{t}_{\downarrow k}$ for all $k.$ If $\mathfrak{s}\unrhd\mathfrak{t}$ and $\mathfrak{s}\neq\mathfrak{t},$ then we write $\mathfrak{s}\rhd\mathfrak{t}.$

Following [ChPA1, Section 4.3], the combinatorial objects appearing in the representation theory of the Yokonuma-Hecke algebra $\text{Y}_{r,n}$ will be $r$-compositions (resp. $r$-partitions). By definition, an $r$-composition (resp. $r$-partition) of $n$ is an ordered $r$-tuple $\bm{\mu}=(\mu^{(1)},\mu^{(2)},\ldots,\mu^{(r)})$ of compositions (resp. partitions) $\mu^{(k)}$ such that $\sum_{k=1}^{r}|\mu^{(k)}|=n.$ We denote by $\mathcal{C}_{r,n}$ (resp. $\mathcal{P}_{r,n}$) the set of $r$-compositions (resp. $r$-partitions) of $n.$ The Young diagram $[\bm{\mu}]$ of an $r$-composition $\bm{\mu}$ is the ordered $r$-tuple of the Young diagram of its components.

Let $\bm{\mu}=(\mu^{(1)},\mu^{(2)},\ldots,\mu^{(r)})$ be an $r$-composition of $n.$ A $\bm{\mu}$-tableau $\mathfrak{t}=(\mathfrak{t}^{(1)},\ldots,\mathfrak{t}^{(r)})$ is obtained by placing each node of $[\bm{\mu}]$ by one of the integers $1,2,\ldots,n,$ allowing no repeats. We will call the number $n$ the size of $\mathfrak{t}$ and the $\mathfrak{t}^{(k)}$'s the components of $\mathfrak{t}.$

For each $\bm{\mu}\in \mathcal{C}_{r,n},$ a $\bm{\mu}$-tableau is called row standard if the numbers increase along any row (from left to right) of each diagram in $[\bm{\mu}].$ For each $\bm{\lambda}\in \mathcal{P}_{r,n},$ a $\bm{\lambda}$-tableau is called standard if the numbers increase along any row (from left to right) and down any column (from top to bottom) of each diagram in $[\bm{\lambda}].$ For $\bm{\mu}\in \mathcal{C}_{r,n},$ we denote by $\text{r-Std}(\bm{\mu})$ the set of row standard $\bm{\mu}$-tableaux of size $n$, which is endowed with an action of $\mathfrak{S}_{n}$ from the right by permuting the entries in each $\bm{\mu}$-tableau. For $\bm{\lambda}\in \mathcal{P}_{r,n},$ let $\text{Std}(\bm{\lambda})$ denote the set of standard $\lambda$-tableaux of size $n$.

For each $\bm{\mu}\in \mathcal{C}_{r,n},$ we denote by $\mathfrak{t}^{\bm{\mu}}$ the standard $\bm{\mu}$-tableau in which $1,2,\ldots,n$ appear in increasing order from left to right along the rows of the first diagram, and then along the rows of the second diagram, and so on. For each $\bm{\mu}=(\mu^{(1)},\mu^{(2)},\ldots,\mu^{(r)})\in \mathcal{C}_{r,n},$ we have a Young subgroup
$$\mathfrak{S}_{\bm{\mu}} :=\mathfrak{S}_{\mu^{(1)}}\times\mathfrak{S}_{\mu^{(2)}}\cdots\times\mathfrak{S}_{\mu^{(r)}},$$
which is exactly the row stabilizer of $\mathfrak{t}^{\bm{\mu}}.$

For each $\bm{\mu}\in \mathcal{C}_{r,n}$ and a row standard $\bm{\mu}$-tableau $\mathfrak{s}$, let $d(\mathfrak{s})$ be the unique element of $\mathfrak{S}_{n}$ such that $\mathfrak{s}=\mathfrak{t}^{\bm{\mu}}d(\mathfrak{s}).$ Then $d(\mathfrak{s})$ is a distinguished right coset representative of $\mathfrak{S}_{\bm{\mu}}$ in $\mathfrak{S}_{n},$ that is, $l(wd(\mathfrak{s}))=l(w)+l(d(\mathfrak{s}))$ for any $w\in \mathfrak{S}_{\bm{\mu}}.$ In this way, we obtain a bijection between the set $\text{r-Std}(\bm{\mu})$ of row standard $\bm{\mu}$-tableaux and the set $\mathcal{D}_{\bm{\mu}}$ of distinguished right coset representatives of $\mathfrak{S}_{\bm{\mu}}$ in $\mathfrak{S}_{n}.$

Let $\bm{\mu}\in \mathcal{C}_{r,n}$ and $\mathfrak{t}$ be a $\bm{\mu}$-tableau. For $j=1,\ldots,n$, we define $\text{p}_{\mathfrak{t}}(j)=k$ if $j$ appears in the $k$-th component $\mathfrak{t}^{(k)}$ of $\mathfrak{t}.$ When $\mathfrak{t}=\mathfrak{t}^{\bm{\mu}},$ we write $\text{p}_{\bm{\mu}}(j)$ instead of $\text{p}_{\mathfrak{t}^{\bm{\mu}}}(j).$

We now define a partial order on the set of $r$-compositions, which is similar to the case of compositions.

\begin{definition}\label{Def-def}
Let $\bm{\lambda}=(\lambda^{(1)},\lambda^{(2)},\ldots,\lambda^{(r)})$ and $\bm{\mu}=(\mu^{(1)},\mu^{(2)},\ldots,\mu^{(r)})$ be two $r$-compositions of $n$. We say that $\bm{\lambda}$ dominates $\bm{\mu},$ and we write $\bm{\lambda}\unrhd \bm{\mu}$ if and only if
$$\sum_{i=1}^{k-1}|\lambda^{(i)}|+\sum_{j=1}^{l}\lambda_{l}^{(k)}\geq \sum_{i=1}^{k-1}|\mu^{(i)}|+\sum_{j=1}^{l}\mu_{l}^{(k)}$$
for all $k$ and $l$ with $1\leq k\leq r$ and $l\geq 0.$ If $\bm{\lambda}\unrhd \bm{\mu}$ and $\bm{\lambda}\neq \bm{\mu},$ we write $\bm{\lambda}\rhd \bm{\mu}.$
\end{definition}

We now fix once and for all a total order on the set of $r$-th roots of unity via setting $\zeta_{k} :=\zeta^{k-1}$ for $1\leq k\leq r.$ Set $S :=\{\zeta_{1},\zeta_{2},\ldots,\zeta_{r}\}.$ Then we define a set partition $A_{\bm{\lambda}}\in \mathcal{SP}_{n}$ for any $r$-composition $\bm{\lambda}.$

\begin{definition}\label{definition2-3-ulambda}
Let $\bm{\lambda}=(\lambda^{(1)},\ldots,\lambda^{(r)})\in \mathcal{C}_{r,n}.$ Suppose that we choose all $1\leq i_{1}< i_{2}<\cdots < i_{p}\leq r$ such that $\lambda^{(i_{1})},$ $\lambda^{(i_{2})},\ldots,$ $\lambda^{(i_{p})}$ are the nonempty components of $[\bm{\lambda}]$. Define $a_{k} :=\sum_{j=1}^{k}|\lambda^{(i_{j})}|$ for $1\leq k\leq p.$ Then the set partition $A_{\bm{\lambda}}$ associated with $\bm{\lambda}$ is defined as $$A_{\bm{\lambda}} :=\{\{1,\ldots,a_{1}\},\{a_{1}+1,\ldots,a_{2}\},\ldots,\{a_{p-1}+1,\ldots,n\}\},$$ which may be written as $A_{\bm{\lambda}}=\{I_{1},I_{2},\ldots,I_{p}\},$ and is referred to the blocks of $A_{\bm{\lambda}}$ in the order given above.
\end{definition}

\begin{definition}\label{definition2-4-ulambda}
Let $\bm{\lambda}=(\lambda^{(1)},\ldots,\lambda^{(r)})\in \mathcal{C}_{r,n},$ and let $a_{k} :=\sum_{j=1}^{k}|\lambda^{(i_{j})}|$ $(1\leq k\leq p)$ be defined as above. Then we define $$u_{\bm{\lambda}} :=u_{a_{1},i_{1}}u_{a_{2},i_{2}}\cdots u_{a_{p},i_{p}},$$
where $u_{i,k}=\Pi_{l=1;l\neq k}^{r}(t_{i}-\zeta_{l})$ for $1\leq i\leq n$ and $1\leq k\leq r.$
\end{definition}

\begin{definition}
Let $\bm{\lambda}\in \mathcal{C}_{r,n}.$ We set $U_{\bm{\lambda}} :=u_{\bm{\lambda}}E_{A_{\bm{\lambda}}},$ and define $x_{\bm{\lambda}}=\sum_{w\in \mathfrak{S}_{\bm{\lambda}}}q^{l(w)}g_{w}.$ Then we define the element $m_{\bm{\lambda}}$ of $\text{Y}_{r,n}$ as follows: $$m_{\bm{\lambda}} :=U_{\bm{\lambda}}x_{\bm{\lambda}}=u_{\bm{\lambda}}E_{A_{\bm{\lambda}}}x_{\bm{\lambda}}.$$
\end{definition}

Let $\ast$ denote the $\mathcal{R}$-linear anti-automorphism of $\text{Y}_{r,n}$, which is determined by $g_{i}^{\ast}=g_{i}$ and $t_{j}^{\ast}=t_{j}$ for $1\leq i\leq n-1$ and $1\leq j\leq n.$

\begin{definition}
Let $\bm{\lambda}\in \mathcal{C}_{r,n},$ and let $\mathfrak{s}$ and $\mathfrak{t}$ be two row standard $\bm{\lambda}$-tableaux. We then define $m_{\mathfrak{s}\mathfrak{t}}=g_{d(\mathfrak{s})}^{\ast}m_{\bm{\lambda}}g_{d(\mathfrak{t})}.$
\end{definition}

For each $\bm{\lambda}\in \mathcal{P}_{r,n},$ let $\text{Y}_{r,n}^{\rhd \bm{\lambda}}$ be the $\mathcal{R}$-submodule of $\text{Y}_{r,n}$ spanned by $m_{\mathfrak{u}\mathfrak{v}}$ with $\mathfrak{u}, \mathfrak{v}\in \mathrm{Std}(\bm{\mu})$ for various $\bm{\mu}\in \mathcal{P}_{r,n}$ such that $\bm{\mu}\rhd\bm{\lambda}.$

\begin{theorem}\label{cellula-bases-2}
{\rm (See [ER, Theorem 20].)} The algebra $\emph{Y}_{r,n}$ is a free $\mathcal{R}$-module with a cellular basis $$\mathcal{B}_{r,n}=\{m_{\mathfrak{s}\mathfrak{t}}\:|\:\mathfrak{s}, \mathfrak{t}\in \mathrm{Std}(\lambda)~for~some~r\mathrm{-}partition~\lambda~of~n\},$$ that is, the following properties hold$:$

$\mathrm{(i)}$ The $\mathcal{R}$-linear map determined by $m_{\mathfrak{s}\mathfrak{t}}\mapsto m_{\mathfrak{t}\mathfrak{s}}$ $(m_{\mathfrak{s}\mathfrak{t}}\in \mathcal{B}_{r,n})$ is an anti-automorphism on $\emph{Y}_{r,n}$.

$\mathrm{(ii)}$ For a given $h\in \emph{Y}_{r,n},$ $\bm{\mu}\in \mathcal{P}_{r,n}$ and $\mathfrak{t}\in \mathrm{Std}(\bm{\mu}),$ there exist $r_{\mathfrak{v}\mathfrak{t}}(h)\in \mathcal{R}$ such that for all $\mathfrak{s}\in \mathrm{Std}(\bm{\mu})$, we have $$m_{\mathfrak{s}\mathfrak{t}}h\equiv\sum_{\mathfrak{v}\in \mathrm{Std}(\bm{\mu})}r_{\mathfrak{v}\mathfrak{t}}(h)m_{\mathfrak{s}\mathfrak{v}}~~~~\mathrm{mod}~\emph{Y}_{r,n}^{\rhd \bm{\mu}},$$
where $r_{\mathfrak{v}\mathfrak{t}}(h)$ may depend on $\mathfrak{v}, \mathfrak{t}$ and $h,$ but not on $\mathfrak{s}.$
\end{theorem}

For each $\bm{\lambda}\in \mathcal{P}_{r,n},$ let $\overline{m}_{\bm{\lambda}}$ be the image of $m_{\bm{\lambda}}$ under the algebra homomorphism $\text{Y}_{r,n}\rightarrow \text{Y}_{r,n}/\text{Y}_{r,n}^{\rhd \bm{\lambda}}.$ We denote by $S^{\bm{\lambda}}$ the right $\text{Y}_{r,n}$-submodule of $\text{Y}_{r,n}/\text{Y}_{r,n}^{\rhd \bm{\lambda}}$ generated by $\overline{m}_{\bm{\lambda}},$ which is called the Specht module associated to $\bm{\lambda}.$ By Theorem \ref{cellula-bases-2}, $S^{\bm{\lambda}}$ is a free $\mathcal{R}$-module with basis $\{\overline{m}_{\bm{\lambda}}g_{d(\mathfrak{t})}\:|\:\mathfrak{t}\in \mathrm{Std}(\bm{\lambda})\}.$ We can define an associative symmetric bilinear form on $S^{\bm{\lambda}}$ by $$m_{\bm{\lambda}}g_{d(\mathfrak{s})}g_{d(\mathfrak{t})}^{\ast}m_{\bm{\lambda}}\equiv \langle \overline{m}_{\bm{\lambda}}g_{d(\mathfrak{s})}, \overline{m}_{\bm{\lambda}}g_{d(\mathfrak{t})}\rangle m_{\bm{\lambda}}\quad \mathrm{mod}~\text{Y}_{r,n}^{\rhd \bm{\lambda}}.$$

Let $\text{rad}\hspace{0.5mm}S^{\bm{\lambda}}=\{u\in S^{\bm{\lambda}}\:|\: \langle u, v\rangle=0~\mathrm{for~all}~v\in S^{\bm{\lambda}}\}.$ Consequently, $\text{rad}\hspace{0.5mm}S^{\bm{\lambda}}$ is a $\text{Y}_{r,n}$-submodule of $S^{\bm{\lambda}}.$ Let $D^{\bm{\lambda}}=S^{\bm{\lambda}}/\text{rad}\hspace{0.5mm}S^{\bm{\lambda}}$ for each $\bm{\lambda}\in \mathcal{P}_{r,n}.$ By a general theory of cellular basis, if $\mathcal{R}=\mathbb{K}$ is an algebraically closed field of characteristic $p\geq 0$ such that $p$ does not divide $r,$ the set $\{D^{\bm{\lambda}}\neq 0\:|\: \bm{\lambda}\in \mathcal{P}_{r,n}\}$ gives a complete set of non-isomorphic irreducible $\text{Y}_{r,n}$-modules. In fact, by [ER, Theorem 7 and (46)] (see also [JaPA, $\S$4.1] and [CW, Theorem 6.3]), $\{\bm{\lambda}\in \mathcal{P}_{r,n}\:|\: D^{\bm{\lambda}}\neq 0\}$ is just the set $\mathcal{K}_{r,n},$ where $$\mathcal{K}_{r,n}=\big\{\bm{\lambda}\in \mathcal{P}_{r,n}\:|\:\bm{\lambda}=(\lambda^{(1)},\ldots,\lambda^{(r)})~\mathrm{with ~each}~ \lambda^{(i)} ~\mathrm{being~ an~}e\mathrm{-restricted~ partition}\big\}.$$

\section{Yokonuma-Schur algebra and its cellular basis}

For an $r$-composition $\bm{\lambda}$ of $n,$ a $\bm{\lambda}$-tableau $\mathrm{S}=(S^{(1)},\ldots, S^{(r)})$ is a map $\mathrm{S}: [\bm{\lambda}]\rightarrow \{1,\ldots,n\}\times \{1,\ldots,r\},$ which can be regarded as the diagram $[\bm{\lambda}],$ together with an ordered pair $(i,k)$ ($1\leq i\leq n, 1\leq k\leq r$) attached to each node. Given $\bm{\lambda}\in \mathcal{P}_{r,n}$ and $\bm{\mu}\in \mathcal{C}_{r,n}$, a $\bm{\lambda}$-tableau $\text{S}$ is said to be of type $\bm{\mu}$ if the number of $(i,k)$ in the entry of $\text{S}$ is equal to $\mu_{i}^{(k)}.$ Given $\mathfrak{s}\in \mathrm{Std}(\bm{\lambda})$, $\bm{\mu}(\mathfrak{s})$, a $\bm{\lambda}$-tableau of type $\bm{\mu},$ is defined by replacing each entry $m$ in $\mathfrak{s}$ by $(i,k)$ if $m$ is in the $i$-th row of the $k$-th component of $\mathfrak{t}^{\bm{\mu}}.$

We define a total order on the set of pairs $(i,k)$ by $(i_{1}, k_{1})< (i_{2}, k_{2})$ if $k_{1}< k_{2},$ or $k_{1}=k_{2}$ and $i_{1}<i_{2}.$ Let $\bm{\lambda}\in \mathcal{P}_{r,n}$ and $\bm{\mu}\in \mathcal{C}_{r,n}$. Suppose that $\text{S}=(S^{(1)},\ldots, S^{(r)})$ is a $\bm{\lambda}$-tableau of type $\bm{\mu}.$ $\text{S}$ is said to be semistandard if each component $S^{(k)}$ is non-decreasing in rows, strictly increasing in columns, and all entries of $S^{(k)}$ are of the form $(i,l)$ with $l\geq k.$ We denote by $\mathcal{T}_{0}(\bm{\lambda}, \bm{\mu})$ the set of semistandard $\bm{\lambda}$-tableaux of type $\bm{\mu}.$

Let us consider the special case when $r=1.$ Suppose that $\bm{\lambda}, \bm{\mu}\in \mathcal{C}_{1,n}.$ A $\bm{\lambda}$-tableau $\text{S}$ of type $\bm{\mu}$ is said to be row semistandard if the entries in each row of $\text{S}$ are non-decreasing. $\text{S}$ is said to be semistandard if $\bm{\lambda}\in \mathcal{P}_{1,n},$ $\text{S}$ is row semistandard and the entries in each column are strictly increasing. Assume that $\bm{\lambda}, \bm{\mu}\in \mathcal{C}_{1,n}$ and put $\mathcal{D}_{\bm{\lambda}\bm{\mu}}=\mathcal{D}_{\bm{\lambda}}\cap\mathcal{D}_{\bm{\mu}}^{-1}.$ Then $\mathcal{D}_{\bm{\lambda}\bm{\mu}}$ is the set of minimal length elements in the double cosets $\mathfrak{S}_{\bm{\lambda}}\backslash\mathfrak{S}_{n}/\mathfrak{S}_{\bm{\mu}}$, and the map $d\mapsto \bm{\mu}(\mathfrak{t}^{\bm{\lambda}}d)$ gives a bijection between the set $\mathcal{D}_{\bm{\lambda}\bm{\mu}}$ and the set of row semistandard $\bm{\lambda}$-tableaux of type $\bm{\mu}.$

Let $d\in \mathcal{D}_{\bm{\lambda}\bm{\mu}},$ and put $\text{S}=\bm{\mu}(\mathfrak{t}^{\bm{\lambda}}d),$ $\text{T}=\bm{\lambda}(\mathfrak{t}^{\bm{\mu}}d^{-1}).$ Then $\text{S}$ and $\text{T}$ are both row semistandard, and we have
\begin{equation}\label{Yoko-Schur-3-1}
\sum_{\substack{y\in \mathcal{D}_{\bm{\mu}}\\\bm{\lambda}(\mathfrak{t}^{\bm{\mu}}y)=\text{T}}}q^{l(y)}g_{y}^{\ast}x_{\bm{\mu}}=\sum_{w\in \mathfrak{S}_{\bm{\lambda}}d\mathfrak{S}_{\bm{\mu}}}q^{l(w)}g_{w}=\sum_{\substack{x\in \mathcal{D}_{\bm{\lambda}}\\\bm{\mu}(\mathfrak{t}^{\bm{\lambda}}x)=\text{S}}}q^{l(x)}x_{\bm{\lambda}}g_{x}.
\end{equation}

For any $\bm{\kappa}\in \mathcal{C}_{r,n},$ we define its type $\alpha(\bm{\kappa})$ by $\alpha(\bm{\kappa})=(n_{1},\ldots, n_{r})$ with $n_{i}=|\kappa^{(i)}|.$ Assume that $\bm{\lambda}\in \mathcal{P}_{r,n}$ and $\bm{\mu}\in \mathcal{C}_{r,n}.$ We define a subset $\mathcal{T}_{0}^{+}(\bm{\lambda}, \bm{\mu})$ of $\mathcal{T}_{0}(\bm{\lambda}, \bm{\mu})$ by $$\mathcal{T}_{0}^{+}(\bm{\lambda}, \bm{\mu})=\{\text{S}\in \mathcal{T}_{0}(\bm{\lambda}, \bm{\mu})\:|\:\alpha(\bm{\lambda})=\alpha(\bm{\mu})\}.$$

Take $\text{S}\in \mathcal{T}_{0}(\bm{\lambda}, \mu)$. One can check that $\mathrm{S}\in \mathcal{T}_{0}^{+}(\bm{\lambda}, \bm{\mu})$ if and only if each entry of $S^{(k)}$ is of the form $(i,k)$ for some $i.$ Moreover, if $\mathfrak{s}\in \mathrm{Std}(\bm{\lambda})$ is such that $\bm{\mu}(\mathfrak{s})=\text{S}$ with $\mathrm{S}\in \mathcal{T}_{0}^{+}(\bm{\lambda}, \bm{\mu}),$ then the entries of the $i$-th component of $\mathfrak{s}$ consist of numbers $a_{i}+1,\ldots, a_{i+1},$ where $a_{i}=\sum_{k=1}^{i-1}n_{k}.$ In particular, $d(\mathfrak{s})\in \mathfrak{S}_{\alpha}$ for $\alpha=\alpha(\bm{\lambda}).$

Take $\text{S}\in \mathcal{T}_{0}^{+}(\bm{\lambda}, \bm{\mu})$. Let $\mathfrak{s}_{1}=$ first$(\mathrm{S}),$ which is the unique element of Std$(\bm{\lambda})$ satisfying the property that $\bm{\mu}(\mathfrak{s}_{1})=\text{S}$ and that $\mathfrak{s}_{1}\unrhd \mathfrak{s}$ for any $\mathfrak{s}\in \mathrm{Std}(\bm{\lambda})$ such that $\bm{\mu}(\mathfrak{s})=\mathrm{S}.$ Let $\alpha=\alpha(\bm{\lambda})=\alpha(\bm{\mu})$. Then $d=d(\mathfrak{s}_{1})\in \mathfrak{S}_{\alpha},$ which is given as $d=(d_{1},\ldots,d_{r})$ with $d_{k}$ a distinguished double coset representative in $\mathfrak{S}_{\lambda^{(k)}}\backslash\mathfrak{S}_{n_{k}}/\mathfrak{S}_{\mu^{(k)}}.$ From \eqref{Yoko-Schur-3-1} we have
\begin{equation}\label{Yoko-Schur-3-2}
\sum_{\substack{\mathfrak{s}\in \mathrm{Std}(\bm{\lambda})\\\bm{\mu}(\mathfrak{s})=\text{S}}}q^{l(d(\mathfrak{s}))}x_{\bm{\lambda}}g_{d(\mathfrak{s})}=\sum_{w\in \mathfrak{S}_{\bm{\lambda}}d\mathfrak{S}_{\bm{\mu}}}q^{l(w)}g_{w}=hg_{d}x_{\bm{\mu}},\end{equation}
where $h=\sum g_{v},$ the sum running over certain elements $v\in \mathfrak{S}_{\bm{\lambda}}.$

For each $\bm{\mu}\in \mathcal{C}_{r,n},$ let $M^{\bm{\mu}}=m_{\bm{\mu}}\text{Y}_{r,n}.$ The following lemma gives a basis of $M^{\bm{\mu}}$ as an $\mathcal{R}$-module.
\begin{lemma}\label{Yoko-Schur-lemma-3-1}
For each $\bm{\mu}\in \mathcal{C}_{r,n},$ $\{m_{\bm{\mu}}g_{d}\:|\:d\in \mathcal{D}_{\bm{\mu}}\}$ is an $\mathcal{R}$-basis of $M^{\bm{\mu}}$.
\end{lemma}
\begin{proof}
Since $m_{\bm{\mu}}g_{d}=\sum_{w\in \mathfrak{S}_{\bm{\mu}}}U_{\bm{\mu}}g_{wd}$ for each $d\in \mathcal{D}_{\bm{\mu}},$ and hence $\{m_{\bm{\mu}}g_{d}\}$ is linearly independent. Since
$m_{\bm{\mu}}t_{i}=\zeta_{\text{p}_{\bm{\mu}}(i)}m_{\bm{\mu}}$ by [ER, Lemma 11(4)] and $m_{\bm{\mu}}g_{w}=q^{l(w)}m_{\bm{\mu}}$ for $w\in \mathfrak{S}_{\bm{\mu}},$ then the set $\{m_{\bm{\mu}}g_{d}\:|\:d\in \mathcal{D}_{\bm{\mu}}\}$ spans $M^{\bm{\mu}}$ by \eqref{basis-theorem-2}. Thus, $\{m_{\bm{\mu}}g_{d}\:|\:d\in \mathcal{D}_{\bm{\mu}}\}$, or equivalently, $\{m_{\bm{\mu}}g_{d(\mathfrak{t})}\:|\:\mathfrak{t}\in \mathrm{r}\mathrm{-Std}(\bm{\mu})\}$ is an $\mathcal{R}$-basis of $M^{\bm{\mu}}$.
\end{proof}

We now construct a basis of $M^{\bm{\mu}}$ related to the cellular basis $\{m_{\mathfrak{s}\mathfrak{t}}\}.$ For $\text{S}\in \mathcal{T}_{0}^{+}(\bm{\lambda}, \bm{\mu})$ and $\mathfrak{t}\in \mathrm{Std}(\bm{\lambda}),$ we define $$m_{\text{S}\mathfrak{t}}=\sum_{\substack{\mathfrak{s}\in \mathrm{Std}(\bm{\lambda})\\\bm{\mu}(\mathfrak{s})=\text{S}}}q^{l(d(\mathfrak{s}))+l(d(\mathfrak{t}))}m_{\mathfrak{s}\mathfrak{t}}.$$
We have
\begin{lemma}\label{Yoko-Schur-lemma-3-2}
Let $\emph{S}\in \mathcal{T}_{0}^{+}(\bm{\lambda}, \bm{\mu})$ and $\mathfrak{t}\in \mathrm{Std}(\bm{\lambda}).$ Then $m_{\emph{S}\mathfrak{t}}\in M^{\bm{\mu}}.$
\end{lemma}
\begin{proof}
By \eqref{Yoko-Schur-3-2} we have
\begin{align*}
m_{\text{S}\mathfrak{t}}&=\sum_{\substack{\mathfrak{s}\in \mathrm{Std}(\bm{\lambda})\\\bm{\mu}(\mathfrak{s})=\text{S}}}q^{l(d(\mathfrak{s}))+l(d(\mathfrak{t}))}g_{d(\mathfrak{s})}^{\ast}x_{\bm{\lambda}}U_{\bm{\lambda}}g_{d(\mathfrak{t})}\\
&=q^{l(d(\mathfrak{t}))}x_{\bm{\mu}}g_{d}^{\ast}h^{\ast}U_{\bm{\lambda}}g_{d(\mathfrak{t})}.
\end{align*}
Since $h=\sum g_{v},$ where $v\in \mathfrak{S}_{\bm{\lambda}},$ hence $U_{\bm{\lambda}}$ commutes with $h^{\ast}$ by [ER, Lemma 11(3)]. Since $d\in \mathfrak{S}_{\alpha}$ with $\alpha=\alpha(\bm{\lambda}),$ $U_{\bm{\lambda}}$ commutes with $g_{d}^{\ast}$ by the same reason. Noting that $\alpha=\alpha(\bm{\lambda})=\alpha(\bm{\mu}),$ we have $U_{\bm{\lambda}}=U_{\bm{\mu}}.$ Thus, we see that $$x_{\bm{\mu}}g_{d}^{\ast}h^{\ast}U_{\bm{\lambda}}=x_{\bm{\mu}}U_{\bm{\mu}}g_{d}^{\ast}h^{\ast}\in m_{\bm{\mu}}\text{Y}_{r,n}=M^{\bm{\mu}},$$
and $m_{\text{S}\mathfrak{t}}\in M^{\bm{\mu}}$ as required.
\end{proof}

\begin{proposition}\label{Yoko-Schur-propo-3-3}
For each $\bm{\mu}\in \mathcal{C}_{r,n},$ $M^{\bm{\mu}}$ is free with an $\mathcal{R}$-basis
$$\big\{m_{\mathrm{S}\mathfrak{t}}\:|\:\mathrm{S}\in \mathcal{T}_{0}^{+}(\bm{\lambda}, \bm{\mu})~and~\mathfrak{t}\in \mathrm{Std}(\bm{\lambda})~for~some~\bm{\lambda}\in \mathcal{P}_{r,n}\big\}.$$
\end{proposition}
\begin{proof}
The basis elements $m_{\mathfrak{s}\mathfrak{t}}$ contained in the expression of $m_{\mathrm{S}\mathfrak{t}}$ are disjoint for different $m_{\mathrm{S}\mathfrak{t}}.$ It follows that $m_{\mathrm{S}\mathfrak{t}}$ are linearly independent. By Lemma \ref{Yoko-Schur-lemma-3-1}, $M^{\bm{\mu}}$ is a free $\mathcal{R}$-module, and its rank is equal to the number of $\{m_{\mathrm{S}\mathfrak{t}}\}$ given in the proposition by [SS, Lemma 2.5(ii) and Corollary 4.5(ii)]. Hence, any element in $M^{\bm{\mu}}$ can be written as a linear combination of various $m_{\mathrm{S}\mathfrak{t}}$ with coefficients in the quotient field of $\mathcal{R}$. But since the set $\{q^{l(d(\mathfrak{s}))+l(d(\mathfrak{t}))}m_{\mathfrak{s}\mathfrak{t}}\}$ is an $\mathcal{R}$-basis of $\text{Y}_{r,n},$ these coefficients are actually in $\mathcal{R}.$ This proves the proposition.
\end{proof}

Let $\bm{\mu}, \bm{\nu}\in \mathcal{C}_{r,n}$ be such that $\alpha(\bm{\mu})=\alpha(\bm{\nu})=\alpha.$ Put $\alpha=(n_{1},\ldots,n_{r}).$ Let us take $d\in \mathcal{D}_{\bm{\mu}\bm{\nu}}\cap \mathfrak{S}_{\alpha}.$ We have $d=(d_{1},\ldots,d_{r})$ with $d_{k}\in \mathcal{D}_{\mu^{(k)}\nu^{(k)}}$ with respect to $\mathfrak{S}_{n_{k}}.$ Then we can define a map $\varphi_{\bm{\mu}\bm{\nu}}^{d}: M^{\bm{\nu}}\rightarrow M^{\bm{\mu}}$ by $$\varphi_{\bm{\mu}\bm{\nu}}^{d}(m_{\bm{\nu}}h)=\sum_{w\in \mathfrak{S}_{\bm{\mu}}d\mathfrak{S}_{\bm{\nu}}}q^{l(w)}U_{\bm{\mu}}g_{w}h$$ for all $h\in \text{Y}_{r,n}.$ In fact, by \eqref{Yoko-Schur-3-1}, we have
\begin{equation}\label{Yoko-Schur-3-3}
\sum_{\substack{y\in \mathcal{D}_{\bm{\nu}}\cap \mathfrak{S}_{\alpha}\\\bm{\mu}(\mathfrak{t}^{\bm{\nu}}y)=\mathrm{T}}}q^{l(y)}U_{\bm{\mu}}g_{y}^{\ast}x_{\bm{\nu}}=\sum_{w\in \mathfrak{S}_{\bm{\mu}}d\mathfrak{S}_{\bm{\nu}}}q^{l(w)}U_{\bm{\mu}}g_{w}=\sum_{\substack{x\in \mathcal{D}_{\bm{\mu}}\cap \mathfrak{S}_{\alpha}\\\bm{\nu}(\mathfrak{t}^{\bm{\mu}}x)=\mathrm{S}}}q^{l(x)}m_{\bm{\mu}}g_{x},
\end{equation}
where $\mathrm{S}=\bm{\mu}(\mathfrak{t}^{\bm{\nu}}d)$ and $\mathrm{T}=\bm{\nu}(\mathfrak{t}^{\bm{\mu}}d^{-1})$ are row semistandard tableaux. Noting that $U_{\bm{\mu}}=U_{\bm{\nu}}$ and $y\in \mathfrak{S}_{\alpha},$ we have $U_{\bm{\mu}}g_{y}^{\ast}x_{\bm{\nu}}=g_{y}^{\ast}U_{\bm{\nu}}x_{\bm{\nu}}=g_{y}^{\ast}m_{\bm{\nu}},$ and $\varphi_{\bm{\mu}\bm{\nu}}^{d}$ is well-defined.

The proof of the next proposition is inspired by that of [SS, Proposition 5.2], although it should be noted that the basic set-up there is different from ours. It allows us to restrict ourselves to considering the subset $\mathcal{T}_{0}^{+}(\bm{\lambda}, \bm{\mu}).$

\begin{proposition}\label{Yoko-Schur-propo-3-4}
Let $\bm{\mu}, \bm{\nu}\in \mathcal{C}_{r,n}$. Then

$\mathrm{(i)}$ Assume that $\alpha(\bm{\mu})\neq\alpha(\bm{\nu}).$ Then $\mathrm{Hom}_{\emph{Y}_{r,n}}(M^{\bm{\nu}}, M^{\bm{\mu}})=0.$

$\mathrm{(ii)}$ Assume that $\alpha(\bm{\mu})=\alpha(\bm{\nu}).$ Then $\mathrm{Hom}_{\emph{Y}_{r,n}}(M^{\bm{\nu}}, M^{\bm{\mu}})$ is a free $\mathcal{R}$-module with basis $\{\varphi_{\bm{\mu}\bm{\nu}}^{d}\:|\:d\in \mathcal{D}_{\bm{\mu}\bm{\nu}}\cap \mathfrak{S}_{\alpha}\}.$
\end{proposition}
\begin{proof}
Suppose that $\varphi\in \mathrm{Hom}_{\text{Y}_{r,n}}(M^{\bm{\nu}}, M^{\bm{\mu}}).$ Then, for all $h\in \text{Y}_{r,n},$ we have $\varphi(m_{\bm{\nu}}h)=\varphi(m_{\bm{\nu}})h;$ hence $\varphi$ is completely determined by $\varphi(m_{\bm{\nu}}).$ Since $\varphi(m_{\bm{\nu}})\in M^{\bm{\mu}},$ by Lemma \ref{Yoko-Schur-lemma-3-1}, there exist some $c_{x}\in \mathcal{R}$ such that $\varphi(m_{\bm{\nu}})=\sum_{x\in \mathcal{D}_{\bm{\mu}}}c_{x}m_{\bm{\mu}}g_{x}$. By [ER, Lemma 10(49)], for each $k$, we have
\begin{equation}\label{Yoko-Schur-3-4}
\varphi(m_{\bm{\nu}}t_{k})=\varphi(\zeta_{\text{p}_{\bm{\nu}}(k)}m_{\bm{\nu}})=\sum_{x\in \mathcal{D}_{\bm{\mu}}}\zeta_{\text{p}_{\bm{\nu}}(k)}c_{x}m_{\bm{\mu}}g_{x}.
\end{equation}
Now assume that $c_{y}\neq 0$ for some $y\in \mathcal{D}_{\bm{\mu}},$ which is equal to some $d(\mathfrak{s})$ for some row standard $\bm{\mu}$-tableau $\mathfrak{s}.$ Then we have $$(c_{y}m_{\bm{\mu}}g_{y})t_{k}=c_{y}m_{\bm{\mu}}t_{kd(\mathfrak{s})^{-1}}g_{y}=\zeta_{\text{p}_{\bm{\mu}}(kd(\mathfrak{s})^{-1})}c_{y}m_{\bm{\mu}}g_{y}.$$ Since $\mathfrak{s}=\mathfrak{t}^{\bm{\mu}}d(\mathfrak{s}),$ we have that $\text{p}_{\bm{\mu}}(kd(\mathfrak{s})^{-1})=\text{p}_{\mathfrak{s}}(k),$ and hence
\begin{equation}\label{Yoko-Schur-3-5}
(c_{y}m_{\bm{\mu}}g_{y})t_{k}=\zeta_{\text{p}_{\mathfrak{s}}(k)}c_{y}m_{\bm{\mu}}g_{y}.
\end{equation}
By comparing \eqref{Yoko-Schur-3-4} and \eqref{Yoko-Schur-3-5}, we have $\text{p}_{\mathfrak{t}^{\bm{\nu}}}(k)=\text{p}_{\mathfrak{s}}(k)$ for all $k=1,\ldots, n.$ This implies that $\alpha(\bm{\mu})=\alpha(\bm{\nu}).$ Thus, (i) is proved.

Now assume that $\alpha(\bm{\mu})=\alpha(\bm{\nu})=\alpha.$ Since $\text{p}_{\mathfrak{t}^{\bm{\nu}}}(k)=\text{p}_{\mathfrak{t}^{\bm{\mu}}}(kd(\mathfrak{s})^{-1})$ for all $k=1,\ldots, n,$ we must have $y=d(\mathfrak{s})\in \mathfrak{S}_{\alpha}.$ Let $d$ be the unique minimal length element in $\mathfrak{S}_{\bm{\mu}}y\mathfrak{S}_{\bm{\nu}}.$ Then $d\in \mathcal{D}_{\bm{\mu}\bm{\nu}}\cap \mathfrak{S}_{\alpha},$ and a similar argument as in the proof of [Ma1, Theorem 4.8] implies that $c_{d}\neq 0.$ Set $\varphi'=\varphi-c_{d}\varphi_{\bm{\mu}\bm{\nu}}^{d}.$ Then $\varphi'\in \mathrm{Hom}_{\text{Y}_{r,n}}(M^{\bm{\nu}}, M^{\bm{\mu}}),$ and $\varphi'(m_{\bm{\nu}})$ can be written as $\varphi'(m_{\bm{\nu}})=\sum_{x\in \mathcal{D}_{\bm{\mu}}}a_{x}m_{\bm{\mu}}g_{x},$ where $a_{x}=c_{x}$ if $\mathfrak{S}_{\bm{\mu}}x\mathfrak{S}_{\bm{\nu}}\neq\mathfrak{S}_{\bm{\mu}}d\mathfrak{S}_{\bm{\nu}},$ and $a_{x}=0$ for $x\in \mathfrak{S}_{\bm{\mu}}d\mathfrak{S}_{\bm{\nu}}$ by the argument as in the proof of [Ma1, Theorem 4.8]. Hence, by induction we can write $\varphi$ as a linear combination of $\varphi_{\bm{\mu}\bm{\nu}}^{d}$ with $d\in \mathcal{D}_{\bm{\mu}\bm{\nu}}\cap \mathfrak{S}_{\alpha}$ as required.

Finally, we have to show that $\{\varphi_{\bm{\mu}\bm{\nu}}^{d}\:|\:d\in \mathcal{D}_{\bm{\mu}\bm{\nu}}\cap \mathfrak{S}_{\alpha}\}$ is linearly independent. This follows from the fact that $\varphi_{\bm{\mu}\bm{\nu}}^{d}(m_{\bm{\nu}})$ is a linearly independent subset of $M^{\bm{\mu}},$ since the set $\{U_{\bm{\mu}}g_{w}\}$ is linearly independent by the basis theorem for $\text{Y}_{r,n}.$
\end{proof}

We write $M^{\bm{\nu}\ast}=(M^{\bm{\nu}})^{\ast}=\text{Y}_{r,n}m_{\bm{\nu}}.$ As a corollary to Proposition \ref{Yoko-Schur-propo-3-4}, we have the next result.
\begin{corollary}\label{Yoko-Schur-corolla-3-5}
Let $\bm{\mu}, \bm{\nu}\in \mathcal{C}_{r,n}$. Then $\mathrm{Hom}_{\emph{Y}_{r,n}}(M^{\bm{\nu}}, M^{\bm{\mu}})$ and $M^{\bm{\nu}\ast}\cap M^{\bm{\mu}}$ are canonically isomorphic as $\mathcal{R}$-modules.
\end{corollary}
\begin{proof}
Every homomorphism $\varphi$ in $\mathrm{Hom}_{\text{Y}_{r,n}}(M^{\bm{\nu}}, M^{\bm{\mu}})$ is determined by $\varphi(m_{\bm{\nu}}),$ and moreover, $\varphi(m_{\bm{\nu}})\in M^{\bm{\nu}\ast}\cap M^{\bm{\mu}}$ by Proposition \ref{Yoko-Schur-propo-3-4}. As a result, the map $\mathrm{Hom}_{\text{Y}_{r,n}}(M^{\bm{\nu}}, M^{\bm{\mu}})$\\$\rightarrow M^{\bm{\nu}\ast}\cap M^{\bm{\mu}}$ given by $\varphi\mapsto \varphi(m_{\bm{\nu}})$ is an isomorphism of $\mathcal{R}$-modules.
\end{proof}

\begin{remark} It is shown in $[$CR, 61.2$]$ that whenever $A$ is a quasi-hereditary algebra, $a\in A$ and $J$ is an ideal of $A$ then $\mathrm{Hom}_{A}(aA, J)\cong Aa\cap J.$ By $\mathrm{[ChPA1, Proposition \:10]}$ $($see also $\mathrm{[C1, Corollary \:4.5]}$$),$ $\text{Y}_{r,n}$ is quasi-Frobenius, so this gives another proof of $\mathrm{Corollary \:3.5}$.
\end{remark}

Let $\bm{\mu}, \bm{\nu}\in \mathcal{C}_{r,n}$ and $\bm{\lambda}\in \mathcal{P}_{r,n}.$ We assume that $\alpha(\bm{\mu})=\alpha(\bm{\nu})=\alpha(\bm{\lambda}).$ For $\mathrm{S}\in \mathcal{T}_{0}^{+}(\bm{\lambda}, \bm{\mu})$, $\mathrm{T}\in \mathcal{T}_{0}^{+}(\bm{\lambda}, \bm{\nu}),$ put $$m_{\mathrm{S}\mathrm{T}}=\sum_{\mathfrak{s}, \mathfrak{t}}q^{l(d(\mathfrak{s}))+l(d(\mathfrak{t}))}m_{\mathfrak{s}\mathfrak{t}},$$ where the sum is taken over all $\mathfrak{s}, \mathfrak{t}\in \mathrm{Std}(\bm{\lambda})$ such that $\bm{\mu}(\mathfrak{s})=\mathrm{S}$ and $\bm{\nu}(\mathfrak{t})=\mathrm{T}.$

\begin{proposition}\label{Yoko-Schur-propo-3-7}
Suppose that $\bm{\mu}, \bm{\nu}\in \mathcal{C}_{r,n}$ with $\alpha(\bm{\mu})=\alpha(\bm{\nu})$. Then the set $$\{m_{\mathrm{S}\mathrm{T}}\:|\:\mathrm{S}\in \mathcal{T}_{0}^{+}(\bm{\lambda}, \bm{\mu})~and~\mathrm{T}\in \mathcal{T}_{0}^{+}(\bm{\lambda}, \bm{\nu})~for~some~\bm{\lambda}\in \mathcal{P}_{r,n}\}$$ is an $\mathcal{R}$-basis of $M^{\bm{\nu}\ast}\cap M^{\bm{\mu}}.$
\end{proposition}
\begin{proof}
Since $$m_{\mathrm{S}\mathrm{T}}=\sum_{\substack{\mathfrak{s}\in \mathrm{Std}(\bm{\lambda})\\\bm{\mu}(\mathfrak{s})=\mathrm{S}}}m_{\mathrm{T}\mathfrak{s}}^{\ast}=\sum_{\substack{\mathfrak{t}\in \mathrm{Std}(\bm{\lambda})\\\bm{\nu}(\mathfrak{t})=\mathrm{T}}}m_{\mathrm{S}\mathfrak{t}},$$ we see that $m_{\mathrm{S}\mathrm{T}}\in M^{\bm{\nu}\ast}\cap M^{\bm{\mu}}$ by Lemma \ref{Yoko-Schur-lemma-3-2}. Moreover, the elements $m_{\mathrm{S}\mathrm{T}}$ are linearly independent since the basis elements $m_{\mathfrak{s}\mathfrak{t}}$ involved in the $m_{ST}$ are distinct. Now suppose that $h\in M^{\bm{\nu}\ast}\cap M^{\bm{\mu}}.$ Since $h\in \text{Y}_{r,n},$ we may express $h$ with respect to the standard basis, that is, we may write $h=\sum r_{\mathfrak{s}\mathfrak{t}}m_{\mathfrak{s}\mathfrak{t}}$ for some $r_{\mathfrak{s}\mathfrak{t}}\in \mathcal{R}$. Since $h\in M^{\bm{\mu}},$ by Proposition \ref{Yoko-Schur-propo-3-3} if $r_{\mathfrak{s}\mathfrak{t}}\neq 0$ then $\bm{\mu}(\mathfrak{s})\in \mathcal{T}_{0}^{+}(\bm{\lambda}, \bm{\mu})$ for some $\bm{\lambda}\in \mathcal{P}_{r,n}$ and $r_{\mathfrak{s}\mathfrak{t}}=r_{\mathfrak{s}'\mathfrak{t}}$ whenever $\bm{\mu}(\mathfrak{s})=\bm{\mu}(\mathfrak{s}')$. Similarly, since $h\in M^{\bm{\nu}\ast},$ if $r_{\mathfrak{s}\mathfrak{t}}\neq 0$ then $\bm{\nu}(\mathfrak{t})\in \mathcal{T}_{0}^{+}(\bm{\lambda}, \bm{\nu})$ for some $\bm{\lambda}\in \mathcal{P}_{r,n}$ and $r_{\mathfrak{s}\mathfrak{t}}=r_{\mathfrak{s}\mathfrak{t}'}$ whenever $\bm{\nu}(\mathfrak{t})=\bm{\nu}(\mathfrak{t}')$. Consequently, if $\bm{\mu}(\mathfrak{s})=\bm{\mu}(\mathfrak{s}')\in \mathcal{T}_{0}^{+}(\bm{\lambda}, \bm{\mu})$ and $\bm{\nu}(\mathfrak{t})=\bm{\nu}(\mathfrak{t}')\in \mathcal{T}_{0}^{+}(\bm{\lambda}, \bm{\nu}),$ then $r_{\mathfrak{s}\mathfrak{t}}=r_{\mathfrak{s}'\mathfrak{t}}=r_{\mathfrak{s}'\mathfrak{t}'}=r_{\mathfrak{s}\mathfrak{t}'}.$ This proves the proposition. \end{proof}

\begin{definition}
Suppose that $M_{n}^{r}=\bigoplus_{\bm{\mu}\in \mathcal{C}_{r,n}}M^{\bm{\mu}}.$ We define the Yokonuma-Schur algebra YS$_{n}^{r}=\text{YS}_{q}(r,n)$ as the endomorphism algebra $$\mathrm{YS}_{n}^{r}=\mathrm{End}_{\text{Y}_{r,n}}(M_{n}^{r}),$$ which is isomorphic to $\bigoplus_{\bm{\mu}, \bm{\nu}\in \mathcal{C}_{r,n}}\mathrm{Hom}_{\text{Y}_{r,n}}(M^{\bm{\nu}}, M^{\bm{\mu}}).$
\end{definition}

Let $\mathrm{S}\in \mathcal{T}_{0}^{+}(\bm{\lambda}, \bm{\mu})$ and $\mathrm{T}\in \mathcal{T}_{0}^{+}(\bm{\lambda}, \bm{\nu}).$ In view of Proposition \ref{Yoko-Schur-propo-3-7}, we can define $\varphi_{\mathrm{S}\mathrm{T}}\in \mathrm{Hom}_{\text{Y}_{r,n}}(M^{\bm{\nu}}, M^{\bm{\mu}})$ by
\begin{equation}\label{yokonuma-schur-3-6}
\varphi_{\mathrm{S}\mathrm{T}}(m_{\bm{\nu}}h)=m_{\mathrm{S}\mathrm{T}}h
\end{equation}
for all $h\in \text{Y}_{r,n}.$ We extend $\varphi_{\mathrm{S}\mathrm{T}}$ to an element of YS$_{n}^{r}$ by defining $\varphi_{\mathrm{S}\mathrm{T}}$ to be zero on $M^{\bm{\kappa}}$ for $\bm{\nu}\neq \bm{\kappa}\in \mathcal{C}_{r,n}.$ For any $\bm{\lambda}\in \mathcal{P}_{r,n}$, let $\mathcal{T}_{0}^{+}(\bm{\lambda})=\bigcup_{\bm{\mu}\in \mathcal{C}_{r,n}}\mathcal{T}_{0}^{+}(\bm{\lambda}, \bm{\mu}).$ We denote by YS$_{r,n}^{\rhd \bm{\lambda}}$ the $\mathcal{R}$-submodule of YS$_{n}^{r}$ spanned by $\varphi_{\mathrm{S}\mathrm{T}}$ such that $\mathrm{S}, \mathrm{T}\in \mathcal{T}_{0}^{+}(\bm{\nu})$ with $\bm{\nu}\rhd \bm{\lambda}.$ Then we have the next theorem.

\begin{theorem}\label{yokonu-cellu-basis-3-9}
The Yokonuma-Schur algebra $\mathrm{YS}_{n}^{r}$ is free as an $\mathcal{R}$-module with a basis $$\big\{\varphi_{\mathrm{S}\mathrm{T}}\:|\:\mathrm{S}, \mathrm{T}\in \mathcal{T}_{0}^{+}(\bm{\lambda})~for~some~\bm{\lambda}\in \mathcal{P}_{r,n}\big\}.$$ Moreover, this basis satisfies the following properties$:$

$\mathrm{(i)}$ The $\mathcal{R}$-linear map $\ast: \mathrm{YS}_{n}^{r}\rightarrow \mathrm{YS}_{n}^{r}$ determined by $\varphi_{\mathrm{S}\mathrm{T}}^{\ast}=\varphi_{\mathrm{T}\mathrm{S}},$ for all $\mathrm{S}, \mathrm{T}\in \mathcal{T}_{0}^{+}(\bm{\lambda})$ and all $\bm{\lambda}\in \mathcal{P}_{r,n},$ is an anti-automorphism of $\mathrm{YS}_{n}^{r}.$

$\mathrm{(ii)}$ Let $\mathrm{T}\in \mathcal{T}_{0}^{+}(\bm{\lambda})$ and $\varphi\in \mathrm{YS}_{n}^{r}.$ Then for each $\mathrm{V}\in \mathcal{T}_{0}^{+}(\bm{\lambda})$, there exists $r_{\mathrm{V}}=r_{\mathrm{V}, \mathrm{T},\varphi}\in \mathcal{R}$ such that for all $\mathrm{S}\in \mathcal{T}_{0}^{+}(\bm{\lambda}),$ we have $$\varphi_{\mathrm{S}\mathrm{T}}\varphi\equiv\sum_{\mathrm{V}\in \mathcal{T}_{0}^{+}(\bm{\lambda})}r_{\mathrm{V}}\varphi_{\mathrm{S}\mathrm{V}}\quad \mathrm{mod}~\mathrm{YS}_{r,n}^{\rhd \bm{\lambda}}.$$
In particular, this basis $\{\varphi_{\mathrm{S}\mathrm{T}}\}$ is a cellular basis of $\mathrm{YS}_{n}^{r}.$
\end{theorem}
\begin{proof}
The proof is similar to that of [Ma1, Theorem 4.14] and [DJM, Theorem 6.6]. By Corollary \ref{Yoko-Schur-corolla-3-5} and Proposition \ref{Yoko-Schur-propo-3-7}, the set $\{\varphi_{\mathrm{S}\mathrm{T}}\}$ is an $\mathcal{R}$-basis of $\mathrm{YS}_{n}^{r}$. Next we need to verify (i) and (ii).

(i) Let $\bm{\lambda}\in \mathcal{P}_{r,n}$ and $\bm{\mu}, \bm{\nu}\in \mathcal{C}_{r,n},$ and take $\mathrm{S}\in \mathcal{T}_{0}^{+}(\bm{\lambda}, \bm{\mu}), \mathrm{T}\in \mathcal{T}_{0}^{+}(\bm{\lambda}, \bm{\nu}).$ Then $\varphi_{\mathrm{S}\mathrm{T}}^{\ast}(m_{\bm{\mu}})=m_{\mathrm{T}\mathrm{S}}=(m_{\mathrm{S}\mathrm{T}})^{\ast}=(\varphi_{\mathrm{S}\mathrm{T}}(m_{\bm{\nu}}))^{\ast}.$ By the $\mathcal{R}$-linearity, we have $\varphi^{\ast}(m_{\bm{\mu}})=(\varphi(m_{\bm{\nu}}))^{\ast}$ for any $\varphi\in \mathrm{Hom}_{\text{Y}_{r,n}}(M^{\bm{\nu}}, M^{\bm{\mu}}).$ Given $\varphi\in \mathrm{Hom}_{\text{Y}_{r,n}}(M^{\bm{\nu}}, M^{\bm{\mu}})$ and $\psi\in \mathrm{Hom}_{\text{Y}_{r,n}}(M^{\bm{\kappa}}, M^{\bm{\lambda}}),$ we may assume that $\bm{\mu}=\bm{\kappa}$ since otherwise $\psi\varphi=0.$ Write $\varphi(m_{\bm{\nu}})=x_{1}m_{\bm{\nu}}$ and $\psi(m_{\bm{\mu}})=x_{2}m_{\bm{\mu}}$ for some $x_{1}, x_{2}\in \text{Y}_{r,n}.$ We have
\begin{align*}
(\psi\varphi)^{\ast}(m_{\bm{\lambda}})&=(\psi\varphi(m_{\bm{\nu}}))^{\ast}=(x_{2}x_{1}m_{\bm{\nu}})^{\ast}=m_{\bm{\nu}}x_{1}^{\ast}x_{2}^{\ast}\\
&=\varphi^{\ast}(m_{\bm{\mu}})x_{2}^{\ast}=\varphi^{\ast}(m_{\bm{\mu}}x_{2}^{\ast})=\varphi^{\ast}\psi^{\ast}(m_{\bm{\lambda}}).
\end{align*}
Hence, $(\psi\varphi)^{\ast}=\varphi^{\ast}\psi^{\ast}$ and $\ast$ is an anti-automorphism.

(ii) Take $\mathrm{S}\in \mathcal{T}_{0}^{+}(\bm{\lambda}, \bm{\mu})$, $\mathrm{T}\in \mathcal{T}_{0}^{+}(\bm{\lambda}, \bm{\nu}).$ We may assume that $\varphi\in \mathrm{Hom}_{\text{Y}_{r,n}}(M^{\bm{\kappa}}, M^{\bm{\nu}})$ for some $\bm{\kappa}\in \mathcal{C}_{r,n}$ with $\alpha(\bm{\kappa})=\alpha(\bm{\nu}).$ We have $\varphi(m_{\bm{\kappa}})=m_{\bm{\nu}}h$ for some $h\in \text{Y}_{r,n}.$ Then $\varphi_{\mathrm{S}\mathrm{T}}\varphi(m_{\bm{\kappa}})=m_{\mathrm{S}\mathrm{T}}h.$ By Corollary \ref{Yoko-Schur-corolla-3-5}, we see that $m_{\mathrm{S}\mathrm{T}}h\in M^{\bm{\kappa}\ast}\cap M^{\bm{\mu}}.$ Hence by Proposition \ref{Yoko-Schur-propo-3-7}, we may write $m_{\mathrm{S}\mathrm{T}}h=\sum_{\mathrm{U},\mathrm{V}}r_{\mathrm{U}\mathrm{V}}m_{\mathrm{U}\mathrm{V}},$ where $r_{\mathrm{U}\mathrm{V}}\in \mathcal{R},$ and the sum is over $\mathrm{U}\in \mathcal{T}_{0}^{+}(\bm{\alpha}, \bm{\mu})$ and $\mathrm{V}\in \mathcal{T}_{0}^{+}(\bm{\alpha}, \bm{\kappa})$ for some $\bm{\alpha}\in \mathcal{P}_{r,n}.$ By applying Theorem \ref{cellula-bases-2}(ii), we can write $m_{\mathrm{S}\mathrm{T}}h$ as $$m_{\mathrm{S}\mathrm{T}}h=\sum_{\mathrm{V}\in \mathcal{T}_{0}^{+}(\bm{\lambda}, \bm{\kappa})}r_{\mathrm{V}}m_{\mathrm{S}\mathrm{V}}+\sum_{\substack{\bm{\alpha}\in \mathcal{P}_{r,n}\\\bm{\alpha}\rhd \bm{\lambda}}}\sum_{\substack{\mathrm{U}'\in \mathcal{T}_{0}^{+}(\bm{\alpha}, \bm{\mu})\\\mathrm{V}'\in \mathcal{T}_{0}^{+}(\bm{\alpha}, \bm{\kappa})}}r_{\mathrm{U}'\mathrm{V}'}m_{\mathrm{U}'\mathrm{V}'},$$ where $r_{\mathrm{V}}, r_{\mathrm{U}'\mathrm{V}'}\in \mathcal{R}.$ Therefore, we have
$$\varphi_{\mathrm{S}\mathrm{T}}\varphi\equiv\sum_{\mathrm{V}\in \mathcal{T}_{0}^{+}(\bm{\lambda}, \bm{\kappa})}r_{\mathrm{V}}\varphi_{\mathrm{S}\mathrm{V}}\quad \mathrm{mod}~\mathrm{YS}_{r,n}^{\rhd \bm{\lambda}}.$$ We are done.
\end{proof}

For each $\bm{\lambda}\in \mathcal{P}_{r,n},$ let $\mathrm{T}^{\bm{\lambda}}=\bm{\lambda}(\mathfrak{t}^{\bm{\lambda}}).$ Then $\mathrm{T}^{\bm{\lambda}}\in \mathcal{T}_{0}^{+}(\bm{\lambda}, \bm{\lambda})$ and $\mathrm{T}^{\bm{\lambda}}$ is the unique semistandard $\bm{\lambda}$-tableau of type $\bm{\lambda}$. Moreover $\mathfrak{t}=\mathfrak{t}^{\bm{\lambda}}$ is the unique element in Std$(\bm{\lambda})$ such that $\bm{\lambda}(\mathfrak{t})=\mathrm{T}^{\bm{\lambda}}.$ Thus, $m_{\mathrm{T}^{\bm{\lambda}}\mathrm{T}^{\bm{\lambda}}}=m_{\mathfrak{t}^{\bm{\lambda}}\mathfrak{t}^{\bm{\lambda}}}=m_{\bm{\lambda}},$ and $\varphi_{\bm{\lambda}}=\varphi_{\mathrm{T}^{\bm{\lambda}}\mathrm{T}^{\bm{\lambda}}}$ is the identity map on $M^{\bm{\lambda}}.$

The Weyl module $W^{\bm{\lambda}}$ is defined as the right $\mathrm{YS}_{n}^{r}$-submodule of $\mathrm{YS}_{n}^{r}/\mathrm{YS}_{r,n}^{\rhd \bm{\lambda}}$ spanned by the image of $\varphi_{\bm{\lambda}}.$ For each $\mathrm{S}\in \mathcal{T}_{0}^{+}(\bm{\lambda}, \bm{\mu}),$ we denote by $\varphi_{\mathrm{S}}$ the image of $\varphi_{\mathrm{T}^{\bm{\lambda}}\mathrm{S}}$ in $\mathrm{YS}_{n}^{r}/\mathrm{YS}_{r,n}^{\rhd \bm{\lambda}}.$ Then by Theorem \ref{yokonu-cellu-basis-3-9}, we see that $W^{\bm{\lambda}},$ as an $\mathcal{R}$-module, is free with basis $\{\varphi_{\mathrm{S}}\:|\:\mathrm{S}\in \mathcal{T}_{0}^{+}(\bm{\lambda})\}.$

The Weyl module $W^{\bm{\lambda}}$ possesses an associative symmetric bilinear form, which is completely determined by the equation $$\varphi_{\mathrm{T}^{\bm{\lambda}}\mathrm{S}}\varphi_{\mathrm{T}\mathrm{T}^{\bm{\lambda}}}\equiv \langle \varphi_{\mathrm{S}}, \varphi_{\mathrm{T}}\rangle \varphi_{\bm{\lambda}}\quad \mathrm{mod}~\mathrm{YS}_{r,n}^{\rhd \bm{\lambda}}$$ for all $\mathrm{S}, \mathrm{T}\in \mathcal{T}_{0}^{+}(\bm{\lambda}).$ Note that $\langle \varphi_{\mathrm{S}}, \varphi_{\mathrm{T}}\rangle=0$ unless $\mathrm{S}$ and $\mathrm{T}$ are semistandard tableaux of the same type. Let $L^{\bm{\lambda}}=W^{\bm{\lambda}}/\mathrm{rad}\hspace{0.3mm}W^{\bm{\lambda}},$ where $\mathrm{rad}\hspace{0.3mm}W^{\bm{\lambda}}=\{x\in W^{\bm{\lambda}}\:|\:\langle x, y\rangle=0\mathrm{~for~all~}y\in W^{\bm{\lambda}}\}.$

\begin{proposition}\label{yokonu-cellubasis-propo-3-10}
Suppose that $\mathcal{R}=\mathbb{K}$ is a field. Then for each $\bm{\lambda}\in \mathcal{P}_{r,n},$ $L^{\bm{\lambda}}$ is an absolutely irreducible $\mathrm{YS}_{n}^{r}$-module. Moreover, $\{L^{\bm{\lambda}}\:|\:\bm{\lambda}\in \mathcal{P}_{r,n}\}$ is a complete set of non-isomorphic irreducible $\mathrm{YS}_{n}^{r}$-modules.
\end{proposition}
\begin{proof}
For each $\bm{\lambda}\in \mathcal{P}_{r,n},$ we have $$\varphi_{\mathrm{T}^{\bm{\lambda}}\mathrm{T}^{\bm{\lambda}}}\varphi_{\mathrm{T}^{\bm{\lambda}}\mathrm{T}^{\bm{\lambda}}}\equiv \langle \varphi_{\mathrm{T}^{\bm{\lambda}}}, \varphi_{\mathrm{T}^{\bm{\lambda}}}\rangle \varphi_{\bm{\lambda}}~~~~\mathrm{mod}~\mathrm{YS}_{r,n}^{\rhd \bm{\lambda}}.$$ But since $\varphi_{\mathrm{T}^{\bm{\lambda}}\mathrm{T}^{\bm{\lambda}}}\varphi_{\mathrm{T}^{\bm{\lambda}}\mathrm{T}^{\bm{\lambda}}}=\varphi_{\bm{\lambda}}$ is the identity map on $M^{\bm{\lambda}},$ we see that $\langle \varphi_{\mathrm{T}^{\bm{\lambda}}}, \varphi_{\mathrm{T}^{\bm{\lambda}}}\rangle=1,$ and so $L^{\bm{\lambda}}$ is nonzero. Then the assertions follow from [GL, (3.4)].
\end{proof}

If $\bm{\lambda}, \bm{\mu}\in \mathcal{P}_{r,n},$ let $d_{\bm{\lambda}\bm{\mu}}$ denote the composition multiplicity of $L^{\bm{\mu}}$ as a composition factor of $W^{\bm{\lambda}}.$ Then $(d_{\bm{\lambda}\bm{\mu}})_{\bm{\lambda}, \bm{\mu}\in \mathcal{P}_{r,n}}$ is the decomposition matrix of $\mathrm{YS}_{n}^{r}$. The theory of cellular algebras [GL, (3.6)] yields the following result.
\begin{corollary}\label{yokonu-cellubasis-corollar-3-11}
Suppose that $\mathcal{R}=\mathbb{K}$ is a field. $(d_{\bm{\lambda}\bm{\mu}})_{\bm{\lambda}, \bm{\mu}\in \mathcal{P}_{r,n}}$ is unitriangular. That is, for $\bm{\lambda}, \bm{\mu}\in \mathcal{P}_{r,n},$ we have $d_{\bm{\mu}\bm{\mu}}=1$ and $d_{\bm{\lambda}\bm{\mu}}\neq 0$ only if $\bm{\lambda} \unrhd\bm{\mu}.$
\end{corollary}

Combining Proposition \ref{yokonu-cellubasis-propo-3-10} with [GL, (3.10)], we have the next result.
\begin{corollary}\label{yokonu-cellubasis-corollar-3-12}
Suppose that $\mathcal{R}=\mathbb{K}$ is a field. The Yokonuma-Schur algebra $\mathrm{YS}_{n}^{r}$ is quasi-hereditary.
\end{corollary}

\begin{remark} For each $\bm{\lambda} \in \mathcal{P}_{r,n}$ and for each $\mathfrak{t}\in$ Std$(\bm{\lambda}),$ let $m_{\mathfrak{t}}\in S^{\bm{\lambda}}$ be the image of $m_{\mathfrak{t}^{\bm{\lambda}}\mathfrak{t}}$ under the map $\text{Y}_{r,n}/\text{Y}_{r,n}^{\rhd \bm{\lambda}}.$ Then $\{m_{\mathfrak{t}}\}=\{\overline{m}_{\bm{\lambda}}g_{d(\mathfrak{t})}\}$ gives an $\mathcal{R}$-basis of $S^{\bm{\lambda}}.$ For $\mathrm{T}\in \mathcal{T}_{0}^{+}(\bm{\lambda}, \bm{\mu}),$ put $m_{\mathrm{T}}=\sum_{\mathfrak{t}}q^{l(d(\mathfrak{t}))+l(d(\mathfrak{t}^{\bm{\lambda}}))}m_{\mathfrak{t}}\in S^{\bm{\lambda}},$ where the sum is taken over all $\mathfrak{t}$ such that $\bm{\mu}(\mathfrak{t})=\mathrm{T}.$ Since $m_{\mathrm{T}}$ is the image of $m_{\mathrm{T}^{\bm{\lambda}}\mathrm{T}},$ one obtains a well-defined map $\varphi_{\mathrm{T}}\in \mathrm{Hom}_{\text{Y}_{r,n}}(M^{\bm{\mu}}, S^{\bm{\lambda}})$ by $\varphi_{\mathrm{T}}(m_{\bm{\mu}})=m_{\mathrm{T}},$ which is regarded as an element of $\mathrm{Hom}_{\text{Y}_{r,n}}(M_{n}^{r}, S^{\bm{\lambda}})$ by extending by 0 outside. In a similar way as in [Ma1, Proposition 4.15], we see that $W^{\bm{\lambda}}$ is isomorphic to the $\mathrm{YS}_{n}^{r}$-submodule of $\mathrm{Hom}_{\text{Y}_{r,n}}(M_{n}^{r}, S^{\bm{\lambda}})$ with basis $\{\varphi_{\mathrm{T}}\:|\:\mathrm{T}\in \mathcal{T}_{0}^{+}(\bm{\lambda})\}.$
\end{remark}

\section{Schur functors}

In this section, we will follow the approach in [HM2, $\S$4.3] to define an exact functor from the category of $\mathrm{YS}_{n}^{r}$-modules to the category of $\text{Y}_{r,n}$-modules. For an algebra $A,$ let $A\mathrm{-mod}$ be the category of finite dimensional right $A$-modules.

Let $\dot{\mathcal{C}}_{r,n}=\mathcal{C}_{r,n}\cup \{\omega\},$ where $\omega$ is a dummy symbol. Set $M^{\omega}=\text{Y}_{r,n}$ and $\dot{M}_{n}^{r}=M_{n}^{r}\oplus M^{\omega}.$ The extended Yokonuma-Schur algebra is the algebra $$\dot{\mathrm{YS}_{n}^{r}}=\mathrm{End}_{\text{Y}_{r,n}}(\dot{M}_{n}^{r}).$$
Suppose that $\bm{\lambda}\in \mathcal{P}_{r,n},$ and set $\mathcal{T}_{0}^{+}(\bm{\lambda}, \omega) :=\text{Std}(\bm{\lambda}).$ Let $m_{\omega}=1$ so that $M^{\omega}=m_{\omega}\text{Y}_{r,n}$. Let $\mathfrak{t}^{\omega}=1$ and $m_{\mathfrak{t}^{\omega}\mathfrak{t}^{\omega}}=1.$ We regard $\mathrm{YS}_{n}^{r}$ as a subalgebra of $\dot{\mathrm{YS}_{n}^{r}}$ in the obvious way.

Extending \eqref{yokonuma-schur-3-6}, if $\bm{\lambda}\in \mathcal{P}_{r,n},$ $\bm{\mu}, \bm{\nu}\in \dot{\mathcal{C}}_{r,n},$ and $\mathrm{S}\in \mathcal{T}_{0}^{+}(\bm{\lambda}, \bm{\mu}), \mathrm{T}\in \mathcal{T}_{0}^{+}(\bm{\lambda}, \bm{\nu}),$ we define $$\varphi_{\mathrm{S}\mathrm{T}}(m_{\bm{\nu}}h)=m_{\mathrm{S}\mathrm{T}}h$$ for all $h\in \text{Y}_{r,n}.$ Then $\varphi_{\mathrm{S}\mathrm{T}}\in \dot{\mathrm{YS}_{n}^{r}}.$ For each $\bm{\lambda}\in \mathcal{P}_{r,n},$ set $\dot{\mathcal{T}}_{0}^{+}(\bm{\lambda})=\mathcal{T}_{0}^{+}(\bm{\lambda})\cup\mathcal{T}_{0}^{+}(\bm{\lambda}, \omega)=\mathcal{T}_{0}^{+}(\bm{\lambda})\cup \mathrm{Std}(\bm{\lambda}).$

\begin{proposition}\label{Schur-functor-prop4-1}
The algebra $\dot{\mathrm{YS}_{n}^{r}}$ is a cellular algebra with a cellular basis $$\{\varphi_{\mathrm{S}\mathrm{T}}\:|\:S, T\in \dot{\mathcal{T}}_{0}^{+}(\bm{\lambda})~for~some~\bm{\lambda}\in \mathcal{P}_{r,n}\}.$$
Moreover, if $\mathcal{R}=\mathbb{K}$ is a field, then $\dot{\mathrm{YS}_{n}^{r}}$ is a quasi-hereditary algebra with Weyl modules $\{\dot{W}^{\bm{\lambda}}\:|\:\bm{\lambda}\in \mathcal{P}_{r,n}\}$ and simple modules $\{\dot{L}^{\bm{\lambda}}\:|\:\bm{\lambda}\in \mathcal{P}_{r,n}\}.$
\end{proposition}
\begin{proof}
By definition, $\mathrm{YS}_{n}^{r}$ is a subalgebra of $\dot{\mathrm{YS}_{n}^{r}}$ and, as an $\mathcal{R}$-module, $$\dot{\mathrm{YS}_{n}^{r}}=\mathrm{YS}_{n}^{r}\oplus \mathrm{Hom}_{\text{Y}_{r,n}}(M^{\omega}, M_{n}^{r})\oplus \mathrm{Hom}_{\text{Y}_{r,n}}(M_{n}^{r}, M^{\omega})\oplus \mathrm{End}_{\text{Y}_{r,n}}(M^{\omega}, M^{\omega}).$$ For $\bm{\mu}\in \dot{\mathcal{C}}_{r,n},$ there are isomorphisms of $\mathcal{R}$-modules $M^{\bm{\mu}}\cong \mathrm{Hom}_{\text{Y}_{r,n}}(M^{\omega}, M^{\bm{\mu}})$ given by $m_{\mathrm{S}\mathfrak{t}}\mapsto \varphi_{\mathrm{S}\mathfrak{t}},$ for $\mathrm{S}\in \mathcal{T}_{0}^{+}(\bm{\lambda}, \bm{\mu})$ and $\mathfrak{t}\in$ Std$(\bm{\lambda})$ with some $\bm{\lambda} \in \mathcal{P}_{r,n}.$ For $\bm{\nu}\in \dot{\mathcal{C}}_{r,n},$ there are isomorphisms of $\mathcal{R}$-modules $M^{\bm{\nu}\ast}\cong \mathrm{Hom}_{\text{Y}_{r,n}}(M^{\bm{\nu}}, M^{\omega})$ given by $m_{\mathfrak{s}\mathrm{T}}\mapsto \varphi_{\mathfrak{s}\mathrm{T}},$ for $\mathfrak{s}\in$ Std$(\bm{\lambda})$ and $\mathrm{T}\in \mathcal{T}_{0}^{+}(\bm{\lambda}, \bm{\nu})$ with some $\bm{\lambda} \in \mathcal{P}_{r,n},$ where $m_{\mathfrak{s}\mathrm{T}}=m_{\mathrm{T}\mathfrak{s}}^{\ast}.$ Therefore, the elements in the statement of this proposition give a basis of $\dot{\mathrm{YS}_{n}^{r}}$ by Proposition \ref{Yoko-Schur-propo-3-3} and Theorem \ref{yokonu-cellu-basis-3-9}.

Now suppose that $\mathcal{R}=\mathbb{K}$ is a field. Repeating the arguments from Theorem \ref{yokonu-cellu-basis-3-9} and Proposition \ref{yokonu-cellubasis-propo-3-10} shows that $\dot{\mathrm{YS}_{n}^{r}}$ is a quasi-hereditary cellular algebra.
\end{proof}

By Proposition \ref{Schur-functor-prop4-1}, there exist Weyl modules $\dot{W}^{\bm{\lambda}}$ and simple modules $\dot{L}^{\bm{\lambda}}=\dot{W}^{\bm{\lambda}}/\mathrm{rad}\hspace{0.3mm}\dot{W}^{\bm{\lambda}}$ for $\dot{\mathrm{YS}_{n}^{r}}$, for each $\bm{\lambda}\in \mathcal{P}_{r,n}.$ Let $\{\varphi_{\mathrm{S}}\:|\:\mathrm{S}\in \dot{\mathcal{T}}_{0}^{+}(\bm{\lambda})\}$ be the basis of $\dot{W}^{\bm{\lambda}}.$ For each $\bm{\mu}\in \mathcal{C}_{r,n},$ let $\varphi_{\bm{\mu}}$ be the identity map on $M^{\bm{\mu}}.$ We extend $\varphi_{\bm{\mu}}$ to an element of YS$_{n}^{r}$ by defining $\varphi_{\bm{\mu}}$ to be zero on $M^{\bm{\kappa}}$ for $\bm{\mu}\neq \bm{\kappa}\in \mathcal{C}_{r,n}.$ In particular, $\varphi_{\bm{\mu}}=\varphi_{\mathrm{T}^{\bm{\mu}}\mathrm{T}^{\bm{\mu}}}$ if $\bm{\mu}\in \mathcal{P}_{r,n}.$ As an $\mathcal{R}$-module, every $\mathrm{YS}_{n}^{r}$-module $M$ has a weight space decomposition
\begin{equation}\label{schur-func-4-1}
M=\bigoplus_{\bm{\mu}\in \mathcal{C}_{r,n}}M_{\bm{\mu}},\quad \mathrm{where}~M_{\bm{\mu}}=M\varphi_{\bm{\mu}}.
\end{equation}

Set $\varphi_{n}^{r}=\sum_{\bm{\mu}\in \mathcal{C}_{r,n}}\varphi_{\bm{\mu}}$ and let $\varphi_{\omega}$ be the identity map on $M^{\omega}=\text{Y}_{r,n}.$ Then $\varphi_{n}^{r}$ is the identity element of $\mathrm{YS}_{n}^{r}$ and $\varphi_{n}^{r}+\varphi_{\omega}$ is the identity element of $\dot{\mathrm{YS}_{n}^{r}}.$ By definition, $\varphi_{n}^{r}$ and $\varphi_{\omega}$ are both idempotents in $\dot{\mathrm{YS}_{n}^{r}}$ and $\varphi_{n}^{r}\dot{\mathrm{YS}_{n}^{r}}\varphi_{n}^{r}\cong \mathrm{YS}_{n}^{r}.$ Therefore, by [HM2, (2.10)], there are exact functors $$\dot{\mathrm{F}}_{n}^{\omega}: \dot{\mathrm{YS}_{n}^{r}}\mathrm{-mod}\rightarrow \mathrm{YS}_{n}^{r}\mathrm{-mod},\quad \dot{\mathrm{G}}_{n}^{\omega}: \mathrm{YS}_{n}^{r}\mathrm{-mod}\rightarrow \dot{\mathrm{YS}_{n}^{r}}\mathrm{-mod}$$ given by $\dot{\mathrm{F}}_{n}^{\omega}(M)=M\varphi_{n}^{r}$ and $\dot{\mathrm{G}}_{n}^{\omega}(N)=N\otimes_{\mathrm{YS}_{n}^{r}}\varphi_{n}^{r}\dot{\mathrm{YS}_{n}^{r}}.$ By [HM, $\S$2.4], there are functors $\mathrm{H}_{n}^{\omega} :=\mathrm{H}_{\varphi_{n}^{r}},$ $\mathrm{O}_{n}^{\omega} :=\mathrm{O}_{\varphi_{n}^{r}},$ $\mathrm{O}_{\omega}^{n} :=\mathrm{O}^{\varphi_{n}^{r}}$ from $\dot{\mathrm{YS}_{n}^{r}}\mathrm{-mod}$ to $\mathrm{YS}_{n}^{r}\mathrm{-mod}$ such that $\mathrm{H}_{n}^{\omega}(M)=M/\mathrm{O}_{n}^{\omega}(M).$

\begin{lemma}\label{Schur-functor-lemma4-2}
Suppose that $\mathcal{R}=\mathbb{K}$ is a field. Then the functors $\dot{\mathrm{F}}_{n}^{\omega}$ and $\dot{\mathrm{G}}_{n}^{\omega}$ induce mutually inverse equivalences of categories between $\dot{\mathrm{YS}_{n}^{r}}\mathrm{-mod}$ and $\mathrm{YS}_{n}^{r}\mathrm{-mod}.$ Moreover, $\dot{\mathrm{F}}_{n}^{\omega}(\dot{W}^{\bm{\lambda}})\cong W^{\bm{\lambda}}$ and $\dot{\mathrm{F}}_{n}^{\omega}(\dot{L}^{\bm{\lambda}})\cong L^{\bm{\lambda}}$ for all $\bm{\lambda}\in \mathcal{P}_{r,n}.$
\end{lemma}
\begin{proof}
Let $M$ be a $\dot{\mathrm{YS}_{n}^{r}}$-module. Then, extending \eqref{schur-func-4-1}, $M$ has a weight space decomposition $$M=\bigoplus_{\bm{\mu}\in \dot{\mathcal{C}}_{r,n}}M_{\bm{\mu}},~~~~\mathrm{where}~M_{\bm{\mu}}=M\varphi_{\bm{\mu}}.$$ Then, essentially by definition, $\dot{\mathrm{F}}_{n}^{\omega}(M)=\bigoplus_{\bm{\lambda}\in \mathcal{P}_{r,n}}M_{\bm{\lambda}}.$ That is, $\dot{\mathrm{F}}_{n}^{\omega}$ removes the $\omega$-weight space of $M.$ In particular, $\dot{\mathrm{F}}_{n}^{\omega}(\dot{W}^{\bm{\lambda}})=W^{\bm{\lambda}}$ and $\dot{\mathrm{F}}_{n}^{\omega}(\dot{L}^{\bm{\lambda}})=L^{\bm{\lambda}}$ for all $\bm{\lambda}\in \mathcal{P}_{r,n}.$ The fact that $\dot{\mathrm{F}}_{n}^{\omega}(\dot{L}^{\bm{\mu}})=L^{\bm{\mu}}$ for all $\bm{\mu}\in \mathcal{P}_{r,n}$ implies that $\mathrm{O}_{\omega}^{n}(M)=M,$ $\mathrm{O}_{n}^{\omega}(M)=0$ for all $M\in \dot{\mathrm{YS}_{n}^{r}}\mathrm{-mod}.$ Therefore, $\mathrm{H}_{n}^{\omega}$ is the identity functor and $\dot{\mathrm{G}}_{n}^{\omega}\cong \mathrm{H}_{n}^{\omega}\circ\dot{\mathrm{G}}_{n}^{\omega}.$ Hence, this lemma is an application of the theory of quotient functors given in [HM2, Theorem 2.11].
\end{proof}

The identity map $\varphi_{\omega}$ on $\text{Y}_{r,n}=M^{\omega}$ is idempotent in $\dot{\mathrm{YS}_{n}^{r}}$ and there is an isomorphism of $\mathcal{R}$-algebras $\varphi_{\omega}\dot{\mathrm{YS}_{n}^{r}}\varphi_{\omega}\cong \text{Y}_{r,n}.$ Therefore, by [HM2, (2.10)], there are functors $$\dot{\mathrm{F}}_{n}^{r}: \dot{\mathrm{YS}_{n}^{r}}\mathrm{-mod}\rightarrow \text{Y}_{r,n}\mathrm{-mod},\quad \dot{\mathrm{G}}_{n}^{r}: \text{Y}_{r,n}\mathrm{-mod}\rightarrow \dot{\mathrm{YS}_{n}^{r}}\mathrm{-mod}$$ given by $\dot{\mathrm{F}}_{n}^{r}(M)=M\varphi_{\omega}=M_{\omega}$ and $\dot{\mathrm{G}}_{n}^{r}(N)=N\otimes_{\text{Y}_{r,n}}\varphi_{\omega}\dot{\mathrm{YS}_{n}^{r}}.$

\begin{proposition}\label{Schur-functor-propo4-3}
Suppose that $\mathcal{R}=\mathbb{K}$ is a field. Then there is an exact functor $\mathrm{F}_{n}^{r}: \mathrm{YS}_{n}^{r}\mathrm{-mod}\rightarrow \mathrm{Y}_{r,n}\mathrm{-mod}$ given by $\mathrm{F}_{n}^{r}(M)=(M\otimes_{\mathrm{YS}_{n}^{r}}\varphi_{n}^{r}\dot{\mathrm{YS}_{n}^{r}})\varphi_{\omega},$ for $M\in \mathrm{YS}_{n}^{r}\mathrm{-mod},$ such that if $\bm{\lambda}, \bm{\mu}\in \mathcal{P}_{r,n},$ then $\mathrm{F}_{n}^{r}(W^{\bm{\lambda}})\cong S^{\bm{\lambda}},$ and
\begin{align}\label{schur-fun-4-2}
\mathrm{F}_{n}^{r}(L^{\bm{\mu}})\cong\begin{cases}D^{\bm{\mu}}& \hbox {if } \bm{\mu}\in \mathcal{K}_{r,n}; \\0& \hbox {if } \bm{\mu} \notin \mathcal{K}_{r,n}.\end{cases}
\end{align}
\end{proposition}
\begin{proof}
By definition, $\mathrm{F}_{n}^{r}=\dot{\mathrm{F}}_{n}^{r}\circ \dot{\mathrm{G}}_{n}^{\omega},$ so $\mathrm{F}_{n}^{r}$ is an exact functor from $\mathrm{YS}_{n}^{r}\mathrm{-mod}$ to $\text{Y}_{r,n}\mathrm{-mod}.$ The functor $\dot{\mathrm{F}}_{n}^{r}$ is nothing more than projection onto the $\omega$-weight space. Hence, if $\bm{\lambda}\in \mathcal{P}_{r,n},$ then $\dot{\mathrm{F}}_{n}^{r}(\dot{W}^{\bm{\lambda}})$ is spanned by the maps $\{\varphi_{\mathfrak{t}}\:|\:\mathfrak{t}\in \mathrm{Std}(\bm{\lambda})\},$ since $\mathcal{T}_{0}^{+}(\bm{\lambda}, \omega)=$ Std$(\bm{\lambda}).$ The map $\varphi_{\mathfrak{t}}\mapsto m_{\mathfrak{t}},$ for $\mathfrak{t}\in \mathrm{Std}(\bm{\lambda}),$ defines an isomorphism $\dot{\mathrm{F}}_{n}^{r}(\dot{W}^{\bm{\lambda}})\cong S^{\bm{\lambda}}$ of $\text{Y}_{r,n}$-modules. Therefore, $\mathrm{F}_{n}^{r}(W^{\bm{\lambda}})\cong S^{\bm{\lambda}}$ by Lemma \ref{Schur-functor-lemma4-2}.

By [HM2, Theorem 2.11], $\mathrm{F}_{n}^{r}(L^{\bm{\mu}})$ is an irreducible $\text{Y}_{r,n}$-module whenever it is nonzero. Using the fact that $\mathrm{F}_{n}^{r}(W^{\bm{\lambda}})\cong S^{\bm{\lambda}}$ and Corollary \ref{yokonu-cellubasis-corollar-3-11}, a straightforward argument by induction on the dominance ordering shows that $\mathrm{F}_{n}^{r}(L^{\bm{\mu}})\cong D^{\bm{\mu}}$ if $\bm{\mu}\in \mathcal{K}_{r,n}$ and that $\mathrm{F}_{n}^{r}(L^{\bm{\mu}})=0$ otherwise.
\end{proof}

Since $\mathrm{F}_{n}^{r}$ is exact, we obtain the promised relationship between the decomposition numbers of $\mathrm{YS}_{n}^{r}$ and $\text{Y}_{r,n}.$
\begin{corollary}\label{Schur-functor-corollar4-4}
Suppose that $\mathcal{R}=\mathbb{K}$ is a field and that $\bm{\lambda}\in \mathcal{P}_{r,n},$ $\bm{\mu}\in \mathcal{K}_{r,n}$. Then we have $[S^{\bm{\lambda}}: D^{\bm{\mu}}]=[W^{\bm{\lambda}}: L^{\bm{\mu}}].$
\end{corollary}

\begin{lemma}\label{Schur-functor-lemma4-5} {\rm (A double centralizer property)} There are canonical isomorphisms of algebras such that $\dot{\mathrm{YS}_{n}^{r}}=\mathrm{End}_{\mathrm{Y}_{r,n}}(\dot{M}_{n}^{r})$ and $\mathrm{Y}_{r,n}=\mathrm{End}_{\dot{\mathrm{YS}_{n}^{r}}}(\dot{M}_{n}^{r}).$ In particular, the functor $\dot{\mathrm{F}}_{n}^{r}$ is fully faithful on projectives.
\end{lemma}
\begin{proof}
The first isomorphism is the definition of $\dot{\mathrm{YS}_{n}^{r}}$, whereas the second follows directly from the definition of $\dot{\mathrm{YS}_{n}^{r}}$ because $$\text{Y}_{r,n}\cong \mathrm{Hom}_{\text{Y}_{r,n}}(\text{Y}_{r,n}, \text{Y}_{r,n})\cong \varphi_{\omega}\dot{\mathrm{YS}_{n}^{r}}\varphi_{\omega}\cong \mathrm{End}_{\dot{\mathrm{YS}_{n}^{r}}}(\varphi_{\omega}\dot{\mathrm{YS}_{n}^{r}}),$$ and $\varphi_{\omega}\dot{\mathrm{YS}_{n}^{r}}\cong \dot{M}_{n}^{r}$ as a right $\dot{\mathrm{YS}_{n}^{r}}$-module.
\end{proof}

\begin{corollary}\label{Schur-functor-corollar4-6}
$\mathrm{YS}_{n}^{r}$ is a quasi-hereditary cover of $\mathrm{Y}_{r,n}$ in the sense of Rouquier $[$Ro, Definition 4.34$].$
\end{corollary}
\begin{proof}
Recall that $\dot{M}_{n}^{r}\cong\varphi_{\omega}\dot{\mathrm{YS}_{n}^{r}}$ is a projective $\dot{\mathrm{YS}_{n}^{r}}$-module. Using the Morita equivalence between $\dot{\mathrm{YS}_{n}^{r}}$ and $\mathrm{YS}_{n}^{r}$, we see that $M_{n}^{r}$ is a projective $\mathrm{YS}_{n}^{r}$-module. Because $\dot{\mathrm{F}}_{n}^{r}$ is fully faithful on projective modules by Lemma \ref{Schur-functor-lemma4-5} and $\mathrm{F}_{n}^{r}$ is the composition of $\dot{\mathrm{F}}_{n}^{r}$ with an equivalence of categories, so is $\mathrm{F}_{n}^{r}$. This implies that $\mathrm{YS}_{n}^{r}$ is a quasi-hereditary cover of $\text{Y}_{r,n}$ in the sense of Rouquier [Ro, Definition 4.34].
\end{proof}

\section{Tilting modules}

In this section, we introduce the tilting modules for $\mathrm{YS}_{n}^{r}$ and the closely related Young modules for $\text{Y}_{r,n}$ following [Ma2]. Throughout this section we assume that $\mathbb{K}$ is a field, which can be regarded as an $\mathcal{R}$-algebra.

\subsection{Young modules}

Recall from \eqref{schur-func-4-1} that every $\mathrm{YS}_{n}^{r}$-module has a weight space decomposition. Analogously, as a right $\mathrm{YS}_{n}^{r}$-module, the regular representation of $\mathrm{YS}_{n}^{r}$ has a decomposition into a direct sum of left weight spaces $$\mathrm{YS}_{n}^{r}=\bigoplus_{\bm{\mu}\in \mathcal{C}_{r,n}}Z^{\bm{\mu}},\quad \mathrm{where}~Z^{\bm{\mu}}=\varphi_{\bm{\mu}}\mathrm{YS}_{n}^{r}~\mathrm{for}~\bm{\mu}\in \mathcal{C}_{r,n}.$$ The next lemma gives some properties of the right $\mathrm{YS}_{n}^{r}$-module $Z^{\bm{\mu}}.$

\begin{lemma}\label{tiltmod-lemma5-1}
Assume that $\bm{\mu}\in \mathcal{C}_{r,n}.$ Then the following hold$:$

$\mathrm{(i)}$ $Z^{\bm{\mu}}$ is free as an $\mathcal{R}$-module with a basis $$\big\{\varphi_{\mathrm{S}\mathrm{T}}\:|\:\mathrm{S}\in \mathcal{T}_{0}^{+}(\bm{\lambda}, \bm{\mu}), \mathrm{T}\in \mathcal{T}_{0}^{+}(\bm{\lambda}, \bm{\nu})~for~some~\bm{\nu}\in \mathcal{C}_{r,n}~and~\bm{\lambda}\in \mathcal{P}_{r,n}\big\}.$$

$\mathrm{(ii)}$ Let $\mathcal{M}^{\bm{\mu}}=\mathrm{Hom}_{\mathrm{Y}_{r,n}}(M_{n}^{r}, M^{\bm{\mu}}).$ As right $\mathrm{YS}_{n}^{r}$-modules, we have $Z^{\bm{\mu}}\cong\mathcal{M}^{\bm{\mu}}.$

$\mathrm{(iii)}$ As $\mathrm{Y}_{r,n}$-modules, we have $\mathrm{F}_{n}^{r}(Z^{\bm{\mu}})\cong M^{\bm{\mu}}.$
\end{lemma}
\begin{proof}
(i) It follows from Theorem \ref{yokonu-cellu-basis-3-9}.

(ii) It follows from (i).

(iii) We have
\begin{align*}
\mathrm{F}_{n}^{r}(Z^{\bm{\mu}})&=\dot{\mathrm{F}}_{n}^{r}\circ \dot{\mathrm{G}}_{n}^{\omega}(\varphi_{\bm{\mu}}\mathrm{YS}_{n}^{r})=
\dot{\mathrm{F}}_{n}^{r}(\varphi_{\bm{\mu}}\mathrm{YS}_{n}^{r}\otimes_{\mathrm{YS}_{n}^{r}}\varphi_{n}^{r}\dot{\mathrm{YS}_{n}^{r}})\\&\cong \dot{\mathrm{F}}_{n}^{r}(\varphi_{\bm{\mu}}\dot{\mathrm{YS}_{n}^{r}})\cong \mathrm{Hom}_{\text{Y}_{r,n}}(\text{Y}_{r,n}, M^{\bm{\mu}})\cong M^{\bm{\mu}}.
\end{align*}
We are done.
\end{proof}

The next lemma claims that each $Z^{\bm{\mu}}$ has a Weyl filtration.
\begin{lemma}\label{tiltmod-lemma5-2}
Assume that $\bm{\mu}\in \mathcal{C}_{r,n}.$ Then $Z^{\bm{\mu}}$ has a Weyl filtration $$Z^{\bm{\mu}}=M_{1}\supset M_{2}\supset\cdots\supset M_{k}\supset M_{k+1}=0$$ such that for each $1\leq i\leq k$ there exists some $\bm{\lambda}_{i}\in \mathcal{P}_{r,n}$ with $\alpha(\bm{\lambda}_{i})=\alpha(\bm{\mu})$ satisfying $M_{i}/M_{i+1}\cong W^{\bm{\lambda}_{i}}.$ Moreover, $\sharp\{1\leq i\leq k\:|\:\bm{\lambda}_{i}=\bm{\lambda}\}=\sharp\mathcal{T}_{0}^{+}(\bm{\lambda}, \bm{\mu})$ for each $\bm{\lambda}\in \mathcal{P}_{r,n}$ with $\alpha(\bm{\lambda})=\alpha(\bm{\mu}).$
\end{lemma}
\begin{proof}
Choose a total ordering $\{\mathrm{S}_{1},\ldots,\mathrm{S}_{k}\}$ on the set $\cup_{\bm{\lambda}\in \mathcal{P}_{r,n}}\mathcal{T}_{0}^{+}(\bm{\lambda}, \bm{\mu})$ such that $i> j$ whenever $\bm{\lambda}_{i}\rhd \bm{\lambda}_{j},$ where $\mathrm{S}_{i}\in \mathcal{T}_{0}^{+}(\bm{\lambda}_{i}, \bm{\mu}),$ $\mathrm{S}_{j}\in \mathcal{T}_{0}^{+}(\bm{\lambda}_{j}, \bm{\mu}).$ For each $1\leq i\leq k,$ let $M_{i}$ be the $\mathcal{R}$-submodule of $Z^{\bm{\mu}}$ with basis $\{\varphi_{\mathrm{S}_{j}\mathrm{T}}\:|\:j\geq i~\mathrm{and}~\mathrm{T}\in \mathcal{T}_{0}^{+}(\bm{\lambda}_{j})\}.$ Then $M_{i}$ is a $\mathrm{YS}_{n}^{r}$-module by Theorem \ref{yokonu-cellu-basis-3-9}. Further, there is an isomorphism of $\mathrm{YS}_{n}^{r}$-modules $W^{\bm{\lambda}_{i}}\cong M_{i}/M_{i+1}$ given by $\varphi_{\mathrm{T}}\mapsto \varphi_{\mathrm{S}_{i}\mathrm{T}}+M_{i+1}$ for $\mathrm{T}\in \mathcal{T}_{0}^{+}(\bm{\lambda}_{i}),$ because $M_{i}\cap\mathrm{YS}_{r,n}^{\rhd \bm{\lambda}_{i}}\subseteq M_{i+1}.$ Since $\mathrm{YS}_{n}^{r}$ is quasi-hereditary, $[Z^{\bm{\mu}}: W^{\bm{\lambda}}]$ is independent of the choice of Weyl filtration.
\end{proof}

Applying the Schur functor $\mathrm{F}_{n}^{r},$ by Proposition \ref{Schur-functor-propo4-3} and Lemma \ref{tiltmod-lemma5-1}(iii), we also have the following result.
\begin{corollary}\label{tiltmod-corollar5-3}
Assume that $\bm{\mu}\in \mathcal{C}_{r,n}.$ Then $M^{\bm{\mu}}$ has a Specht filtration $$M^{\bm{\mu}}=M_{1}\supset M_{2}\supset\cdots\supset M_{k}\supset M_{k+1}=0$$ such that for each $1\leq i\leq k$ there exists some $\bm{\lambda}_{i}\in \mathcal{P}_{r,n}$ with $\alpha(\bm{\lambda}_{i})=\alpha(\bm{\mu})$ satisfying $M_{i}/M_{i+1}\cong S^{\bm{\lambda}_{i}}.$ Moreover, $\sharp\{1\leq i\leq k\:|\:\bm{\lambda}_{i}=\bm{\lambda}\}=\sharp\mathcal{T}_{0}^{+}(\bm{\lambda}, \bm{\mu})$ for each $\bm{\lambda}\in \mathcal{P}_{r,n}$ with $\alpha(\bm{\lambda})=\alpha(\bm{\mu}).$
\end{corollary}

Since $\varphi_{\bm{\mu}}$ is an idempotent in $\mathrm{YS}_{n}^{r},$ $Z^{\bm{\mu}}$ is a projective $\mathrm{YS}_{n}^{r}$-module. Notice that if $\mathcal{T}_{0}(\bm{\lambda}, \bm{\mu})\neq \emptyset,$ then $\bm{\lambda} \unrhd\bm{\mu}.$ Thus, $W^{\bm{\lambda}}$ appears in $Z^{\bm{\mu}}$ only if $\bm{\lambda} \unrhd\bm{\mu}.$ For each $\bm{\mu}\in \mathcal{P}_{r,n},$ let $P^{\bm{\mu}}$ be the projective cover of $L^{\bm{\mu}}.$ Then by [Ma1, Lemma 2.16], $P^{\bm{\mu}}$ has a filtration by Weyl modules in which $W^{\bm{\lambda}}$ appears with multiplicity $[P^{\bm{\mu}}: W^{\bm{\lambda}}]=[W^{\bm{\lambda}}: L^{\bm{\mu}}].$ From these facts, we can easily get the following lemma.

\begin{lemma}\label{tiltmod-lemma5-4}
Assume that $\bm{\mu}\in \mathcal{P}_{r,n}.$ Then $$Z^{\bm{\mu}}\cong P^{\bm{\mu}}\oplus \bigoplus_{\bm{\lambda}\rhd\bm{\mu}}c_{\bm{\lambda}\bm{\mu}}P^{\bm{\lambda}}$$ for some non-negative integer $c_{\bm{\lambda}\bm{\mu}}.$
\end{lemma}

Suppose that $\bm{\lambda}\in \mathcal{P}_{r,n}$ and $\mathrm{S}\in \mathcal{T}_{0}^{+}(\bm{\lambda}, \bm{\mu}),$ $\mathrm{T}\in \mathcal{T}_{0}^{+}(\bm{\lambda}, \bm{\nu})$. Recall that $\mathcal{M}^{\bm{\mu}}$ is defined as $\mathrm{Hom}_{\text{Y}_{r,n}}(M_{n}^{r}, M^{\bm{\mu}}).$ Thus, we can define a $\mathrm{YS}_{n}^{r}$-module homomorphism $\Phi_{\mathrm{S}\mathrm{T}}: \mathcal{M}^{\bm{\nu}}\rightarrow \mathcal{M}^{\bm{\mu}}$ by $\Phi_{\mathrm{S}\mathrm{T}}(f)=\varphi_{\mathrm{S}\mathrm{T}}f$ for all $f\in \mathcal{M}^{\bm{\nu}}.$ In fact, these maps give a basis of $\mathrm{Hom}_{\mathrm{YS}_{n}^{r}}(\mathcal{M}^{\bm{\nu}}, \mathcal{M}^{\bm{\mu}}).$

\begin{lemma}\label{tiltmod-lemma5-5}
Suppose that $\bm{\mu}, \bm{\nu}\in \mathcal{C}_{r,n}.$ Then $\mathrm{Hom}_{\mathrm{YS}_{n}^{r}}(\mathcal{M}^{\bm{\nu}}, \mathcal{M}^{\bm{\mu}})$ is free as an $\mathcal{R}$-module with basis $\{\Phi_{\mathrm{S}\mathrm{T}}\:|\:\mathrm{S}\in \mathcal{T}_{0}^{+}(\bm{\lambda}, \bm{\mu}), \mathrm{T}\in \mathcal{T}_{0}^{+}(\bm{\lambda}, \bm{\nu})~for~some~\bm{\lambda}\in \mathcal{P}_{r,n}\}.$
\end{lemma}

For each $\bm{\lambda}\in \mathcal{P}_{r,n},$ let $Y^{\bm{\lambda}}=\mathrm{F}_{n}^{r}(P^{\bm{\lambda}}),$ which we call a Young module of $\text{Y}_{r,n}.$

\begin{proposition}\label{tiltmod-propo5-6}
Suppose that $\mathcal{R}=\mathbb{K}$ is a field and that $\bm{\lambda}\in \mathcal{P}_{r,n}.$ Then the following hold$:$

$\mathrm{(i)}$ Each $Y^{\bm{\lambda}}$ is an indecomposable $\mathrm{Y}_{r,n}$-module.

$\mathrm{(ii)}$ If $\bm{\mu}$ is another $r$-partition of $n,$ then $Y^{\bm{\lambda}}\cong Y^{\bm{\mu}}$ if and only if $\bm{\lambda}=\bm{\mu}.$

$\mathrm{(iii)}$ We have $$M^{\bm{\lambda}}\cong Y^{\bm{\lambda}}\oplus \bigoplus_{\bm{\nu}\rhd\bm{\lambda}}c_{\bm{\nu}\bm{\lambda}}Y^{\bm{\nu}}.$$

$\mathrm{(iv)}$ The Young module $Y^{\bm{\lambda}}$ has a Specht filtration in which $S^{\bm{\mu}}$ appears with multiplicity $[Y^{\bm{\lambda}}: S^{\bm{\mu}}]=[W^{\bm{\mu}}: L^{\bm{\lambda}}].$
\end{proposition}
\begin{proof}
(i) By Corollary \ref{Schur-functor-corollar4-6}, the functor $\mathrm{F}_{n}^{r}$ is fully faithful on projective modules, so $\mathrm{End}_{\text{Y}_{r,n}}(Y^{\bm{\lambda}})\cong \mathrm{End}_{\mathrm{YS}_{n}^{r}}(P^{\bm{\lambda}})$ is a local ring since $P^{\bm{\lambda}}$ is indecomposable.

(ii) If $Y^{\bm{\lambda}}\cong Y^{\bm{\mu}},$ then $\mathrm{Hom}_{\text{Y}_{r,n}}(Y^{\bm{\lambda}}, Y^{\bm{\mu}})$ contains an isomorphism and this lifts to give an isomorphism $P^{\bm{\mu}}\cong P^{\bm{\lambda}},$ so $\bm{\lambda}=\bm{\mu}.$

(iii) Applying the Schur functor $\mathrm{F}_{n}^{r},$ it follows from Lemma \ref{tiltmod-lemma5-1}(iii) and Lemma \ref{tiltmod-lemma5-4}.

(iv) Recall that $P^{\bm{\lambda}}$ has a Weyl filtration $P^{\bm{\lambda}}=P_{1}\supset P_{2}\supset \cdots P_{k}\supset P_{k+1}=0.$ Moreover, for each $\bm{\mu}\in \mathcal{P}_{r,n},$ $\sharp\{1\leq i\leq k\:|\:P_{i}/P_{i+1}\cong W^{\bm{\mu}}\}=[P^{\bm{\lambda}}: W^{\bm{\mu}}]=[W^{\bm{\mu}}: L^{\bm{\lambda}}].$ Setting $Y_{i}=\mathrm{F}_{n}^{r}(P_{i}),$ and using Proposition \ref{Schur-functor-propo4-3}, gives a filtration of $Y^{\bm{\lambda}}$ with the required properties.
\end{proof}

The next proposition identify the projective Young modules, and its proof is similar to that of [HM2, Proposition 5.9].

\begin{proposition}\label{tiltmod-propo5-7}
Suppose that $\bm{\mu}\in \mathcal{K}_{r,n}.$ Then $Y^{\bm{\mu}}$ is the projective cover of $D^{\bm{\mu}}.$
\end{proposition}
\begin{proof}
Recall that $P^{\bm{\mu}}$ is the projective cover of $L^{\bm{\mu}}$ and $\mathrm{F}_{n}^{r}(P^{\bm{\mu}})=Y^{\bm{\mu}},$ $\mathrm{F}_{n}^{r}(L^{\bm{\mu}})=D^{\bm{\mu}}$ if $\bm{\mu}\in \mathcal{K}_{r,n}.$ Since $\mathrm{F}_{n}^{r}$ is exact, there is a surjective map $Y^{\bm{\mu}}\mapsto D^{\bm{\mu}}.$ Therefore, it suffices to show that $Y^{\bm{\mu}}$ is projective since it is indecomposable by Proposition \ref{tiltmod-propo5-6}(i).

Recall that $\dot{\mathrm{YS}_{n}^{r}}=\mathrm{End}_{\text{Y}_{r,n}}(\dot{M}_{n}^{r}),$ where $\dot{M}_{n}^{r}=M_{n}^{r}\oplus \text{Y}_{r,n}.$ There is also a Schur functor $\dot{\mathrm{F}}_{n}^{r}$ from $\dot{\mathrm{YS}_{n}^{r}}\mathrm{-mod}$ to $\text{Y}_{r,n}\mathrm{-mod}$ given by $\dot{\mathrm{F}}_{n}^{r}(M)=M\varphi_{\omega}.$ In particular, $\dot{\mathrm{F}}_{n}^{r}(\dot{M}_{n}^{r})\cong \text{Y}_{r,n}$ as right $\text{Y}_{r,n}$-modules.

As a $\dot{\mathrm{YS}_{n}^{r}}$-module, $\dot{M}_{n}^{r}\cong\varphi_{\omega}\dot{\mathrm{YS}_{n}^{r}}.$ In particular, $\dot{M}_{n}^{r}$ is a projective $\dot{\mathrm{YS}_{n}^{r}}$-module. If $\bm{\lambda}\in \mathcal{P}_{r,n},$ let $\dot{P}^{\bm{\lambda}}$ be the projective cover of the irreducible $\dot{\mathrm{YS}_{n}^{r}}$-module $\dot{L}^{\bm{\lambda}}.$ The multiplicity of $\dot{P}^{\bm{\lambda}}$ as a summand of $\dot{M}_{n}^{r}$ is equal to $$\dim \mathrm{Hom}_{\dot{\mathrm{YS}_{n}^{r}}}(\dot{M}_{n}^{r}, \dot{L}^{\bm{\lambda}})=\dim \mathrm{Hom}_{\dot{\mathrm{YS}_{n}^{r}}}(\varphi_{\omega}\dot{\mathrm{YS}_{n}^{r}}, \dot{L}^{\bm{\lambda}})=\dim \dot{L}^{\bm{\lambda}}\varphi_{\omega}=\dim D^{\bm{\lambda}}.$$ Consequently, $\dot{M}_{n}^{r}\cong \bigoplus_{\bm{\lambda}\in \mathcal{K}_{r,n}}(\dim D^{\bm{\lambda}})\dot{P}^{\bm{\lambda}}$ as a $\dot{\mathrm{YS}_{n}^{r}}$-module. By definition, $Y^{\bm{\lambda}}=\mathrm{F}_{n}^{r}(P^{\bm{\lambda}})=\dot{\mathrm{F}}_{n}^{r}(\dot{P}^{\bm{\lambda}})$ for all $\bm{\lambda}\in \mathcal{P}_{r,n}.$ Therefore, $$\text{Y}_{r,n}\cong \dot{\mathrm{F}}_{n}^{r}(\dot{M}_{n}^{r})\cong\bigoplus_{\bm{\lambda}\in \mathcal{K}_{r,n}}(\dim D^{\bm{\lambda}})Y^{\bm{\lambda}}$$ as a right $\text{Y}_{r,n}$-module. The result follows.
\end{proof}

\subsection{Twisted Young modules}
Let $\mathcal{Z}=\mathbb{Z}[\frac{1}{r}][\dot{q},\dot{q}^{-1},\zeta],$ where $\dot{q}$ is an indeterminate, and let $\text{Y}_{r,n}^{\mathcal{Z}}$ be the Yokonuma-Hecke algebra over $\mathcal{Z}.$ It is easy to see that $\text{Y}_{r,n}^{\mathcal{Z}}$ has a $\mathbb{Z}$-algebra involution $'$ which is determined by $$g_{i}'=g_{i},\quad \dot{q}'=-\dot{q}^{-1},\quad \text{and }t_{j}'=t_{j},\quad \zeta_{j}'=\zeta_{r-j+1}$$ for $1\leq i\leq n-1$ and $1\leq j\leq n.$

For each $\bm{\mu}\in \mathcal{C}_{r,n},$ let $y_{\bm{\mu}}=(x_{\bm{\mu}})'=\sum_{w\in \mathfrak{S}_{\bm{\mu}}}(-\dot{q})^{-l(w)}g_{w},$ and set $n_{\bm{\mu}} :=(U_{\bm{\mu}}x_{\bm{\mu}})'=U_{\bm{\mu}}'y_{\bm{\mu}}.$ Suppose that $\bm{\lambda}\in \mathcal{P}_{r,n}$ and $\mathfrak{s}, \mathfrak{t}\in \mathrm{Std}(\bm{\lambda}).$ We define $n_{\mathfrak{s}\mathfrak{t}} :=g_{d(\mathfrak{s})}^{\ast}n_{\bm{\lambda}}g_{d(\mathfrak{t})}.$ Then by definition we have $n_{\mathfrak{s}\mathfrak{t}}=(m_{\mathfrak{s}\mathfrak{t}})'.$ Because $'$ is a $\mathbb{Z}$-algebra involution, $\{n_{\mathfrak{s}\mathfrak{t}}\}$ is a cellular basis of $\text{Y}_{r,n}^{\mathcal{Z}}$ by Theorem \ref{cellula-bases-2}. The ring $\mathcal{R}$ is naturally a $\mathcal{Z}$-module under specialization; that is, $\dot{q}$ acts on $\mathcal{R}$ as multiplication by $q.$ Because $\text{Y}_{r,n}$ is $\mathcal{R}$-free, this induces an isomorphism of $\mathcal{R}$-algebras $\text{Y}_{r,n}\cong \text{Y}_{r,n}^{\mathcal{Z}}\otimes_{\mathcal{Z}} \mathcal{R}$ via $g_{i}\mapsto g_{i}\otimes 1_{\mathcal{R}}$ ($1\leq i\leq n-1$) and $t_{j}\mapsto t_{j}\otimes 1_{\mathcal{R}}$ ($1\leq j\leq n$). Hereafter, we will identify the algebra $\text{Y}_{r,n}$ and $\text{Y}_{r,n}^{\mathcal{Z}}\otimes_{\mathcal{Z}} \mathcal{R}$ via the isomorphism above. Thus, we have the following result.\\

~~~~\emph{The Yokonuma-Hecke algebra $\mathrm{Y}_{r,n}$ is free as an $\mathcal{R}$-module with a cellular basis $\{n_{\mathfrak{s}\mathfrak{t}}\:|\:\mathfrak{s}, \mathfrak{t}\in \mathrm{Std}(\bm{\lambda})$~for~some~$\bm{\lambda} \in \mathcal{P}_{r,n}\}.$}\\

Now we can apply the general theory of cellular algebras. For each $\bm{\lambda}\in \mathcal{P}_{r,n}$, we define the dual Specht module $S_{\bm{\lambda}}$ to be the right $\text{Y}_{r,n}$-module $(n_{\bm{\lambda}}+Y^{r,n}_{\rhd \bm{\lambda}})\text{Y}_{r,n},$ where $Y^{r,n}_{\rhd \bm{\lambda}}=(\text{Y}_{r,n}^{\rhd \bm{\lambda}})'$ is the two-sided ideal of $\text{Y}_{r,n}$ with basis $n_{\mathfrak{u}\mathfrak{v}}$ with $\mathfrak{u}, \mathfrak{v}\in \mathrm{Std}(\bm{\nu})$ for various $\bm{\nu}\in \mathcal{P}_{r,n}$ such that $\bm{\nu}\rhd\bm{\lambda}.$ Then $S_{\bm{\lambda}}$ is $\mathcal{R}$-free with basis $\{n_{\mathfrak{t}}\:|\:\mathfrak{t}\in \mathrm{Std}(\bm{\lambda})\},$ where $n_{\mathfrak{t}}=n_{\mathfrak{t}^{\bm{\lambda}}\mathfrak{t}}+Y^{r,n}_{\rhd \bm{\lambda}}.$ Let $D_{\bm{\lambda}}=S_{\bm{\lambda}}/\mathrm{rad}\hspace{0.3mm}S_{\bm{\lambda}},$ where $\mathrm{rad}\hspace{0.3mm}S_{\bm{\lambda}}$ is the radical of the bilinear form on $S_{\bm{\lambda}}$ which is defined with respect to the cellular basis $\{n_{\mathfrak{s}\mathfrak{t}}\}.$

For each $\bm{\mu}\in \mathcal{C}_{r,n},$ let $N^{\bm{\mu}}=n_{\bm{\mu}}\text{Y}_{r,n}.$ If $\mathrm{S}\in \mathcal{T}_{0}^{+}(\bm{\lambda}, \bm{\mu})$ and $\mathfrak{t}\in \mathrm{Std}(\bm{\lambda}),$ we define $$n_{\mathrm{S}\mathfrak{t}}=\sum_{\substack{\mathfrak{s}\in \mathrm{Std}(\bm{\lambda})\\\bm{\mu}(\mathfrak{s})=\mathrm{S}}}(-q)^{-l(d(\mathfrak{s}))-l(d(\mathfrak{t}))}n_{\mathfrak{s}\mathfrak{t}}.$$ From the definition, we have $n_{\mathrm{S}\mathfrak{t}}=(m_{\mathrm{S}\mathfrak{t}})'.$ Therefore, Proposition \ref{Yoko-Schur-propo-3-3} and the usual specialization argument show that the following holds.
\begin{corollary}\label{tiltmod-corollar5-8}
Suppose that $\bm{\mu}\in \mathcal{C}_{r,n}.$ Then $N^{\bm{\mu}}$ is free as an $\mathcal{R}$-module with basis $\big\{n_{\mathrm{S}\mathfrak{t}}\:|\:\mathrm{S}\in \mathcal{T}_{0}^{+}(\bm{\lambda}, \bm{\mu})~\mathrm{and}~\mathfrak{t}\in \mathrm{Std}(\bm{\lambda})~for~some~\bm{\lambda}\in \mathcal{P}_{r,n}\big\}.$
\end{corollary}

Let $\bm{\mu}, \bm{\nu}\in \mathcal{C}_{r,n}$ and $\bm{\lambda}\in \mathcal{P}_{r,n}.$ Suppose that $\alpha(\bm{\mu})=\alpha(\bm{\nu})=\alpha(\bm{\lambda}).$ For $\mathrm{S}\in \mathcal{T}_{0}^{+}(\bm{\lambda}, \bm{\mu})$, $\mathrm{T}\in \mathcal{T}_{0}^{+}(\bm{\lambda}, \bm{\nu}),$ let
$$n_{\mathrm{S}\mathrm{T}}=\sum_{\substack{\mathfrak{s},\:\mathfrak{t}\in \mathrm{Std}(\bm{\lambda})\\\bm{\mu}(\mathfrak{s})=\mathrm{S}, \:\bm{\nu}(\mathfrak{t})=\mathrm{T}}}(-q)^{-l(d(\mathfrak{s}))-l(d(\mathfrak{t}))}n_{\mathfrak{s}\mathfrak{t}}.$$

We can now define the twisted Yokonuma-Schur algebra as $$\mathrm{YS}^{n}_{r}=\mathrm{End}_{\text{Y}_{r,n}}(N_{n}^{r}),$$ where $N_{n}^{r}=\bigoplus_{\bm{\mu}\in \mathcal{C}_{r,n}}N^{\bm{\mu}}.$ For $\mathrm{S}\in \mathcal{T}_{0}^{+}(\bm{\lambda}, \bm{\mu})$ and $\mathrm{T}\in \mathcal{T}_{0}^{+}(\bm{\lambda}, \bm{\nu}),$ we can also define the homomorphism $\varphi_{\mathrm{S}\mathrm{T}}'$ by $\varphi_{\mathrm{S}\mathrm{T}}'(n_{\bm{\alpha}}h)=\delta_{\bm{\alpha}\bm{\nu}}n_{\mathrm{S}\mathrm{T}}h$ for all $h\in \text{Y}_{r,n}$ and $\bm{\alpha}\in \mathcal{C}_{r,n}.$ Then $\varphi_{\mathrm{S}\mathrm{T}}'\in \mathrm{YS}^{n}_{r}.$ The proof of the next proposition is in exactly the same way as that of [Ma2, Proposition 4.3], and we skip the details.

\begin{proposition}\label{tiltmod-proposi5-9}
$\mathrm{(i)}$ The twisted Yokonuma-Schur algebra $\mathrm{YS}^{n}_{r}$ is free as an $\mathcal{R}$-module with a cellular basis $$\big\{\varphi_{\mathrm{S}\mathrm{T}}'\:|\:\mathrm{S}, \mathrm{T}\in \mathcal{T}_{0}^{+}(\bm{\lambda})~for~some~\bm{\mu}, \bm{\nu}\in\mathcal{C}_{r,n}~and~\bm{\lambda}\in \mathcal{P}_{r,n}\big\}.$$

$\mathrm{(ii)}$ The twisted Yokonuma-Schur algebra $\mathrm{YS}^{n}_{r}$ is quasi-hereditary.

$\mathrm{(iii)}$ The $\mathcal{R}$-algebras $\mathrm{YS}_{n}^{r}$ and $\mathrm{YS}^{n}_{r}$ are canonically isomorphic.
\end{proposition}

Let $W_{\bm{\lambda}}$ and $L_{\bm{\lambda}}$ ($\bm{\lambda}\in \mathcal{P}_{r,n}$) be the Weyl modules and simple modules of $\mathrm{YS}^{n}_{r},$ respectively; they are defined in exactly the same way as the corresponding modules for $\mathrm{YS}_{n}^{r}.$ As in Section 4, we can define an exact Schur functor $\mathrm{F}^{n}_{r}$ from $\mathrm{YS}^{n}_{r}\mathrm{-mod}$ to $\text{Y}_{r,n}\mathrm{-mod}.$ Moreover, we have $\mathrm{F}^{n}_{r}(W_{\bm{\lambda}})\cong S_{\bm{\lambda}}$, $\mathrm{F}^{n}_{r}(L_{\bm{\lambda}})\cong D_{\bm{\lambda}}$ and $[W_{\bm{\lambda}}: L_{\bm{\mu}}]=[S_{\bm{\lambda}}: D_{\bm{\mu}}]$ whenever $D_{\bm{\mu}}\neq 0.$

For each $\bm{\lambda}\in \mathcal{P}_{r,n},$ let $P_{\bm{\lambda}}$ be the projective cover of $L_{\bm{\lambda}}.$ Define $Y_{\bm{\lambda}}=\mathrm{F}^{n}_{r}(P_{\bm{\lambda}}),$ which is called a twisted Young module. The next proposition can be proved in exactly the same way as in Proposition \ref{tiltmod-propo5-6}.

\begin{proposition}\label{tiltmod-proposi5-10}
Suppose that $\mathcal{R}=\mathbb{K}$ is a field and that $\bm{\mu}\in \mathcal{P}_{r,n}.$ Then we have

$\mathrm{(i)}$ Each $Y_{\bm{\mu}}$ is an indecomposable $\text{Y}_{r,n}$-module.

$\mathrm{(ii)}$ If $\bm{\lambda}$ is another $r$-partition of $n,$ then $Y_{\bm{\lambda}}\cong Y_{\bm{\mu}}$ if and only if $\bm{\lambda}=\bm{\mu}.$

$\mathrm{(iii)}$ $$N^{\bm{\mu}}\cong Y_{\bm{\mu}}\oplus \bigoplus_{\bm{\lambda}\rhd\bm{\mu}}c_{\bm{\lambda}\bm{\mu}}Y_{\bm{\lambda}},$$ where the integers $c_{\bm{\lambda}\bm{\mu}}$ are the same as those appearing in Lemma \ref{tiltmod-lemma5-4}.

$\mathrm{(iv)}$ The twisted Young module $Y_{\bm{\mu}}$ has a dual Specht filtration in which the number of subquotients equal to $S_{\bm{\lambda}}$ is $[W_{\bm{\lambda}}: L_{\bm{\mu}}].$
\end{proposition}

\subsection{Non-degenerate bilinear forms}

If $\sigma$ is a composition, its conjugate is the partition $\sigma'=(\sigma_{1}',\sigma_{2}',\ldots),$ where $\sigma_{i}'$ is the number of nodes in column $i$ of the diagram of $\sigma.$ If $\bm{\lambda}=(\lambda^{(1)}, \ldots,\lambda^{(r)})\in \mathcal{C}_{r,n},$ its conjugate $\bm{\lambda}'$ is the $r$-partition $\bm{\lambda}'=((\lambda^{(r)})', \ldots,(\lambda^{(1)})').$ Similarly, the conjugate of a standard $\bm{\lambda}$-tableau $\mathfrak{t}=(\mathfrak{t}^{(1)},\ldots,\mathfrak{t}^{(r)})$ is the standard $\bm{\lambda}'$-tableau $\mathfrak{t}'=((\mathfrak{t}^{(r)})',\ldots,(\mathfrak{t}^{(1)})'),$ where $(\mathfrak{t}^{(k)})'$ is the tableau obtained by interchanging the rows and columns of $\mathfrak{t}^{(k)}.$

If $\mathfrak{v}$ is a standard tableau, let $\mathfrak{v}_{\downarrow k}$ be the subtableau of $\mathfrak{v}$ which contains $1,2,\ldots,k,$ and let $\mathrm{shape}(\mathfrak{v}_{\downarrow k})$ be the associated $r$-composition. Let $\bm{\lambda}, \bm{\mu}\in \mathcal{P}_{r,n}.$ Suppose that $\mathfrak{s}$ is a standard $\bm{\lambda}$-tableau and that $\mathfrak{t}$ is a standard $\bm{\mu}$-tableau, we say that $\mathfrak{s}$ dominates $\mathfrak{t}$, and write $\mathfrak{s}\unrhd\mathfrak{t}$ if $\mathrm{shape}(\mathfrak{s}_{\downarrow k})\unrhd \mathrm{shape}(\mathfrak{t}_{\downarrow k})$ for all $1\leq k\leq n.$ If $\mathfrak{s}\unrhd\mathfrak{t}$ and $\mathfrak{s}\neq\mathfrak{t}$, then we write $\mathfrak{s}\rhd\mathfrak{t}.$ Note that $\mathfrak{s}\unrhd\mathfrak{t}$ if and only if $\mathfrak{t}'\unrhd\mathfrak{s}'.$ We extend the dominance order to pairs of standard tableaux by defining $(\mathfrak{s}, \mathfrak{t})\unrhd (\mathfrak{u}, \mathfrak{v})$ if $\mathfrak{s}\unrhd\mathfrak{u}$ and $\mathfrak{t}\unrhd\mathfrak{v}.$ We write $(\mathfrak{s}, \mathfrak{t})\rhd (\mathfrak{u}, \mathfrak{v})$ if $(\mathfrak{s}, \mathfrak{t})\unrhd (\mathfrak{u}, \mathfrak{v})$ and $(\mathfrak{s}, \mathfrak{t})\neq (\mathfrak{u}, \mathfrak{v}).$

For each $\bm{\lambda}\in \mathcal{P}_{r,n},$ let $\mathfrak{t}_{\bm{\lambda}}=(\mathfrak{t}^{\bm{\lambda}'})';$ that is, $\mathfrak{t}_{\bm{\lambda}}$ is the standard $\bm{\lambda}$-tableau with the numbers $1,2,\ldots,n$ entered in order first down the columns of $\mathfrak{t}_{\bm{\lambda}}^{(r)},$ and then the columns of $\mathfrak{t}_{\bm{\lambda}}^{(r-1)}$ and so on. If $\mathfrak{t}\in \mathrm{Std}(\bm{\lambda}),$ we define two elements $d(\mathfrak{t})$ and $d'(\mathfrak{t})$ in $\mathfrak{S}_{n}$ by $\mathfrak{t}=\mathfrak{t}^{\bm{\lambda}}d(\mathfrak{t})$ and $\mathfrak{t}=\mathfrak{t}_{\bm{\lambda}}d'(\mathfrak{t}).$ Conjugating either of the two equations shows that $d'(\mathfrak{t})=d(\mathfrak{t}').$ Let $w_{\bm{\lambda}}=d(\mathfrak{t}_{\bm{\lambda}}).$ In particular, we have $w_{\bm{\lambda}}=w_{\bm{\lambda}'}^{-1}.$ Moreover, it is easy to see that $w_{\bm{\lambda}}=d(\mathfrak{t})d'(\mathfrak{t})^{-1}$ and $l(w_{\bm{\lambda}})=l(d(\mathfrak{t}))+l(d'(\mathfrak{t}))$ for all $\mathfrak{t}\in \mathrm{Std}(\bm{\lambda}).$

Recall that there is a unique anti-automorphism $\ast$ on $\text{Y}_{r,n}$ such that $g_{i}^{\ast}=g_{i}$ for $1\leq i\leq n-1$ and $t_{j}^{\ast}=t_{j}$ for $1\leq j\leq n.$ Given a right $\text{Y}_{r,n}$-module $M,$ we define its contragredient dual $M^{\circledast}$ to be the dual module $\mathrm{Hom}_{\mathcal{R}}(M, \mathcal{R})$ equipped with the right $\text{Y}_{r,n}$-action $(\varphi h)(m)=\varphi(m h^{\ast})$ for all $\varphi\in M^{\circledast},$ $h\in \text{Y}_{r,n}$ and $m\in M.$ A module $M$ is self-dual if $M\cong M^{\circledast}.$ Equivalently, $M$ is self-dual if and only if $M$ possesses a non-degenerate associative bilinear form $\langle \cdot, \cdot\rangle$, where $\langle \cdot, \cdot\rangle$ is associative if $\langle xh , y\rangle=\langle x, yh^{\ast}\rangle$ for all $x,y\in M$ and $h\in \text{Y}_{r,n}.$

Recall that the following Jucys-Murphy elements $J_{i}$ $(1\leq i\leq n)$ in $\text{Y}_{r,n}$ have been introduced in [ChPA1] by induction \begin{equation}\label{JM-elements}
J_{1} :=1\text{ and } J_{i+1} :=g_{i}J_ig_i\quad\mbox{for}~i=1,\ldots,n-1.
\end{equation}
If $\bm{\lambda}\in \mathcal{P}_{r,n}$ and $\mathfrak{t}\in \mathrm{Std}(\bm{\lambda})$, for $1\leq k\leq n,$ we define the content of $k$ in $\mathfrak{t}$ as the element $\text{c}_{\mathfrak{t}}(k) :=q^{2(j-i)}$ if $k$ appears in row $i$ and column $j$ of some component $\mathfrak{t}^{(s)}$ of $\mathfrak{t}.$ The following proposition is proved in [ER].

\begin{proposition}\label{bilinearform-proposi5-11} {\rm (See [ER, Proposition 3].)}
Suppose that $\bm{\lambda}\in \mathcal{P}_{r,n}$ and $\mathfrak{s}, \mathfrak{t}$ are two standard $\bm{\lambda}$-tableaux. For each $1\leq k\leq n,$ there exist $r_{\mathfrak{u}\mathfrak{v}}\in \mathcal{R}$ such that
\begin{equation}\label{JM-ele-5-2}
m_{\mathfrak{s}\mathfrak{t}}J_{k}=\mathrm{res}_{\mathfrak{t}}(k)m_{\mathfrak{s}\mathfrak{t}}+\sum_{(\mathfrak{u}, \mathfrak{v})}r_{\mathfrak{u}\mathfrak{v}}m_{\mathfrak{u}\mathfrak{v}},
\end{equation}
where the sum is over the pair $(\mathfrak{u}, \mathfrak{v})\in \mathrm{Std}^{2}(\bm{\mu})=\mathrm{Std}(\bm{\mu})\times \mathrm{Std}(\bm{\mu})$ such that $(\mathfrak{u}, \mathfrak{v})\rhd (\mathfrak{s}, \mathfrak{t})$ and $\alpha(\bm{\mu})=\alpha(\bm{\lambda})$ with $\bm{\mu}\in \mathcal{P}_{r,n}.$ Moreover, we have
\begin{equation}\label{JM-ele-5-3}
m_{\mathfrak{s}\mathfrak{t}}\hspace{0.3mm}t_k=\zeta_{\emph{p}_{\mathfrak{t}}(k)}m_{\mathfrak{s}\mathfrak{t}}.
\end{equation}
\end{proposition}

\begin{remark}
There are two ways to define the dominance order on $\mathrm{Std}^{2}(\mathcal{P}_{r,n})=\{(\mathfrak{s}, \mathfrak{t})\:|\:\mathfrak{s}, \mathfrak{t}\in \mathrm{Std}(\bm{\lambda})\mathrm{~for~some~}\bm{\lambda}\in \mathcal{P}_{r,n}\}.$ If $(\mathfrak{s}, \mathfrak{t})\in \mathrm{Std}^{2}(\bm{\lambda})$ and $(\mathfrak{u}, \mathfrak{v})\in \mathrm{Std}^{2}(\bm{\mu})$, then we define $$(\mathfrak{s}, \mathfrak{t})\Gedom (\mathfrak{u}, \mathfrak{v}) ~\mathrm{if}~ \bm{\lambda}\rhd\bm{\mu}, \mathrm{or}~\bm{\lambda}=\bm{\mu}~\mathrm{and}~\mathfrak{s}\unrhd\mathfrak{u}~\mathrm{and} ~\mathfrak{t}\unrhd\mathfrak{v}.$$
By definition, $(\mathfrak{s}, \mathfrak{t})\unrhd(\mathfrak{u}, \mathfrak{v})$ implies that $(\mathfrak{s}, \mathfrak{t})\Gedom(\mathfrak{u}, \mathfrak{v}),$ but the inverse is false in general. In fact, it is proved in [ER] that the equality \eqref{JM-ele-5-2} holds under the dominance order $\Gedom.$ But it is easy to see that the equality \eqref{JM-ele-5-2} still holds under the stronger dominance order $\unrhd.$ Besides, the proof of Proposition \ref{bilinearform-proposi5-11}  essentially reduces to the case of $r=1,$ from which we can easily get the second condition $\alpha(\bm{\mu})=\alpha(\bm{\lambda})$ in the summation of \eqref{JM-ele-5-2}. These facts are crucial to the following arguments.
\end{remark}

Let $\mathbb{K}=\mathbb{Q}(q, \zeta).$ We shall first consider the split semisimple $\mathbb{K}$-algebra $\text{Y}_{r,n}^{\mathbb{K}} :=\text{Y}_{r,n}^{\mathcal{Z}}\otimes_{\mathcal{Z}}\mathbb{K}.$ In particular, we can apply all the results in [Ma3, Section 3].

We shall follow the arguments of [Ma3, Section 3] to construct a ``seminormal" basis of $\text{Y}_{r,n}^{\mathbb{K}}$. For $1\leq k\leq n,$ we define the following two sets:
\[\mathcal{C}(k) :=\{\text{c}_{\mathfrak{t}}(k)\:|\:\mathfrak{t}\in \text{Std}(\bm{\lambda})\text{ for some }\bm{\lambda}\in \mathcal{P}_{r,n}\},\]
and
\[\overline{\mathcal{C}(k)} :=\{\zeta_{\text{p}_{\mathfrak{t}}(k)}\:|\:\mathfrak{t}\in \text{Std}(\bm{\lambda})\text{ for some }\bm{\lambda}\in \mathcal{P}_{r,n}\}.\]
\begin{definition}\label{definition-idempotent}
Suppose that $\bm{\lambda}\in \mathcal{P}_{r,n}$ and that $\mathfrak{s}, \mathfrak{t}\in \text{Std}(\bm{\lambda}).$

(i) Let
\begin{equation}\label{idempotent-ET}
F_{\mathfrak{t}}=\prod_{k=1}^{n}\bigg(\prod_{\substack{c\in \mathcal{C}(k)\\c\neq \text{c}_{\mathfrak{t}}(k)}}\frac{J_{k}-c}{\text{c}_{\mathfrak{t}}(k)-c}\cdot \prod_{\substack{\bar{c}\in \overline{\mathcal{C}(k)}\\\bar{c}\neq \zeta_{\text{p}_{\mathfrak{t}}(k)}}}\frac{t_{k}-\bar{c}}{\zeta_{\text{p}_{\mathfrak{t}}(k)}-\bar{c}}
\bigg).
\end{equation}

(ii) We set  $f_{\mathfrak{s}\mathfrak{t}} :=F_{\mathfrak{s}}\hspace{0.3mm}m_{\mathfrak{s}\mathfrak{t}}\hspace{0.3mm}F_{\mathfrak{t}}$ and $g_{\mathfrak{s}\mathfrak{t}} :=F_{\mathfrak{s}}\hspace{0.3mm}n_{\mathfrak{s}\mathfrak{t}}\hspace{0.3mm}F_{\mathfrak{t}}$
\end{definition}

By Proposition \ref{bilinearform-proposi5-11}, we can now apply the general theory developed in [Ma3, Section 3] to get the following results.

\begin{proposition}\label{idempotent-property-5-14}
$(\emph{i})$ Suppose that $\mathfrak{s}, \mathfrak{t}\in \emph{Std}(\bm{\lambda})\text{ for some }\bm{\lambda}\in \mathcal{P}_{r,n}$. In $\mathrm{Y}_{r,n}^{\mathbb{K}}$ we have
\begin{align}
m_{\mathfrak{s}\mathfrak{t}}=f_{\mathfrak{s}\mathfrak{t}}+\sum_{(\mathfrak{u}, \mathfrak{v})}r_{\mathfrak{u}\mathfrak{v}}f_{\mathfrak{u}\mathfrak{v}},\label{triangular-relation1}
\end{align}
where $r_{\mathfrak{u}\mathfrak{v}}\in \mathbb{K}$ and the sum is over the pair $(\mathfrak{u}, \mathfrak{v})\in \mathrm{Std}^{2}(\bm{\mu})$ $(\bm{\mu}\in \mathcal{P}_{r,n})$ such that $r_{\mathfrak{u}\mathfrak{v}}\neq 0$ only if $(\mathfrak{u}, \mathfrak{v})\rhd (\mathfrak{s}, \mathfrak{t})$ and $\alpha(\bm{\mu})=\alpha(\bm{\lambda});$
\begin{align}
n_{\mathfrak{s}\mathfrak{t}}=g_{\mathfrak{s}\mathfrak{t}}+\sum_{(\mathfrak{u}, \mathfrak{v})}s_{\mathfrak{u}\mathfrak{v}}g_{\mathfrak{u}\mathfrak{v}},\label{triangular-relation2}
\end{align}
where $s_{\mathfrak{u}\mathfrak{v}}\in \mathbb{K}$ and the sum is over the pair $(\mathfrak{u}, \mathfrak{v})\in \mathrm{Std}^{2}(\bm{\nu})$ $(\bm{\nu}\in \mathcal{P}_{r,n})$ such that $s_{\mathfrak{u}\mathfrak{v}}\neq 0$ only if $(\mathfrak{u}, \mathfrak{v})\rhd (\mathfrak{s}, \mathfrak{t})$ and $\alpha(\bm{\nu})=\alpha(\bm{\lambda}).$

$(\emph{ii})$ The set $\{f_{\mathfrak{s}\mathfrak{t}}\:|\:\mathfrak{s}, \mathfrak{t}\in \emph{Std}(\bm{\lambda})\text{ for some }\bm{\lambda}\in \mathcal{P}_{r,n}\}$ is a $\mathbb{K}$-basis of $\mathrm{Y}_{r,n}^{\mathbb{K}}.$

$(\emph{iii})$ For $\bm{\lambda}, \bm{\mu}\in \mathcal{P}_{r,n}$ and $\mathfrak{s}, \mathfrak{t}\in \emph{Std}(\bm{\lambda}),$ $\mathfrak{u}, \mathfrak{v}\in \emph{Std}(\bm{\mu}),$ we have
\begin{equation}\label{result-1}
f_{\mathfrak{s}\mathfrak{t}}\hspace{0.3mm}J_k=\emph{c}_{\mathfrak{t}}(k)f_{\mathfrak{s}\mathfrak{t}},\hspace{1cm}
f_{\mathfrak{s}\mathfrak{t}}\hspace{0.3mm}t_k=\zeta_{\emph{p}_{\mathfrak{t}}(k)}f_{\mathfrak{s}\mathfrak{t}},\hspace{1cm} f_{\mathfrak{s}\mathfrak{t}}\hspace{0.3mm}F_{\mathfrak{u}}=\delta_{\mathfrak{t}, \mathfrak{u}}f_{\mathfrak{s}\mathfrak{u}},
\end{equation}
and moreover, there exists a scalar $0\neq \gamma_{\mathfrak{t}}\in \mathbb{K}$ such that
\begin{align}\label{result-2}
f_{\mathfrak{s}\mathfrak{t}}\hspace{0.3mm}f_{\mathfrak{u}\mathfrak{v}}=
\begin{cases}
\gamma_{\mathfrak{t}}\hspace{0.3mm}f_{\mathfrak{s}\mathfrak{v}} & \text{if } \bm{\lambda}=\bm{\mu}\text{ and }\mathfrak{t}=\mathfrak{u};
\\
\hspace{0.25cm}0 & \text{otherwise.}
\end{cases}
\end{align}
In particular, $\gamma_{\mathfrak{t}}$ depends only on $\mathfrak{t}$ and the set $\{f_{\mathfrak{s}\mathfrak{t}}\:|\:\mathfrak{s}, \mathfrak{t}\in \emph{Std}(\bm{\lambda})\text{ and }\bm{\lambda}\in \mathcal{P}_{r,n}\}$ is a cellular basis of $\mathrm{Y}_{r,n}^{\mathbb{K}}.$

$(\emph{iv})$ For $\bm{\lambda}\in \mathcal{P}_{r,n}$ and $\mathfrak{t}\in \emph{Std}(\bm{\lambda}),$ we have $F_{\mathfrak{t}}=\frac{1}{\gamma_{\mathfrak{t}}}\hspace{0.3mm}f_{\mathfrak{t}\mathfrak{t}}.$ Moreover, these elements $\{F_{\mathfrak{t}}\:|\:\mathfrak{t}\in \emph{Std}(\bm{\lambda})\text{ for some }\bm{\lambda}\in \mathcal{P}_{r,n}\}\}$ give a complete set of pairwise orthogonal primitive idempotents for $\mathrm{Y}_{r,n}^{\mathbb{K}}.$

$(\emph{v})$ For $\bm{\lambda}\in \mathcal{P}_{r,n}$ and $\mathfrak{t}\in \emph{Std}(\bm{\lambda}),$ we have
\begin{equation}\label{result-3}
F_{\mathfrak{t}}\hspace{0.3mm}J_k=J_k\hspace{0.3mm}F_{\mathfrak{t}}=\emph{c}_{\mathfrak{t}}(k)F_{\mathfrak{t}},\hspace{1cm}
F_{\mathfrak{t}}\hspace{0.3mm}t_k=t_k\hspace{0.3mm}F_{\mathfrak{t}}=\zeta_{\emph{p}_{\mathfrak{t}}(k)}F_{\mathfrak{t}}.
\end{equation}

$(\emph{vi})$ The Jucys-Murphy elements $J_1,\ldots,J_n,$ $t_1,\ldots,t_n$ generate a maximal commutative subalgebra of $\mathrm{Y}_{r,n}^{\mathbb{K}}.$
\end{proposition}

From the definitions, we see that for $\bm{\lambda}\in \mathcal{P}_{r,n}$ and $\mathfrak{t}\in \text{Std}(\bm{\lambda}),$ we have  $(\text{c}_{\mathfrak{t}}(k))'=\text{c}_{\mathfrak{t}'}(k)$ and $\zeta_{\text{p}_{\mathfrak{t}}(k)}'=\zeta_{\text{p}_{\mathfrak{t}'}(k)}.$ This implies the following lemma.
\begin{lemma}\label{idempotent-lemma-5-15}
For $\bm{\lambda}\in \mathcal{P}_{r,n}$ and $\mathfrak{t}\in \mathrm{Std}(\bm{\lambda}),$ we have $F_{\mathfrak{t}}'=F_{\mathfrak{t}'},$ and hence $g_{\mathfrak{s}\mathfrak{t}}=f_{\mathfrak{s}\mathfrak{t}}'$ in $\mathrm{Y}_{r,n}^{\mathbb{K}}.$
\end{lemma}

By Proposition \ref{idempotent-property-5-14}(iv) and Lemma \ref{idempotent-lemma-5-15}, we can easily get the following result.
\begin{lemma}\label{idempotent-lemma-5-16}
Suppose that $\bm{\lambda}, \bm{\mu}\in \mathcal{P}_{r,n}$ and $\mathfrak{s}, \mathfrak{t}\in \mathrm{Std}(\bm{\lambda}),$ $\mathfrak{u}, \mathfrak{v}\in \mathrm{Std}(\bm{\mu}).$ Then in $\mathrm{Y}_{r,n}^{\mathbb{K}},$ we have $f_{\mathfrak{s}\mathfrak{t}}g_{\mathfrak{u}\mathfrak{v}}=0$ unless $\mathfrak{t}=\mathfrak{u}'.$
\end{lemma}

By Proposition \ref{idempotent-property-5-14}(i) and Lemma \ref{idempotent-lemma-5-16}, we can easily get the following lemma.

\begin{lemma}\label{idempotent-lemma-5-17}
Let $\bm{\lambda}, \bm{\mu}\in \mathcal{P}_{r,n}$. Suppose that $\mathfrak{s}$ and $\mathfrak{t}$ are standard $\bm{\lambda}$-tableaux and that $\mathfrak{u}$ and $\mathfrak{v}$ are standard $\bm{\mu}$-tableaux. If $m_{\mathfrak{s}\mathfrak{t}}n_{\mathfrak{u}\mathfrak{v}}\neq 0,$ then $\mathfrak{u}'\unrhd \mathfrak{t}$ and $\alpha(\bm{\mu}')=\alpha(\bm{\lambda}).$
\end{lemma}

Recall that $\{t_{1}^{k_1}\cdots t_{n}^{k_n}g_{w}\:|\:0\leq k_1,\ldots,k_{n}\leq r-1~\mathrm{and}~w\in \mathfrak{S}_{n}\}$ is an $\mathcal{R}$-basis of $\text{Y}_{r,n}.$ We can define an $\mathcal{R}$-linear map $\tau: \text{Y}_{r,n}\rightarrow \mathcal{R}$ by
\begin{align}\label{trace-form}
\tau(t_{1}^{k_1}\cdots t_{n}^{k_n}g_{w})=\begin{cases}1& \hbox {if } k_{1}\equiv k_{2}\equiv\cdots\equiv 0~(\mathrm{mod}~r)~\mathrm{and}~w=1; \\0& \hbox {otherwise}.\end{cases}
\end{align}
This map $\tau$ was introduced in [ChPA1, Proposition 10] and was shown to be a trace form; that is, $\tau(ab)=\tau(ba)$ for all $a,b\in \text{Y}_{r,n}.$ Moreover, we have $$\tau(t_{1}^{k_1}\cdots t_{n}^{k_n}g_{w}g_{w'}t_{1}^{l_1}\cdots t_{n}^{l_n})=\begin{cases}1& \hbox {if } w^{-1}=w'~\mathrm{and}~k_{i}+l_{i}\equiv 0~(\mathrm{mod}~r)~\mathrm{for}~1\leq i\leq n; \\0& \hbox {otherwise}.\end{cases}$$ In particular, we get that $\tau(h^{\ast})=\tau(h)$ for all $h\in \text{Y}_{r,n}.$

We now define a symmetric associative bilinear form $\langle \cdot, \cdot\rangle$ on $\text{Y}_{r,n}$ by $\langle h_{1}, h_{2}\rangle=\tau(h_{1}h_{2}^{\ast}).$ We then have the following crucial result.

\begin{theorem}\label{idempotent-theorem-5-18}
Suppose that $\bm{\lambda}=(\lambda^{(1)},\ldots,\lambda^{(r)})\in \mathcal{P}_{r,n}$ and we choose all $1\leq i_{1}< i_{2}<\cdots < i_{p}\leq r$ such that $\lambda^{(i_{1})},$ $\lambda^{(i_{2})},\ldots,$ $\lambda^{(i_{p})}$ are the nonempty components of $[\bm{\lambda}]$. Define $c_{k} :=|\lambda^{(i_{k})}|$ for $1\leq k\leq p.$ Let $\bm{\mu}\in \mathcal{P}_{r,n}$. Suppose that $\mathfrak{s}, \mathfrak{t}$ are two standard $\bm{\lambda}$-tableaux and that $\mathfrak{u}, \mathfrak{v}$ are two standard $\bm{\mu}$-tableaux. Then we have
\begin{align}\label{bilinear-form-5-theorem}
\langle m_{\mathfrak{s}\mathfrak{t}}, n_{\mathfrak{u}\mathfrak{v}}\rangle=\begin{cases}\frac{r^{p}(\zeta_{i_1}\cdots \zeta_{i_{p}})^{-2}}{r^{\sum_{k=1}^{p}\binom{c_k}{2}}}   & \hbox {if } (\mathfrak{u}', \mathfrak{v}')=(\mathfrak{s}, \mathfrak{t}); \\\qquad 0& \hbox {if } (\mathfrak{u}', \mathfrak{v}')\ntrianglerighteq(\mathfrak{s}, \mathfrak{t}).\end{cases}
\end{align}
\end{theorem}
\begin{proof}
Suppose first that $\langle m_{\mathfrak{s}\mathfrak{t}}, n_{\mathfrak{u}\mathfrak{v}}\rangle\neq 0.$ By definition, $\langle m_{\mathfrak{s}\mathfrak{t}}, n_{\mathfrak{u}\mathfrak{v}}\rangle=\tau(m_{\mathfrak{s}\mathfrak{t}}n_{\mathfrak{v}\mathfrak{u}}),$ so $m_{\mathfrak{s}\mathfrak{t}}n_{\mathfrak{v}\mathfrak{u}}\neq 0;$ hence $\mathfrak{v}'\trianglerighteq \mathfrak{t}$ by Lemma \ref{idempotent-lemma-5-17}. Since $\tau$ is a trace form and $\tau(h^{\ast})=\tau(h)$ for all $h\in \text{Y}_{r,n},$ we also have $\tau(m_{\mathfrak{s}\mathfrak{t}}n_{\mathfrak{v}\mathfrak{u}})=\tau(n_{\mathfrak{v}\mathfrak{u}}m_{\mathfrak{s}\mathfrak{t}})=
\tau(m_{\mathfrak{t}\mathfrak{s}}n_{\mathfrak{u}\mathfrak{v}});$ hence $m_{\mathfrak{t}\mathfrak{s}}n_{\mathfrak{u}\mathfrak{v}}\neq 0$ and $\mathfrak{u}'\trianglerighteq \mathfrak{s}$ by Lemma \ref{idempotent-lemma-5-17}. Therefore, if $\langle m_{\mathfrak{s}\mathfrak{t}}, n_{\mathfrak{u}\mathfrak{v}}\rangle\neq 0,$ then $(\mathfrak{u}', \mathfrak{v}')\trianglerighteq(\mathfrak{s}, \mathfrak{t}).$

Now assume that $(\mathfrak{u}', \mathfrak{v}')=(\mathfrak{s}, \mathfrak{t}).$ Then $g_{w_{\bm{\lambda}}}=g_{d(\mathfrak{t})}g_{d(\mathfrak{t}')}^{\ast}=g_{d(\mathfrak{s})}g_{d(\mathfrak{s}')}^{\ast}.$ Therefore, we have
\begin{align}
\langle m_{\mathfrak{s}\mathfrak{t}}, n_{\mathfrak{s}'\mathfrak{t}'}\rangle&=
\tau(m_{\mathfrak{s}\mathfrak{t}}n_{\mathfrak{t}'\mathfrak{s}'})\notag\\&=
\tau(g_{d(\mathfrak{s})}^{\ast}m_{\bm{\lambda}}g_{d(\mathfrak{t})}g_{d(\mathfrak{t}')}^{\ast}n_{\bm{\lambda}'}g_{d(\mathfrak{s}')})\notag\\&
=\tau(g_{d(\mathfrak{s}')}g_{d(\mathfrak{s})}^{\ast}m_{\bm{\lambda}}g_{w_{\bm{\lambda}}}n_{\bm{\lambda}'})\notag\\&=
\tau(g_{w_{\bm{\lambda}}}^{\ast}m_{\bm{\lambda}}g_{w_{\bm{\lambda}}}n_{\bm{\lambda}'})\notag\\
&=\tau(m_{\bm{\lambda}}g_{w_{\bm{\lambda}}}n_{\bm{\lambda}'}g_{w_{\bm{\lambda}}^{-1}})\notag\\
&=\tau(u_{\bm{\lambda}}E_{A_{\bm{\lambda}}}x_{\bm{\lambda}}g_{w_{\bm{\lambda}}}U_{\bm{\lambda}'}'y_{\bm{\lambda}'}g_{w_{\bm{\lambda}'}})\notag\\
&=\tau(u_{\bm{\lambda}}E_{A_{\bm{\lambda}}}g_{w_{\bm{\lambda}}}
u'_{\bm{\lambda}'}E_{A_{\bm{\lambda}'}}y_{\bm{\lambda}'}g_{w_{\bm{\lambda}'}}x_{\bm{\lambda}})\label{important-identity}.
\end{align}

By definition, we have
\begin{equation}\label{w-lambda}
w_{\bm{\lambda}}=\left(
  \begin{array}{cccc}
    1 & 2 & \cdots & |\lambda^{(i_1)}|\\
    j_1 & j_2 &\cdots & j_{|\lambda^{(i_1)}|}\\
  \end{array}
\bigg|
  \begin{array}{cc}
    \cdots & \cdots \\
    \cdots &\cdots \\
  \end{array}
\bigg|
\begin{array}{cccc}
    n-|\lambda^{(i_{p})}|+1  & \cdots & n\\
    k_1  &\cdots & k_{|\lambda^{(i_{p})}|}\\
  \end{array}
\right),
\end{equation}
where $\{j_1,j_2,\ldots,j_{|\lambda^{(i_1)}|}\}=\{n-|\lambda^{(i_1)}|+1,\ldots,n\}$ and $\{k_1,\ldots,k_{|\lambda^{(i_{p})}|}\}=\{1,\ldots,|\lambda^{(i_{p})}|\}.$

By \eqref{w-lambda}, we have
\begin{equation}\label{w-lambda-inverse}
w_{\bm{\lambda}}^{-1}=\left(
  \begin{array}{cccc}
    1 & 2 & \cdots & |\lambda^{(i_{p})}|\\
    k_1 & k_2 &\cdots & k_{|\lambda^{(i_{p})}|}\\
  \end{array}
\bigg|
  \begin{array}{cc}
    \cdots & \cdots \\
    \cdots &\cdots \\
  \end{array}
\bigg|
\begin{array}{cccc}
    n-|\lambda^{(i_{1})}|+1  & \cdots & n\\
    j_1  &\cdots & k_{|\lambda^{(i_{1})}|}\\
  \end{array}
\right),
\end{equation}
where $\{k_1,k_2,\ldots,k_{|\lambda^{(i_{p})}|}\}=\{n-|\lambda^{(i_{p})}|+1,\ldots,n\}$ and $\{j_1,\ldots,j_{|\lambda^{(i_{1})}|}\}=\{1,\ldots,|\lambda^{(i_{1})}|\}.$

Define $A_{\bm{\lambda}}=\{I_{1},I_{2},\ldots,I_{p}\}$ as in Definition \ref{definition2-3-ulambda}. By assumption, we have $\bm{\lambda}'=((\lambda^{(r)})',\ldots,(\lambda^{(1)})')\in \mathcal{P}_{r,n}$ with $(\lambda^{(i_{p})})',$ $(\lambda^{(i_{p-1})})',\ldots,$ $(\lambda^{(i_{1})})'$ being all the nonempty components of $[\bm{\lambda}']$. Suppose that $A_{\bm{\lambda}'}=\{I_{p}',\ldots,I_{1}'\}$, where $I_{p}'=\{1,2,\ldots,|\lambda^{(i_{p})}|\},\ldots,$ and $I_{1}'=\{n-|\lambda^{(i_{1})}|+1,\ldots,n\}.$ By definition \ref{definition2-4-ulambda}, we have
\begin{align}\label{young-module-5-15}
u'_{\bm{\lambda}'}E_{A_{\bm{\lambda}'}}=\prod_{l\neq r+1-(r+1-i_{p})}(t_{|\lambda^{(i_{p})}|}-\zeta_{l})\cdots \prod_{l\neq r+1-(r+1-i_{1})}(t_{n}-\zeta_{l})\cdot E_{I_{p}'}E_{I_{p-1}'}\cdots E_{I_{1}'}.
\end{align}

By \eqref{young-module-5-15}, Lemma \ref{2-1-lemmaaa}, together with the fact that $g_{w}t_{k}=t_{(k)w^{-1}}g_{w},$ we get that
\begin{equation}\label{young-module-5-16}
g_{w_{\bm{\lambda}}}u'_{\bm{\lambda}'}E_{A_{\bm{\lambda}'}}=u_{\bm{\lambda}}E_{A_{\bm{\lambda}}}g_{w_{\bm{\lambda}}}.
\end{equation}

By \eqref{important-identity} and \eqref{young-module-5-16}, we have
\begin{equation}\label{young-module-5-17}
\langle m_{\mathfrak{s}\mathfrak{t}}, n_{\mathfrak{s}'\mathfrak{t}'}\rangle=\tau(E_{A_{\bm{\lambda}}}u_{\bm{\lambda}}U_{\bm{\lambda}}g_{w_{\bm{\lambda}}}y_{\bm{\lambda}'}g_{w_{\bm{\lambda}'}}x_{\bm{\lambda}}).
\end{equation}

By [ER, Lemma 10(49)], we have
\begin{equation}\label{young-module-5-18}
t_{i}U_{\bm{\lambda}}=\zeta_{i_{k}}U_{\bm{\lambda}} \text{ if }i\in I_{k} \text{ for some }1\leq k\leq p.
\end{equation}
Recall that $S=\{\zeta_{1},\zeta_{2},\ldots,\zeta_{r}\}.$ We also have, for any fixed $r$-th root of unity $\xi$,
\begin{equation}\label{young-module-5-19}\prod_{\xi\neq \alpha\in S}(\xi-\alpha)=r\xi^{-1}.
\end{equation}
By \eqref{young-module-5-18} and \eqref{young-module-5-19}, we get that
\begin{equation}\label{young-module-5-20}
u_{\bm{\lambda}}U_{\bm{\lambda}}=r^{p}(\zeta_{i_1}\cdots \zeta_{i_{p}})^{-1}U_{\bm{\lambda}}.
\end{equation}

Combining \eqref{young-module-5-17} and \eqref{young-module-5-20}, we get that
\begin{equation}\label{young-module-5-21}
\langle m_{\mathfrak{s}\mathfrak{t}}, n_{\mathfrak{s}'\mathfrak{t}'}\rangle=r^{p}(\zeta_{i_1}\cdots \zeta_{i_{p}})^{-1}\tau(E_{A_{\bm{\lambda}}}u_{\bm{\lambda}}g_{w_{\bm{\lambda}}}y_{\bm{\lambda}'}g_{w_{\bm{\lambda}'}}x_{\bm{\lambda}}).
\end{equation}

By \eqref{trace-form} and $\zeta_{1}\cdots \zeta_{r}=(-1)^{r-1}$, we have
\begin{align}\label{young-module-5-22}
\tau(E_{A_{\bm{\lambda}}}u_{\bm{\lambda}})&=\prod_{k=1}^{p}\frac{1}{r^{|\lambda^{(i_{k})}|\choose 2 }}\cdot \prod_{k=1}^{p}\Big(\prod_{l\neq i_{k}}(-1)^{r-1}\zeta_{l}\Big)\notag\\
&=(\zeta_{i_1}\cdots \zeta_{i_{p}})^{-1}\frac{1}{r^{\sum_{k=1}^{p}\binom{c_k}{2}}}.
\end{align}

By \eqref{young-module-5-21} and \eqref{young-module-5-22}, we have
\begin{equation}\label{young-module-5-23}
\langle m_{\mathfrak{s}\mathfrak{t}}, n_{\mathfrak{s}'\mathfrak{t}'}\rangle=\frac{r^{p}(\zeta_{i_1}\cdots \zeta_{i_{p}})^{-2}}{r^{\sum_{k=1}^{p}\binom{c_k}{2}}}\tau(g_{w_{\bm{\lambda}}}y_{\bm{\lambda}'}g_{w_{\bm{\lambda}'}}x_{\bm{\lambda}}).
\end{equation}
Since $\mathfrak{S}_{\bm{\lambda}}\cap {}^{w_{\bm{\lambda}}}\!\hspace{0.8mm}\mathfrak{S}_{\bm{\lambda}'}=\{1\}$ and $w_{\bm{\lambda}}=w_{\bm{\lambda}'}^{-1}$ is a distinguished $(\mathfrak{S}_{\bm{\lambda}}, \mathfrak{S}_{\bm{\lambda}'})$-double coset representative, we have
\begin{align*}
g_{w_{\bm{\lambda}}}y_{\bm{\lambda}'}g_{w_{\bm{\lambda}'}}x_{\bm{\lambda}}&=\sum_{\substack{u\in \mathfrak{S}_{\bm{\lambda}'}\\v\in \mathfrak{S}_{\bm{\lambda}}}}g_{w_{\bm{\lambda}}}\cdot(-q)^{-l(u)}g_{u}\cdot g_{w_{\bm{\lambda}}^{-1}}\cdot q^{l(v)}g_{v}
\\&
=\sum_{\substack{u\in \mathfrak{S}_{\bm{\lambda}'}\\v\in \mathfrak{S}_{\bm{\lambda}}}}(-1)^{l(u)}q^{l(v)-l(u)}g_{w_{\bm{\lambda}}}g_{uw_{\bm{\lambda}}^{-1}v}.
\end{align*}
Thus, we get $\tau(g_{w_{\bm{\lambda}}}y_{\bm{\lambda}'}g_{w_{\bm{\lambda}'}}x_{\bm{\lambda}})=1.$ Therefore, by \eqref{young-module-5-23}, we have
\begin{equation*}
\langle m_{\mathfrak{s}\mathfrak{t}}, n_{\mathfrak{s}'\mathfrak{t}'}\rangle=\frac{r^{p}(\zeta_{i_1}\cdots \zeta_{i_{p}})^{-2}}{r^{\sum_{k=1}^{p}\binom{c_k}{2}}}.
\end{equation*}
We have proved the theorem.
\end{proof}

\begin{remark}
Theorem \ref{idempotent-theorem-5-18} implies that $\tau$ is non-degenerate. Consequently, $\text{Y}_{r,n}$ is a symmetric algebra.
\end{remark}
The next corollary can be proved in exactly the same way as in [Ma2, Corollary 5.7] using Theorem \ref{idempotent-theorem-5-18}, which justify the term dual Specht module.

\begin{corollary}\label{idempotent-corollary-5-20}
Suppose that $\bm{\lambda}\in \mathcal{P}_{r,n}.$ Then $S^{\bm{\lambda}'}\cong S_{\bm{\lambda}}^{\circledast}.$
\end{corollary}

If $\mathrm{S}=(S^{(1)},\ldots, S^{(r)})$ is a $\bm{\lambda}$-tableau of type $\bm{\mu}$, we define the conjugate of $\mathrm{S}$ by $\mathrm{S}'=((S^{(r)})',\ldots, (S^{(1)})')$ which is a $\bm{\lambda}'$-tableau of type $\bm{\mu},$ where $(S^{(j)})'$ is the tableau obtained by interchanging the rows and columns of $S^{(j)}$ for each $j.$ A $\bm{\lambda}$-tableau $\mathrm{S}$ is called column semistandard if $\mathrm{S}'$ is semistandard. For $\bm{\lambda}\in \mathcal{P}_{r,n}$ and $\bm{\mu}\in \mathcal{C}_{r,n},$ let $\mathcal{T}^{\mathrm{cs}}(\bm{\lambda}, \bm{\mu})=\{\mathrm{S}\:|\:\mathrm{S}'\in \mathcal{T}_{0}^{+}(\bm{\lambda}', \bm{\mu})\}.$

The proof of the next lemma is in exactly the same way as that of [Ma2, Lemma 5.8] by making use of Lemma \ref{idempotent-lemma-5-17}. We skip the details.
\begin{lemma}\label{idempotent-lemma-5-21}
Suppose that $\bm{\mu}\in \mathcal{C}_{r,n}, \bm{\lambda}\in \mathcal{P}_{r,n}$ and that $m_{\bm{\mu}}n_{\mathfrak{u}\mathfrak{v}}\neq 0$ or $n_{\bm{\mu}}m_{\mathfrak{u}\mathfrak{v}}\neq 0$ for some standard $\bm{\lambda}$-tableaux $\mathfrak{u}$ and $\mathfrak{v}.$ Then $\bm{\mu}(\mathfrak{u})$ is column semistandard and $\alpha(\bm{\lambda}')=\alpha(\bm{\mu});$ that is, $\bm{\mu}(\mathfrak{u})\in \mathcal{T}^{\mathrm{cs}}(\bm{\lambda}, \bm{\mu}).$
\end{lemma}

\begin{remark}
As mentioned in [HM1, p. 15], it is unfortunate that Mathas confused the two partial orders $\trianglerighteq$ and $\Gedom$ on $\mathrm{Std}^{2}(\mathcal{P}_{r,n})$ in [Ma2] and [Ma3]. Anyhow, we can adapt the approach in [Ma3, Section 3] to get Proposition \ref{idempotent-property-5-14} and then Lemmas \ref{idempotent-lemma-5-17} and \ref{idempotent-lemma-5-21}. We leave the details to the reader; see also [HM1, Section 2] for details.
\end{remark}

If $\mathrm{S}\in \mathcal{T}^{\mathrm{cs}}(\bm{\lambda}, \bm{\mu}),$ let $\dot{\mathrm{S}}$ be the unique standard $\bm{\lambda}$-tableau such that $\bm{\mu}(\dot{\mathrm{S}})=\mathrm{S}$ and $d(\dot{\mathrm{S}})$ is a distinguished $(\mathfrak{S}_{\bm{\lambda}}, \mathfrak{S}_{\bm{\mu}})$-double coset representative; that is, $d(\dot{\mathrm{S}})$ is the unique element of minimal length in its double coset.

\begin{proposition}\label{idempotent-proposition-5-23}
Suppose that $\bm{\mu}\in \mathcal{C}_{r,n}.$ Then $M^{\bm{\mu}}$ is free as an $\mathcal{R}$-module with basis $\{m_{\bm{\mu}}n_{\dot{\mathrm{S}}\mathfrak{t}}\:|\:\mathrm{S}\in \mathcal{T}^{\mathrm{cs}}(\bm{\lambda}, \bm{\mu})$ and $\mathfrak{t}\in \mathrm{Std}(\bm{\lambda})$ for some $\bm{\lambda}\in \mathcal{P}_{r,n}\}$ and $N^{\bm{\mu}}$ is free as an $\mathcal{R}$-module with basis $\{n_{\bm{\mu}}m_{\dot{\mathrm{S}}\mathfrak{t}}\:|\:\mathrm{S}\in \mathcal{T}^{\mathrm{cs}}(\bm{\lambda}, \bm{\mu})$ and $\mathfrak{t}\in \mathrm{Std}(\bm{\lambda})$ for some $\bm{\lambda}\in \mathcal{P}_{r,n}\}.$
\end{proposition}
\begin{proof}
We only prove the claim for $M^{\bm{\mu}}.$ Recall that $\{n_{\mathfrak{s}\mathfrak{t}}\}$ is an $\mathcal{R}$-basis of $\text{Y}_{r,n},$ so $M^{\bm{\mu}}$ is spanned by the elements $m_{\bm{\mu}}n_{\mathfrak{s}\mathfrak{t}},$ where $(\mathfrak{s}, \mathfrak{t})\in \mathrm{Std}^{2}(\mathcal{P}_{r,n}).$ Furthermore, if $m_{\bm{\mu}}n_{\mathfrak{s}\mathfrak{t}}\neq 0$ then $\bm{\mu}(\mathfrak{s})$ is column semistandard and $\alpha(\bm{\lambda}')=\alpha(\bm{\mu})$ if $(\mathfrak{s}, \mathfrak{t})\in \mathrm{Std}^{2}(\bm{\lambda})$ by Lemma \ref{idempotent-lemma-5-21}. Hence, $M^{\bm{\mu}}$ is spanned by the elements $m_{\bm{\mu}}n_{\mathfrak{s}\mathfrak{t}},$ where $\bm{\mu}(\mathfrak{s})$ is column semistandard and $(\mathfrak{s}, \mathfrak{t})\in \mathrm{Std}^{2}(\bm{\lambda})$ for various $\bm{\lambda}\in \mathcal{P}_{r,n}$ with $\alpha(\bm{\lambda}')=\alpha(\bm{\mu}).$

For each such element $m_{\bm{\mu}}n_{\mathfrak{s}\mathfrak{t}},$ where $(\mathfrak{s}, \mathfrak{t})\in \mathrm{Std}^{2}(\bm{\nu})$ with $\alpha(\bm{\nu}')=\alpha(\bm{\mu}).$ Since $\bm{\mu}(\mathfrak{s})=\mathrm{S}$ is column semistandard, we choose $\dot{\mathrm{S}}\in \mathrm{Std}(\bm{\nu})$ as above and get that $\bm{\mu}(\mathfrak{t}^{\bm{\nu}}d(\mathfrak{s}))=\bm{\mu}(\mathfrak{t}^{\bm{\nu}}d(\dot{\mathrm{S}})).$ Thus, $d(\mathfrak{s})$ and $d(\dot{\mathrm{S}})$ lie in the same $(\mathfrak{S}_{\bm{\nu}}, \mathfrak{S}_{\bm{\mu}})$-double coset. By definition $d(\dot{\mathrm{S}})$ is the unique element of minimal length in its double coset, therefore we get $m_{\bm{\mu}}g_{d(\mathfrak{s})}^{\ast}n_{\bm{\nu}}=\pm q^{a}m_{\bm{\mu}}g_{d(\dot{\mathrm{S}})}^{\ast}n_{\bm{\nu}}$ for some integer $a.$ Because $M^{\bm{\mu}}$ is $\mathcal{R}$-free and the number of elements in our spanning set is exactly the rank of $M^{\bm{\mu}},$ thus we have proved the first claim. The second statement can be proved in a similar way.
\end{proof}

Combining Lemma \ref{idempotent-lemma-5-21} and Proposition \ref{idempotent-proposition-5-23}, we get the next result.

\begin{corollary}\label{idempotent-corollary-5-24}
Suppose that $\bm{\mu}\in \mathcal{C}_{r,n}, \bm{\lambda}\in \mathcal{P}_{r,n}$ and that $\mathfrak{u}$ and $\mathfrak{v}$ are two standard $\bm{\lambda}$-tableaux. Then $m_{\bm{\mu}}n_{\mathfrak{u}\mathfrak{v}}\neq 0$ if and only if $\bm{\mu}(\mathfrak{u})$ is column semistandard and $\alpha(\bm{\lambda}')=\alpha(\bm{\mu}).$ Similarly, $n_{\bm{\mu}}m_{\mathfrak{u}\mathfrak{v}}\neq 0$ if and only if $\bm{\mu}(\mathfrak{u})$ is column semistandard and $\alpha(\bm{\lambda}')=\alpha(\bm{\mu}).$
\end{corollary}

Using Proposition \ref{idempotent-proposition-5-23} we can get the following result by repeating the argument of Lemma \ref{tiltmod-lemma5-2}.

\begin{corollary}\label{idempotent-corollary-5-25}
Suppose that $\bm{\mu}\in \mathcal{C}_{r,n}.$ Then there exist filtrations $$M^{\bm{\mu}}=H^{1}\supset H^{2}\supset \cdots\supset H^{k}\supset H^{k+1}=0~and~ N^{\bm{\mu}}=H_{1}\supset H_{2}\supset \cdots\supset H_{k}\supset H_{k+1}=0$$ of $M^{\bm{\mu}}$ and $N^{\bm{\mu}},$ respectively, and $r$-partitions $\bm{\lambda}_{1},\ldots,\bm{\lambda}_{k}$ such that $\bm{\mu}'\trianglerighteq \bm{\lambda}_{i},$ $H^{i}/H^{i+1}\cong S_{\bm{\lambda}_{i}}$ and $H_{i}/H_{i+1}\cong S^{\bm{\lambda}_{i}}$ for $1\leq i\leq k.$ Moreover, for any $\bm{\lambda}\in \mathcal{P}_{r,n}$ we have $\sharp\{1\leq i\leq k\:|\:\bm{\lambda}_{i}=\bm{\lambda}\}=\sharp\mathcal{T}^{\mathrm{cs}}(\bm{\lambda}, \bm{\mu}).$
\end{corollary}

Now we can define a bilinear form $\langle\cdot, \cdot\rangle_{\bm{\mu}}$ on $M^{\bm{\mu}}$ by $$\langle m_{\mathrm{S}\mathfrak{t}},m_{\bm{\mu}}n_{\dot{\mathrm{U}}\mathfrak{v}}\rangle_{\bm{\mu}}=\langle m_{\mathrm{S}\mathfrak{t}}, n_{\dot{\mathrm{U}}\mathfrak{v}}\rangle,$$ where $m_{\mathrm{S}\mathfrak{t}}$ and $m_{\bm{\mu}}n_{\dot{\mathrm{U}}\mathfrak{v}}$ rum over the bases of Propositions \ref{Yoko-Schur-propo-3-3} and \ref{idempotent-proposition-5-23}, respectively.

The next proposition can be proved in exactly the same way as in [Ma2, Proposition 5.13]. We omit the details and leave them to the reader.

\begin{proposition}\label{idempotent-propositi-5-26}
Suppose that $\bm{\mu}\in \mathcal{C}_{r,n}.$ Then $\langle\cdot, \cdot\rangle_{\bm{\mu}}$ is a non-degenerate associative bilinear form on $M^{\bm{\mu}}$. In particular, $M^{\bm{\mu}}$ is self-dual. Similarly, $N^{\bm{\mu}}$ is self-dual.
\end{proposition}

By induction and using Propositions \ref{tiltmod-propo5-6} and \ref{tiltmod-proposi5-10} we can get the next result.

\begin{corollary}\label{idempotent-corollary-5-27}
Let $\bm{\lambda}\in \mathcal{P}_{r,n}.$ Then the Young module $Y^{\bm{\lambda}}$ and twisted Young module $Y_{\bm{\lambda}}$ are both self-dual.
\end{corollary}

\subsection{Tilting modules}

Recall that a $\mathrm{YS}_{n}^{r}$-module $T$ is a tilting module if it has both a filtration by Weyl modules $W^{\bm{\lambda}}$ ($\bm{\lambda}\in \mathcal{P}_{r,n}$) and a filtration by dual Weyl modules. Since $\mathrm{YS}_{n}^{r}$ is quasi-hereditary, by [Ri], for each $\bm{\lambda}\in \mathcal{P}_{r,n}$ there exists a unique indecomposable tilting module $T^{\bm{\lambda}}$ such that $[T^{\bm{\lambda}}: W^{\bm{\lambda}}]=1$ and $[T^{\bm{\lambda}}: W^{\bm{\mu}}]\neq 0$ only if $\bm{\lambda}\trianglerighteq\bm{\mu}.$ Moreover, any tilting module $T$ can be written as a direct sum of these $T^{\bm{\lambda}}$'s. The $T^{\bm{\lambda}}$ are the partial tilting modules of $\mathrm{YS}_{n}^{r}.$ A full tilting module for $\mathrm{YS}_{n}^{r}$ is any tilting module which contain every $T^{\bm{\lambda}}$ ($\bm{\lambda}\in \mathcal{P}_{r,n}$) as a direct summand.

For each $\bm{\nu}\in \mathcal{C}_{r,n},$ let $\theta_{\bm{\nu}}\in \mathrm{Hom}_{\text{Y}_{r,n}}(\text{Y}_{r,n}, N^{\bm{\nu}})$ be the map, which is defined by $\theta_{\bm{\nu}}(h)=n_{\bm{\nu}}h$ for all $h\in \text{Y}_{r,n}.$ We define $$E^{\bm{\nu}} :=\dot{\mathrm{F}}_{n}^{\omega}(\theta_{\bm{\nu}}\dot{\mathrm{YS}_{n}^{r}}).$$
Since $E^{\bm{\nu}},$ by definition, is the set of maps from $M_{n}^{r}$ to $N^{\bm{\nu}}$ which factor through $\theta_{\bm{\nu}},$ we get that $E^{\bm{\nu}}$ is a right $\mathrm{YS}_{n}^{r}$-module.

\begin{definition}\label{tilting-mod-5-28}
Suppose that $\bm{\lambda}\in \mathcal{P}_{r,n}$ and $\bm{\mu}, \bm{\nu}\in \mathcal{C}_{r,n}.$ For $\mathrm{S}\in \mathcal{T}^{\mathrm{cs}}(\bm{\lambda}, \bm{\nu})$ and $\mathrm{T}\in \mathcal{T}_{0}^{+}(\bm{\lambda}, \bm{\mu})$ let $\theta_{\mathrm{S}\mathrm{T}}$ be the homomorphism in $E^{\bm{\nu}}$ determined by $\theta_{\mathrm{S}\mathrm{T}}(m_{\bm{\alpha}}h)=\delta_{\bm{\alpha}\bm{\mu}}n_{\bm{\nu}}m_{\dot{\mathrm{S}}\mathrm{T}}h$ for all $h\in \text{Y}_{r,n}$ and all $\bm{\alpha}\in \mathcal{C}_{r,n}.$
\end{definition}

\begin{proposition}\label{tilting-mod-5-29}
Let $\bm{\nu}\in \mathcal{C}_{r,n}.$ Then $E^{\bm{\nu}}$ is free as an $\mathcal{R}$-module with basis $$\{\theta_{\mathrm{S}\mathrm{T}}\:|\:\mathrm{S}\in \mathcal{T}^{\mathrm{cs}}(\bm{\lambda}, \bm{\nu})~and~\mathrm{T}\in \mathcal{T}_{0}^{+}(\bm{\lambda})~for~some~\bm{\lambda}\in \mathcal{P}_{r,n}\}.$$
\end{proposition}
\begin{proof}
Let $\dot{E}^{\bm{\nu}}=\theta_{\bm{\nu}}\dot{\mathrm{YS}_{n}^{r}}.$ Then $\dot{E}^{\bm{\nu}}$ is a right $\dot{\mathrm{YS}_{n}^{r}}$-module and $E^{\bm{\nu}}=\dot{\mathrm{F}}_{n}^{\omega}(\dot{E}^{\bm{\nu}}).$ By Proposition \ref{Schur-functor-prop4-1}, $\dot{E}^{\bm{\nu}}$ is spanned by the maps $\theta_{\bm{\nu}}\varphi_{\mathrm{S}\mathrm{T}},$ where $\mathrm{S}\in \mathcal{T}_{0}^{+}(\bm{\lambda}, \bm{\sigma})$ and $\mathrm{T}\in \mathcal{T}_{0}^{+}(\bm{\lambda}, \bm{\mu})$ for various $\bm{\lambda}\in \mathcal{P}_{r,n}$ and $\bm{\sigma}, \bm{\mu}\in \dot{\mathcal{C}}_{r,n}.$ By definition, $\theta_{\bm{\nu}}\varphi_{\mathrm{S}\mathrm{T}}=0$ unless $\bm{\sigma}=\omega,$ that is, $\mathrm{S}$ is a standard $\bm{\lambda}$-tableau; so $\dot{E}^{\bm{\nu}}$ is spanned by the elements $\theta_{\bm{\nu}}\varphi_{\mathfrak{s}\mathrm{T}}$ with $\mathfrak{s}\in \mathrm{Std}(\bm{\lambda})$ and $\mathrm{T}\in \mathcal{T}_{0}^{+}(\bm{\lambda}, \bm{\mu})$ for $\bm{\lambda}\in \mathcal{P}_{r,n}$ and $\bm{\mu}\in \dot{\mathcal{C}}_{r,n}.$ Furthermore, $\theta_{\bm{\nu}}\varphi_{\mathfrak{s}\mathrm{T}}(m_{\bm{\mu}}h)=n_{\bm{\nu}}m_{\mathfrak{s}\mathrm{T}}h$ for all $h\in \text{Y}_{r,n}.$ Thus, $\theta_{\bm{\nu}}\varphi_{\mathfrak{s}\mathrm{T}}\neq 0$ if and only if $\bm{\nu}(\mathfrak{s})$ is column semistandard and $\alpha(\bm{\lambda}')=\alpha(\bm{\nu})$ by Corollary \ref{idempotent-corollary-5-24}, and in this case $\theta_{\bm{\nu}}\varphi_{\mathfrak{s}\mathrm{T}}=\pm q^{a}\theta_{\mathrm{S}\mathrm{T}}$ for some $a\in \mathbb{Z}$ and $\mathrm{S}=\bm{\nu}(\mathfrak{s}).$ Hence these elements $\{\theta_{\mathrm{S}\mathrm{T}}\:|\:\mathrm{S}\in \mathcal{T}^{\mathrm{cs}}(\bm{\lambda}, \bm{\nu})~\mathrm{and}~\mathrm{T}\in \dot{\mathcal{T}}_{0}^{+}(\bm{\lambda})~\mathrm{for~some}~\bm{\lambda}\in \mathcal{P}_{r,n}\}$ span $\dot{E}^{\bm{\nu}}.$

On the other hand, the elements $\{\theta_{\mathrm{S}\mathrm{T}}\}$ are linearly independent by Proposition \ref{idempotent-proposition-5-23}, so they are a basis of $\dot{E}^{\bm{\nu}}.$ Since the functor $\dot{\mathrm{F}}_{n}^{\omega}$ removes the $\omega$-weight space of $\dot{E}^{\bm{\nu}},$ therefore $\dot{\mathrm{F}}_{n}^{\omega}$ maps the basis $\{\theta_{\mathrm{S}\mathrm{T}}\}$ of $\dot{E}^{\bm{\nu}}$ to the elements stated in the proposition, or to zero if $\bm{\mu}=\omega.$ Hence, $\{\theta_{\mathrm{S}\mathrm{T}}\:|\:\mathrm{S}\in \mathcal{T}^{\mathrm{cs}}(\bm{\lambda}, \bm{\nu})~\mathrm{and}~\mathrm{T}\in \mathcal{T}_{0}^{+}(\bm{\lambda})~\mathrm{for~some}~\bm{\lambda}\in \mathcal{P}_{r,n}\}$ is an $\mathcal{R}$-basis of $E^{\bm{\nu}}.$
\end{proof}

The next proposition can be proved in exactly the same way as in [Ma2, Theorem 6.5] by using Proposition \ref{tilting-mod-5-29}. We skip the details.

\begin{proposition}\label{tilting-mod-5-30}
Let $\bm{\nu}\in \mathcal{C}_{r,n}.$ Then $E^{\bm{\nu}}$ admits a $\mathrm{YS}_{n}^{r}$-module filtration $E^{\bm{\nu}}=E_{1}\supset E_{2}\supset \cdots\supset E_{k}\supset E_{k+1}=0$ such that $E_{i}/E_{i+1}\cong W^{\bm{\lambda}_{i}}$ for some $\bm{\lambda}_{1},\ldots,\bm{\lambda}_{k}\in \mathcal{P}_{r,n}$ and $\bm{\nu}'\trianglerighteq \bm{\lambda}_{i}$ for all $1\leq i\leq k.$ Moreover, if $\bm{\lambda}\in \mathcal{P}_{r,n},$ then $\sharp\{1\leq i\leq k\:|\:\bm{\lambda}_{i}=\bm{\lambda}\}=\sharp\mathcal{T}^{\mathrm{cs}}(\bm{\lambda}, \bm{\nu}).$
\end{proposition}

From Proposition \ref{tilting-mod-5-30} we can easily get the next corollary.

\begin{corollary}\label{tilting-mod-5-31}
Suppose that $\bm{\lambda}, \bm{\mu}\in \mathcal{P}_{r,n}.$ Then $[E^{\bm{\lambda}}: W^{\bm{\lambda}'}]=1$ and $[E^{\bm{\lambda}}: W^{\bm{\mu}}]\neq 0$ only if $\bm{\lambda}'\trianglerighteq\bm{\mu}.$
\end{corollary}

\begin{definition}\label{tilting-mod-5-32}
Suppose that $\bm{\lambda}\in \mathcal{P}_{r,n}$ and $\bm{\mu}, \bm{\nu}\in \mathcal{C}_{r,n}.$ For $\mathrm{A}\in \mathcal{T}_{0}^{+}(\bm{\lambda}, \bm{\nu})$ and $\mathrm{B}\in \mathcal{T}^{\mathrm{cs}}(\bm{\lambda}, \bm{\mu})$ let $\theta_{\mathrm{A}\mathrm{B}}'$ be the homomorphism determined by $\theta_{\mathrm{A}\mathrm{B}}'(m_{\bm{\alpha}}h)=\delta_{\bm{\alpha}\bm{\mu}}n_{\mathrm{A}\dot{\mathrm{B}}}m_{\bm{\mu}}h$ for all $h\in \text{Y}_{r,n}$ and all $\bm{\alpha}\in \mathcal{C}_{r,n}.$
\end{definition}

\begin{proposition}\label{tilting-mod-5-33}
Let $\bm{\nu}\in \mathcal{C}_{r,n}.$ Then $E^{\bm{\nu}}$ is free as an $\mathcal{R}$-module with basis $$\{\theta_{\mathrm{A}\mathrm{B}}'\:|\:\mathrm{A}\in \mathcal{T}_{0}^{+}(\bm{\lambda}, \bm{\nu})~and~\mathrm{B}\in \mathcal{T}^{\mathrm{cs}}(\bm{\lambda}, \bm{\mu})~for~some~\bm{\lambda}\in \mathcal{P}_{r,n}~and~\bm{\mu}\in \mathcal{C}_{r,n}\}.$$
\end{proposition}
\begin{proof}
We first show that $\theta_{\mathrm{A}\mathrm{B}}'\in E^{\bm{\nu}}.$ By Corollary \ref{tiltmod-corollar5-8}, $n_{\mathrm{A}\dot{\mathrm{B}}}=n_{\bm{\nu}}x$ for some $x\in \text{Y}_{r,n}.$ Therefore, we have $$\theta_{\mathrm{A}\mathrm{B}}'(m_{\bm{\mu}}h)=n_{\mathrm{A}\dot{\mathrm{B}}}m_{\bm{\mu}}h=n_{\bm{\nu}}xm_{\bm{\mu}}h=\theta_{\bm{\nu}}(xm_{\bm{\mu}}h).$$ That is, $\theta_{\mathrm{A}\mathrm{B}}'$ factors through $\theta_{\bm{\nu}}$ so that $\theta_{\mathrm{A}\mathrm{B}}'\in E^{\bm{\nu}}$ as claimed. Moreover, the elements stated in the proposition are linearly independent by applying $\ast$ to Proposition \ref{idempotent-proposition-5-23}. Therefore, the elements $\{\theta_{\mathrm{A}\mathrm{B}}'\}$ give a basis of $E^{\bm{\nu}}$ by counting dimensions using Proposition \ref{tilting-mod-5-29}.
\end{proof}

The contragredient dual $E^{\circledast}$ of a $\mathrm{YS}_{n}^{r}$-module $E$ can be defined in exactly the same way as that of $\text{Y}_{r,n}$-modules. Again, we say that $E$ is self-dual if $E\cong E^{\circledast}.$

Using the two bases $\{\theta_{\mathrm{S}\mathrm{T}}\}$ and $\{\theta_{\mathrm{A}\mathrm{B}}'\}$ in Propositions \ref{tilting-mod-5-29} and \ref{tilting-mod-5-33}, we now define a bilinear form $\{\cdot, \cdot\}_{\bm{\nu}}$ on $E^{\bm{\nu}}$ by $$\{ \theta_{\mathrm{S}\mathrm{T}}, \theta_{\mathrm{A}\mathrm{B}}'\}_{\bm{\nu}}=\begin{cases}\langle m_{\dot{\mathrm{S}}\mathrm{T}}, n_{\mathrm{A}\dot{\mathrm{B}}}\rangle& \hbox {if } \mathrm{Type}(\mathrm{T})=\mathrm{Type}(\mathrm{B}); \\0& \hbox {otherwise}.\end{cases}$$

The next theorem can be proved in exactly the same way as in [Ma2, Theorem 6.17]. We skip the details.

\begin{theorem}\label{tilting-mod-5-34}
Suppose that $\bm{\nu}\in \mathcal{C}_{r,n}.$ Then $\{\cdot, \cdot\}_{\bm{\nu}}$ defines a non-degenerate associative bilinear form on $E^{\bm{\nu}};$ that is, $E^{\bm{\nu}}$ is self-dual.
\end{theorem}

Using Corollary \ref{tilting-mod-5-31} and Theorem \ref{tilting-mod-5-34} we can easily get the next result.

\begin{corollary}\label{tilting-mod-5-35}
Let $\bm{\lambda}\in \mathcal{P}_{r,n}.$ Then we have

$\mathrm{(i)}$ $E^{\bm{\lambda}}$ is a tilting module. Moreover, $E^{\bm{\lambda}}\cong T^{\bm{\lambda}'}\oplus \bigoplus_{\bm{\lambda}'\rhd \bm{\mu}}e_{\bm{\lambda}\bm{\mu}}T^{\bm{\mu}}$ for some non-negative integers $e_{\bm{\lambda}\bm{\mu}}.$

$\mathrm{(ii)}$ $T^{\bm{\lambda}}$ is self-dual. Moreover, the tilting modules of $\mathrm{YS}_{n}^{r}$ are the indecomposable direct summands of the modules $\{E^{\bm{\lambda}}\:|\:\bm{\lambda}\in \mathcal{P}_{r,n}\}.$
\end{corollary}

Recall that the Schur functor $\mathrm{F}_{n}^{r}: \mathrm{YS}_{n}^{r}\mathrm{-mod}\rightarrow \text{Y}_{r,n}\mathrm{-mod}$ defined in Proposition \ref{Schur-functor-propo4-3}.

\begin{lemma}\label{tilting-mod-5-36}
Suppose that $\bm{\mu}\in \mathcal{C}_{r,n}.$ Then $\mathrm{F}_{n}^{r}(E^{\bm{\mu}})\cong N^{\bm{\mu}}$ as $\mathrm{Y}_{r,n}$-modules.
\end{lemma}
\begin{proof}
By Lemma \ref{Schur-functor-lemma4-2} and Proposition \ref{Schur-functor-propo4-3} we have
\begin{align*}
\mathrm{F}_{n}^{r}(E^{\bm{\mu}})&=\mathrm{F}_{n}^{r}(\dot{\mathrm{F}}_{n}^{\omega}(\theta_{\bm{\mu}}\dot{\mathrm{YS}_{n}^{r}}))
=\dot{\mathrm{F}}_{n}^{r}(\theta_{\bm{\mu}}\dot{\mathrm{YS}_{n}^{r}})\\
&=\theta_{\bm{\mu}}\dot{\mathrm{YS}_{n}^{r}}\varphi_{\omega}\cong \mathrm{Hom}_{\text{Y}_{r,n}}(\text{Y}_{r,n}, N^{\bm{\mu}})\\
&\cong N^{\bm{\mu}}
\end{align*}
as required.
\end{proof}

Let $\bm{\mu}, \bm{\nu}\in \mathcal{C}_{r,n}.$ Recall that for each $\mathrm{S}\in \mathcal{T}_{0}^{+}(\bm{\lambda}, \bm{\mu})$ and $\mathrm{T}\in \mathcal{T}_{0}^{+}(\bm{\lambda}, \bm{\nu})$ there is a $\text{Y}_{r,n}$-module homomorphism $\varphi_{\mathrm{S}\mathrm{T}}': N^{\bm{\nu}}\rightarrow N^{\bm{\mu}};$ this induces a $\mathrm{YS}_{n}^{r}$-module homomorphism $\Phi_{\mathrm{S}\mathrm{T}}': E^{\bm{\nu}}\rightarrow E^{\bm{\mu}}$ defined by $\Phi_{\mathrm{S}\mathrm{T}}'(\theta)=\varphi_{\mathrm{S}\mathrm{T}}'\theta$ for $\theta\in E^{\bm{\nu}}.$ The next proposition, which can be proved in exactly the same way as in [Ma2, Proposition 7.1], shows that these maps $\{\Phi_{\mathrm{S}\mathrm{T}}'\}$ give a basis of all the $\mathrm{YS}_{n}^{r}$-module homomorphisms from $E^{\bm{\nu}}$ to $E^{\bm{\mu}}.$

\begin{proposition}\label{tilting-mod-5-37}
Suppose that $\bm{\mu}, \bm{\nu}\in \mathcal{C}_{r,n}.$ Then $\mathrm{Hom}_{\mathrm{YS}_{n}^{r}}(E^{\bm{\nu}}, E^{\bm{\mu}})$ is free as an $\mathcal{R}$-module with basis $$\{\Phi_{\mathrm{S}\mathrm{T}}'\:|\:\mathrm{S}\in \mathcal{T}_{0}^{+}(\bm{\lambda}, \bm{\mu}) ~and ~\mathrm{T}\in \mathcal{T}_{0}^{+}(\bm{\lambda}, \bm{\nu}) ~for~some~\bm{\lambda}\in \mathcal{P}_{r,n}\}.$$
\end{proposition}

By definition $E_{n}^{r}=\bigoplus_{\bm{\mu}\in \mathcal{C}_{r,n}}E^{\bm{\mu}}$ is a full tilting module for $\mathrm{YS}_{n}^{r}.$ Define the Ringel dual of $\mathrm{YS}_{n}^{r}$ to be the algebra $\mathrm{End}_{\mathrm{YS}_{n}^{r}}(E_{n}^{r}).$ If $A$ is an algebra, let $A^{\mathrm{op}}$ be the opposite algebra in which the order of multiplication is reserved. The following corollary gives a description of the Ringel dual of $\mathrm{YS}_{n}^{r}.$

\begin{corollary}\label{tilting-mod-5-38}
There exist canonical isomorphisms of $\mathcal{R}$-algebras $$\mathrm{End}_{\mathrm{YS}_{n}^{r}}(E_{n}^{r})\cong (\mathrm{YS}_{r}^{n})^{\mathrm{op}}.$$
\end{corollary}

\begin{corollary}\label{tilting-mod-5-39}
Suppose that $\bm{\lambda}\in \mathcal{P}_{r,n}.$ Then $\mathrm{F}_{n}^{r}(T^{\bm{\lambda}'})\cong Y_{\bm{\lambda}}$ as $\mathrm{Y}_{r,n}$-modules.
\end{corollary}
\begin{proof}
By Lemma \ref{tilting-mod-5-36} the natural map $\mathrm{Hom}_{\text{Y}_{r,n}}(N^{\bm{\nu}}, N^{\bm{\mu}})\rightarrow \mathrm{Hom}_{\mathrm{YS}_{n}^{r}}(E^{\bm{\nu}}, E^{\bm{\mu}})$ is injective; by Proposition \ref{tiltmod-proposi5-9}(i) and Proposition \ref{tilting-mod-5-37} this is an isomorphism. Consequently, if an indecomposable tilting module $T^{\bm{\nu}}$ is a direct summand of $E^{\bm{\lambda}}$ then $\mathrm{F}_{n}^{r}(T^{\bm{\nu}})$ is an indecomposable direct summand of $N^{\bm{\lambda}}.$ Now, $E^{\bm{\lambda}}\cong T^{\bm{\lambda}'}\oplus \bigoplus_{\bm{\lambda}'\rhd \bm{\mu}}e_{\bm{\lambda}\bm{\mu}}T^{\bm{\mu}}$ by Corollary \ref{tilting-mod-5-35}(i) and $\mathrm{F}_{n}^{r}(E^{\bm{\lambda}})\cong N^{\bm{\lambda}}\cong Y_{\bm{\lambda}}\oplus \bigoplus_{\bm{\nu}\rhd\bm{\lambda}}c_{\bm{\nu}\bm{\lambda}}Y_{\bm{\nu}}$ by Proposition \ref{tiltmod-proposi5-10}(iii). Hence, the result follows by induction on the dominance order.
\end{proof}

\begin{corollary}\label{tilting-mod-5-40}
Let $\bm{\lambda}, \bm{\mu}\in \mathcal{P}_{r,n}.$ Then $[T^{\bm{\lambda}}: (W^{\bm{\mu}})^{\circledast}]=[W^{\bm{\mu}'}: L^{\bm{\lambda}'}].$
\end{corollary}
\begin{proof}
Recall that $\mathrm{F}_{n}^{r}(W^{\bm{\mu}})\cong S^{\bm{\mu}}$ by Proposition \ref{Schur-functor-propo4-3}. By definition, it is easy to see that the functor $\mathrm{F}_{n}^{r}$ commutes with duality. Then we have $\mathrm{F}_{n}^{r}((W^{\bm{\mu}})^{\circledast})\cong (\mathrm{F}_{n}^{r}(W^{\bm{\mu}}))^{\circledast}\cong S_{\bm{\mu}'}$ by Corollary \ref{idempotent-corollary-5-20}. Thus we have
\begin{align*}
[T^{\bm{\lambda}}: (W^{\bm{\mu}})^{\circledast}]&=[\mathrm{F}_{n}^{r}(T^{\bm{\lambda}}): \mathrm{F}_{n}^{r}((W^{\bm{\mu}})^{\circledast})]=[Y_{\bm{\lambda}'}: S_{\bm{\mu}'}]\\&=[W_{\bm{\mu}'}: L_{\bm{\lambda}'}]=[W^{\bm{\mu}'}: L^{\bm{\lambda}'}],
\end{align*}
where the last equality follows from Proposition \ref{tiltmod-proposi5-9}(iii).
\end{proof}

\section{Appendix. Cyclotomic Yokonuma-Schur algebras}

In this appendix, we will generalize the results above to define and study the cyclotomic Yokonuma-Schur algebra by using the cellular basis of the cyclotomic Yokonuma-Hecke algebra $\text{Y}_{r,n}^{d}$ constructed in [C1]. Since this approach is very similar, we only mention the main results and skip all the details of the proofs.

We first recall the definition of $\text{Y}_{r,n}^{d}$ and the construction of a cellular basis of it.

By definition, The affine Yokonuma-Hecke algebra $\widehat{\mathrm{Y}}_{r,n}=\widehat{\mathrm{Y}}_{r,n}(q)$ is an $\mathcal{R}$-associative algebra generated by the elements $t_{1},\ldots,t_{n},g_{1},\ldots,g_{n-1},X_{1}^{\pm 1},$ in which the generators $t_{1},\ldots,t_{n},g_{1},$ $\ldots,g_{n-1}$ satisfy the following relations:
\begin{equation}\label{rel-def-Y1}\begin{array}{rclcl}
g_ig_j\hspace*{-7pt}&=&\hspace*{-7pt}g_jg_i && \mbox{for all $i,j=1,\ldots,n-1$ such that $\vert i-j\vert \geq 2$;}\\[0.1em]
g_ig_{i+1}g_i\hspace*{-7pt}&=&\hspace*{-7pt}g_{i+1}g_ig_{i+1} && \mbox{for all $i=1,\ldots,n-2$;}\\[0.1em]
t_it_j\hspace*{-7pt}&=&\hspace*{-7pt}t_jt_i &&  \mbox{for all $i,j=1,\ldots,n$;}\\[0.1em]
g_it_j\hspace*{-7pt}&=&\hspace*{-7pt}t_{j s_i}g_i && \mbox{for all $i=1,\ldots,n-1$ and $j=1,\ldots,n$;}\\[0.1em]
t_i^r\hspace*{-7pt}&=&\hspace*{-7pt}1 && \mbox{for all $i=1,\ldots,n$;}\\[0.2em]
g_{i}^{2}\hspace*{-7pt}&=&\hspace*{-7pt}1+(q-q^{-1})e_{i}g_{i} && \mbox{for all $i=1,\ldots,n-1$,}
\end{array}
\end{equation}
where $s_{i}$ is the transposition $(i,i+1)$, and for each $1\leq i\leq n-1$,
$$e_{i} :=\frac{1}{r}\sum\limits_{s=0}^{r-1}t_{i}^{s}t_{i+1}^{-s},$$
together with the following relations concerning the generators $X_{1}^{\pm1}$:
\begin{equation}\label{rel-def-Y2}\begin{array}{rclcl}
X_{1}X_{1}^{-1}\hspace*{-7pt}&=&\hspace*{-7pt}X_{1}^{-1}X_{1}=1;&&\\[0.1em]
g_{1}X_{1}g_{1}X_{1}\hspace*{-7pt}&=&\hspace*{-7pt}X_{1}g_{1}X_{1}g_{1};\\[0.1em]
g_{i}X_{1}\hspace*{-7pt}&=&\hspace*{-7pt}X_{1}g_{i} && \mbox{for all $i=2,\ldots,n-1$;}\\[0.1em]
t_{j}X_{1}\hspace*{-7pt}&=&\hspace*{-7pt}X_{1}t_{j} && \mbox{for all $j=1,\ldots,n$.}
\end{array}
\end{equation}

We define inductively the following elements in $\widehat{\mathrm{Y}}_{r,n}$:
\begin{equation}\label{JM-elements}
X_{i+1} :=g_{i}X_ig_i\quad\mbox{for}~i=1,\ldots,n-1.
\end{equation}

Let $d\geq 1$ and $v_1,\ldots,v_d$ be some invertible indeterminates. Set $f_{1} :=(X_{1}-v_{1})\cdots (X_{1}-v_{d}).$ Let $\mathcal{J}_{d}$ denote the two-sided ideal of $\widehat{\mathrm{Y}}_{r,n}$ generated by $f_{1},$ and define the cyclotomic Yokonuma-Hecke algebra $\mathrm{Y}_{r,n}^{d}=\mathrm{Y}_{r,n}^{d}(q)$ to be the quotient $$\mathrm{Y}_{r,n}^{d}=\widehat{\mathrm{Y}}_{r,n}/\mathcal{J}_{d}.$$

It has been proved in [C1] (see also [ChPA2, Theorem 4.15]) that the set of the following elements $$\{t_{1}^{\beta_1}\cdots t_{n}^{\beta_n}X_{1}^{\alpha_1}\cdots X_{n}^{\alpha_n}g_{w}\:|\:0\leq\alpha_{1},\ldots,\alpha_{n}\leq d-1, ~0\leq \beta_{1},\ldots,\beta_{n}\leq r-1, ~w\in \mathfrak{S}_{n}\}$$
forms an $\mathcal{R}$-basis of $\text{Y}_{r,n}^{d}.$

Let $d\in \mathbb{Z}_{\geq 1}.$ Following [ChPA2, Section 3.1], the combinatorial objects appearing in the representation theory of the cyclotomic Yokonuma-Hecke algebra $\text{Y}_{r,n}^{d}$ will be $m$-compositions (resp. $m$-partitions) with $m=rd,$ which can also be identified with $r$-tuples of $d$-compositions (resp. $d$-partitions). We will call such an object an $(r,d)$-composition (resp. $(r,d)$-partition). By definition, an $(r,d)$-composition (resp. $(r,d)$-partition) of $n$ is an ordered $r$-tuple $\underline{\bm{\lambda}}=(\bm{\lambda}^{(1)},\ldots,\bm{\lambda}^{(r)})=((\lambda_{1}^{(1)},\ldots,\lambda_{d}^{(1)}),\ldots,(\lambda_{1}^{(r)},\ldots,\lambda_{d}^{(r)}))$ of $d$-compositions (resp. $d$-partitions) $(\lambda_{1}^{(k)},\ldots,\lambda_{d}^{(k)})$ ($1\leq k\leq r$) such that $\sum_{k=1}^{r}\sum_{j=1}^{d}|\lambda_{j}^{(k)}|=n.$ We denote by $\mathcal{C}_{r,n}^{d}$ (resp. $\mathcal{P}_{r,n}^{d}$) the set of $(r,d)$-compositions (resp. $(r,d)$-partitions) of $n.$ We will say that the $l$-th composition (resp. partition) of the $k$-th $r$-tuple has position $(k,l).$

A triplet $\bm{\theta}=(\theta, k, l)$ consisting of a node $\theta,$ an integer $k\in \{1,\ldots,r\},$ and an integer $l\in \{1,\ldots,d\}$ is called an $(r,d)$-node. We shall say that the $(r,d)$-node $\bm{\theta}$ has position $(k,l).$ We shall denote by $[\underline{\bm{\lambda}}]$ the set of $(r,d)$-nodes such that the subset consisting of the $(r,d)$-nodes having position $(k,l)$ forms a usual composition (resp. partition) $\lambda_{l}^{(k)}$, for any $k\in \{1,\ldots,r\}$ and $l\in \{1,\ldots,d\}.$

Let $\underline{\bm{\lambda}}=((\lambda_{1}^{(1)},\ldots,\lambda_{d}^{(1)}),\ldots,(\lambda_{1}^{(r)},\ldots,\lambda_{d}^{(r)}))$ be an $(r,d)$-composition of $n.$ An $(r,d)$-tableau $\mathfrak{t}=((\mathfrak{t}_{1}^{(1)},\ldots,\mathfrak{t}_{d}^{(1)}),\ldots,(\mathfrak{t}_{1}^{(r)},\ldots,\mathfrak{t}_{d}^{(r)}))$ of shape $\underline{\bm{\lambda}}$ is obtained by placing each $(r,d)$-node of $[\underline{\bm{\lambda}}]$ by one of the integers $1,2,\ldots,n,$ allowing no repeats. We will call the number $n$ the size of $\mathfrak{t}$ and the $\mathfrak{t}_{l}^{(k)}$'s the components of $\mathfrak{t}.$ Each $(r,d)$-node $\bm{\theta}$ of $\mathfrak{t}$ is labelled by $((a, b), k, l)$ if it lies in row $a$ and column $b$ of the component $\mathfrak{t}_{l}^{(k)}$ of $\mathfrak{t}.$

For each $\underline{\bm{\mu}}\in \mathcal{C}_{r,n}^{d},$ an $(r,d)$-tableau of shape $\underline{\bm{\mu}}$ is called row standard if the numbers increase along any row (from left to right) of each diagram in $[\underline{\bm{\mu}}].$ For each $\underline{\bm{\lambda}}\in \mathcal{P}_{r,n}^{d},$ an $(r,d)$-tableau of shape $\underline{\bm{\lambda}}$ is called standard if the numbers increase along any row (from left to right) and down any column (from top to bottom) of each diagram in $[\underline{\bm{\lambda}}].$ From now on, we denote by $\text{Std}(\underline{\bm{\lambda}})$ the set of all standard $(r,d)$-tableaux of size $n$ and of shape $\underline{\bm{\lambda}},$ which is endowed with an action of $\mathfrak{S}_{n}$ from the right by permuting the entries in each $(r,d)$-tableau.

For each $\underline{\bm{\lambda}}\in \mathcal{C}_{r,n}^{d},$ we denote by $\mathfrak{t}^{\underline{\bm{\lambda}}}$ the standard $(r,d)$-tableau of shape $\underline{\bm{\lambda}}$ in which $1,2,\ldots,n$ appear in increasing order from left to right along the rows of the first diagram, and then along the rows of the second diagram, and so on.

For each $\underline{\bm{\lambda}}=((\lambda_{1}^{(1)},\ldots,\lambda_{d}^{(1)}),\ldots,(\lambda_{1}^{(r)},\ldots,\lambda_{d}^{(r)}))\in \mathcal{C}_{r,n}^{d},$ we have a Young subgroup
$$\mathfrak{S}_{\underline{\bm{\lambda}}} :=\mathfrak{S}_{\lambda_{1}^{(1)}}\times\cdots\times\mathfrak{S}_{\lambda_{d}^{(1)}}\times\cdots\times
\mathfrak{S}_{\lambda_{1}^{(r)}}\times\cdots\times\mathfrak{S}_{\lambda_{d}^{(r)}},$$
which is exactly the row stabilizer of $\mathfrak{t}^{\underline{\bm{\lambda}}}.$

For each $\underline{\bm{\lambda}}\in \mathcal{C}_{r,n}^{d}$ and a row standard $(r,d)$-tableau $\mathfrak{s}$ of shape $\underline{\bm{\lambda}},$ let $d(\mathfrak{s})$ be the element of $\mathfrak{S}_{n}$ such that $\mathfrak{s}=\mathfrak{t}^{\underline{\bm{\lambda}}}d(\mathfrak{s}).$ Then $d(\mathfrak{s})$ is a distinguished right coset representative of $\mathfrak{S}_{\underline{\bm{\lambda}}}$ in $\mathfrak{S}_{n},$ that is, $l(wd(\mathfrak{s}))=l(w)+l(d(\mathfrak{s}))$ for any $w\in \mathfrak{S}_{\underline{\bm{\lambda}}}.$ In this way, we obtain a correspondence between the set of row standard $(r,d)$-tableaux of shape $\underline{\bm{\lambda}}$ and the set of distinguished right coset representatives of $\mathfrak{S}_{\underline{\bm{\lambda}}}$ in $\mathfrak{S}_{n}.$

We now define a partial order on the set of $(r,d)$-compositions.

\begin{definition}\label{Def-def}
Let $\underline{\bm{\lambda}}=((\lambda_{1}^{(1)},\ldots,\lambda_{d}^{(1)}),\ldots,(\lambda_{1}^{(r)},\ldots,\lambda_{d}^{(r)}))$ and $\underline{\bm{\mu}}=((\mu_{1}^{(1)},\ldots,\mu_{d}^{(1)}),$ $\ldots,(\mu_{1}^{(r)},\ldots,\mu_{d}^{(r)}))$ be two $(r,d)$-compositions of $n.$ We say that $\underline{\bm{\lambda}}$ dominates $\underline{\bm{\mu}},$ and we write $\underline{\bm{\lambda}}\unrhd \underline{\bm{\mu}}$ if and only if $$\sum_{i=1}^{k-1}\sum_{j=1}^{d}|\lambda_{j}^{(i)}|+\sum_{j=1}^{l-1}|\lambda_{j}^{(k)}|+\sum_{i=1}^{p}\lambda_{l,i}^{(k)}\geq \sum_{i=1}^{k-1}\sum_{j=1}^{d}|\mu_{j}^{(i)}|+\sum_{j=1}^{l-1}|\mu_{j}^{(k)}|+\sum_{i=1}^{p}\mu_{l,i}^{(k)}$$
for all $k,$ $l$ and $p$ with $1\leq k\leq r,$ $1\leq l\leq d$ and $p\geq 0.$ If $\underline{\bm{\lambda}}\unrhd\underline{\bm{\mu}}$ and $\underline{\bm{\lambda}}\neq \underline{\bm{\mu}},$ we write $\underline{\bm{\lambda}}\rhd \underline{\bm{\mu}}.$
\end{definition}

\begin{definition}\label{definition-1}
Let $\underline{\bm{\lambda}}=((\lambda_{1}^{(1)},\ldots,\lambda_{d}^{(1)}),\ldots,(\lambda_{1}^{(r)},\ldots,\lambda_{d}^{(r)}))\in \mathcal{C}_{r,n}^{d}.$ Suppose that we choose all $1\leq i_{1}< i_{2}<\cdots < i_{p}\leq r$ such that $(\lambda_{1}^{(i_1)},\ldots,\lambda_{d}^{(i_1)}),$ $(\lambda_{1}^{(i_2)},\ldots,\lambda_{d}^{(i_2)}),\ldots,$$(\lambda_{1}^{(i_p)},$\\$\ldots,\lambda_{d}^{(i_p)})$ are nonempty. Define $a_{k} :=\sum_{j=1}^{k}|\bm{\lambda}^{(i_{j})}|$ for $1\leq k\leq p,$ where $|\bm{\lambda}^{(i_{j})}|=\sum_{l=1}^{d}|\lambda_{l}^{(i_{j})}|.$ Then the set partition $A_{\underline{\bm{\lambda}}}$ associated with $\underline{\bm{\lambda}}$ is defined as $$A_{\underline{\bm{\lambda}}} :=\{\{1,\ldots,a_{1}\},\{a_{1}+1,\ldots,a_{2}\},\ldots,\{a_{p-1}+1,\ldots,n\}\},$$ which may be written as $A_{\underline{\bm{\lambda}}}=\{I_{1},I_{2},\ldots,I_{p}\},$ and is referred to the blocks of $A_{\underline{\bm{\lambda}}}$ in the order given above.
\end{definition}

\begin{definition}\label{definition-2}
Let $\underline{\bm{\lambda}}=((\lambda_{1}^{(1)},\ldots,\lambda_{d}^{(1)}),\ldots,(\lambda_{1}^{(r)},\ldots,\lambda_{d}^{(r)}))\in \mathcal{C}_{r,n}^{d},$ and let $a_{k} :=\sum_{j=1}^{k}|\bm{\lambda}^{(i_{j})}|$ $(1\leq k\leq p)$ be defined as above. Then we define $$u_{\underline{\bm{\lambda}}} :=u_{a_{1},i_{1}}u_{a_{2},i_{2}}\cdots u_{a_{p},i_{p}}.$$
\end{definition}

\begin{definition}\label{definition-3}
Let $\underline{\bm{\lambda}}=((\lambda_{1}^{(1)},\ldots,\lambda_{d}^{(1)}),\ldots,(\lambda_{1}^{(r)},\ldots,\lambda_{d}^{(r)}))\in \mathcal{C}_{r,n}^{d}.$ Associated with $\underline{\bm{\lambda}}$ we can define the following elements $a_{l}^{k}$ and $b_{k}$:$$a_{l}^{k} :=\sum_{m=1}^{l-1}|\lambda_{m}^{(k)}|,~~~~b_{k} :=\sum_{j=1}^{k-1}\sum_{i=1}^{d}|\lambda_{i}^{(j)}|~~~~\mathrm{for}~1\leq k\leq r~\mathrm{and}~1\leq l\leq d.$$ Associated with these elements we can define an element $u_{\mathbf{a}}^{+} :=u_{\mathbf{a}, 1}u_{\mathbf{a}, 2}\cdots u_{\mathbf{a}, r},$ where $$u_{\mathbf{a}, k} :=\prod_{l=1}^{d}\prod_{j=1}^{a_{l}^{k}}(X_{b_{k}+j}-v_{l}).$$
\end{definition}

We can now define the key ingredient of the cellular basis for $Y_{r,n}^{d}.$

\begin{definition}\label{definition-4}
Let $\underline{\bm{\lambda}}\in \mathcal{C}_{r,n}^{d}$ and define $u_{\mathbf{a}}^{+}$ as above. Let $x_{\underline{\bm{\lambda}}}=\sum_{w\in \mathfrak{S}_{\underline{\bm{\lambda}}}}q^{l(w)}g_{w}.$ Then we define the element $m_{\underline{\bm{\lambda}}}$ of $Y_{r,n}^{d}$ as follows:
\begin{equation}\label{mUlam-lambda}
m_{\underline{\bm{\lambda}}} :=U_{\underline{\bm{\lambda}}}u_{\mathbf{a}}^{+}x_{\underline{\bm{\lambda}}}=
u_{\underline{\bm{\lambda}}}E_{A_{\underline{\bm{\lambda}}}}u_{\mathbf{a}}^{+}x_{\underline{\bm{\lambda}}}.
\end{equation}
\end{definition}

Let $\ast$ denote the $\mathcal{R}$-linear anti-automorphism of $\text{Y}_{r,n}^{d}$, which is determined by
$$g_{i}^{\ast}=g_{i},\quad t_{j}^{\ast}=t_{j},\quad X_{j}^{\ast}=X_{j}\quad \mathrm{for}~1\leq i\leq n-1~\mathrm{and}~1\leq j\leq n.$$

\begin{definition}
Let $\underline{\bm{\lambda}}\in \mathcal{C}_{r,n}^{d},$ and let $\mathfrak{s}$ and $\mathfrak{t}$ be two row standard $(r,d)$-tableaux of shape $\underline{\bm{\lambda}}.$ We then define $m_{\mathfrak{s}\mathfrak{t}}=g_{d(\mathfrak{s})}^{\ast}m_{\underline{\bm{\lambda}}}g_{d(\mathfrak{t})}.$
\end{definition}

For each $\underline{\bm{\lambda}}\in \mathcal{P}_{r,n}^{d},$ let $\text{Y}_{r,n}^{d, \rhd \underline{\bm{\lambda}}}$ be the $\mathcal{R}$-submodule of $\text{Y}_{r,n}^{d}$ spanned by $m_{\mathfrak{u}\mathfrak{v}}$ with $\mathfrak{u}, \mathfrak{v}\in \mathrm{Std}(\underline{\bm{\mu}})$ for various $\underline{\bm{\mu}}\in \mathcal{P}_{r,n}^{d}$ such that $\underline{\bm{\mu}}\rhd\underline{\bm{\lambda}}.$

\begin{theorem}\label{cycl-yokoschur-theo-6-7}
{\rm (See [C1, Theorem 6.18].)} The algebra $\mathrm{Y}_{r,n}^{d}$ is a free $\mathcal{R}$-module with a cellular basis $$\mathcal{B}_{r,n}^{d}=\{m_{\mathfrak{s}\mathfrak{t}}\:|\:\mathfrak{s}, \mathfrak{t}\in \mathrm{Std}(\underline{\bm{\lambda}})~for~some~(r,d)\mathrm{-}partition~\underline{\bm{\lambda}}~of~n\},$$ that is, the following properties hold$:$

$\mathrm{(i)}$ The $\mathcal{R}$-linear map determined by $m_{\mathfrak{s}\mathfrak{t}}\mapsto m_{\mathfrak{t}\mathfrak{s}}$ $(m_{\mathfrak{s}\mathfrak{t}}\in \mathcal{B}_{r,n}^{d})$ is an anti-automorphism on $\mathrm{Y}_{r,n}^{d}$.

$\mathrm{(ii)}$ For a given $h\in \mathrm{Y}_{r,n}^{d},$ $\underline{\bm{\mu}}\in \mathcal{P}_{r,n}^{d}$ and $\mathfrak{t}\in \mathrm{Std}(\underline{\bm{\mu}}),$ there exist $r_{\mathfrak{v}\mathfrak{t}}(h)\in \mathcal{R}$ such that for all $\mathfrak{s}\in \mathrm{Std}(\underline{\bm{\mu}})$, we have $$m_{\mathfrak{s}\mathfrak{t}}h\equiv\sum_{\mathfrak{v}\in \mathrm{Std}(\underline{\bm{\mu}})}r_{\mathfrak{v}\mathfrak{t}}(h)m_{\mathfrak{s}\mathfrak{v}}~~~~\mathrm{mod}~\text{Y}_{r,n}^{d, \rhd \underline{\bm{\mu}}},$$
where $r_{\mathfrak{v}\mathfrak{t}}(h)$ may depend on $\mathfrak{v}, \mathfrak{t}$ and $h,$ but not on $\mathfrak{s}.$
\end{theorem}

For $\underline{\bm{\lambda}}\in \mathcal{C}_{r,n}^{d},$ a $\underline{\bm{\lambda}}$-tableau $\mathrm{S}=((S_{1}^{(1)},\ldots, S_{d}^{(1)}),\ldots,(S_{1}^{(r)},\ldots, S_{d}^{(r)}))$ is a map $\mathrm{S}: [\underline{\bm{\lambda}}]\rightarrow \{1,\ldots,n\}\times \{1,\ldots,d\}\times \{1,\ldots,r\},$ which can be regarded as the diagram $[\underline{\bm{\lambda}}],$ together with an ordered triple $(i, j, k)$ ($1\leq i\leq n,$ $1\leq j\leq d,$ $1\leq k\leq r$) attached to each node. Given $\underline{\bm{\lambda}}\in \mathcal{P}_{r,n}^{d}$ and $\underline{\bm{\mu}}\in \mathcal{C}_{r,n}^{d}$, a $\underline{\bm{\lambda}}$-tableau $\mathrm{S}$ is said to be of type $\underline{\bm{\mu}}$ if the number of $(i,j,k)$ in the entry of $\mathrm{S}$ is equal to $\mu_{j, i}^{(k)}.$ Given $\mathfrak{s}\in \mathrm{Std}(\underline{\bm{\lambda}})$, $\underline{\bm{\mu}}(\mathfrak{s})$, a $\underline{\bm{\lambda}}$-tableau of type $\underline{\bm{\mu}},$ is defined by replacing each entry $m$ in $\mathfrak{s}$ by $(i,j,k)$ if $m$ is in the $i$-th row of the $(k,j)$-th component of $\mathfrak{t}^{\underline{\bm{\mu}}}.$

We define a total order on the set of triples $(i,j,k)$ by $(i_{1}, j_{1}, k_{1})< (i_{2}, j_{2}, k_{2})$ if $k_{1}< k_{2},$ or $k_{1}=k_{2}$ and $j_{1}<j_{2},$ or $k_{1}=k_{2},$ $j_{1}=j_{2},$ and $i_{1}<i_{2},$ Let $\underline{\bm{\lambda}}\in \mathcal{P}_{r,n}^{d}$ and $\underline{\bm{\mu}}\in \mathcal{C}_{r,n}^{d}$. Suppose that $\mathrm{S}=((S_{1}^{(1)},\ldots, S_{d}^{(1)}),\ldots,(S_{1}^{(r)},\ldots, S_{d}^{(r)}))$ is a $\underline{\bm{\lambda}}$-tableau of type $\underline{\bm{\mu}}.$ $\mathrm{S}$ is said to be semistandard if each component $S^{(k)}_{j}$ is non-decreasing in rows, strictly increasing in columns, and all entries of $S^{(k)}_{j}$ are of the form $(i,h,l)$ with $h\geq j$ and $l\geq k.$ We denote by $\mathcal{T}_{0}(\underline{\bm{\lambda}}, \underline{\bm{\mu}})$ the set of semistandard $\underline{\bm{\lambda}}$-tableaux of type $\underline{\bm{\mu}}.$

For any $\underline{\bm{\kappa}}\in \mathcal{C}_{r,n}^{d},$ we define its type $\alpha(\underline{\bm{\kappa}})$ by $\alpha(\underline{\bm{\kappa}})=(n_{1},\ldots, n_{r})$ with $n_{i}=|\bm{\kappa}^{(i)}|.$ Assume that $\underline{\bm{\lambda}}\in \mathcal{P}_{r,n}^{d}$ and $\underline{\bm{\mu}}\in \mathcal{C}_{r,n}^{d}.$ We define a subset $\mathcal{T}_{0}^{+}(\underline{\bm{\lambda}}, \underline{\bm{\mu}})$ of $\mathcal{T}_{0}(\underline{\bm{\lambda}}, \underline{\bm{\mu}})$ by $$\mathcal{T}_{0}^{+}(\underline{\bm{\lambda}}, \underline{\bm{\mu}})=\{\mathrm{S}\in \mathcal{T}_{0}(\underline{\bm{\lambda}}, \underline{\bm{\mu}})\:|\:\alpha(\underline{\bm{\lambda}})=\alpha(\underline{\bm{\mu}})\}.$$

For each $\underline{\bm{\mu}}\in \mathcal{C}_{r,n}^{d},$ let $M^{\underline{\bm{\mu}}}=m_{\underline{\bm{\mu}}}\text{Y}_{r,n}^{d}.$ We now construct a basis of $M^{\underline{\bm{\mu}}}$ related to the cellular basis $\{m_{\mathfrak{s}\mathfrak{t}}\}$ in Theorem \ref{cycl-yokoschur-theo-6-7}. For $\mathrm{S}\in \mathcal{T}_{0}^{+}(\underline{\bm{\lambda}}, \underline{\bm{\mu}})$ and $\mathfrak{t}\in \mathrm{Std}(\underline{\bm{\lambda}}),$ we define $$m_{\mathrm{S}\mathfrak{t}}=\sum_{\substack{\mathfrak{s}\in \mathrm{Std}(\underline{\bm{\lambda}})\\\underline{\bm{\mu}}(\mathfrak{s})=\mathrm{S}}}q^{l(d(\mathfrak{s}))+l(d(\mathfrak{t}))}m_{\mathfrak{s}\mathfrak{t}}.$$
The following theorem can be proved in exactly the same way as in [DJM, Theorem 4.14] by combining [DJM] with [C1].

\begin{theorem}\label{cycl-yokoschur-theo-6-8}
Let $\mathrm{S}\in \mathcal{T}_{0}^{+}(\underline{\bm{\lambda}}, \underline{\bm{\mu}})$ and $\mathfrak{t}\in \mathrm{Std}(\underline{\bm{\lambda}})$ for some $\underline{\bm{\lambda}}\in \mathcal{P}_{r,n}^{d}$ and $\underline{\bm{\mu}}\in \mathcal{C}_{r,n}^{d}.$ Then $m_{\mathrm{S}\mathfrak{t}}\in M^{\underline{\bm{\mu}}}.$ Moreover, $M^{\underline{\bm{\mu}}}$ is free with an $\mathcal{R}$-basis
$$\big\{m_{\mathrm{S}\mathfrak{t}}\:|\:\mathrm{S}\in \mathcal{T}_{0}^{+}(\underline{\bm{\lambda}}, \underline{\bm{\mu}})~and~\mathfrak{t}\in \mathrm{Std}(\underline{\bm{\lambda}})~for~some~\underline{\bm{\lambda}}\in \mathcal{P}_{r,n}^{d}\big\}.$$
\end{theorem}

Let $\underline{\bm{\mu}}, \underline{\bm{\nu}}\in \mathcal{C}_{r,n}^{d}$ and $\underline{\bm{\lambda}}\in \mathcal{P}_{r,n}^{d}.$ We assume that $\alpha(\underline{\bm{\mu}})=\alpha(\underline{\bm{\nu}})=\alpha(\underline{\bm{\lambda}}).$ For $\mathrm{S}\in \mathcal{T}_{0}^{+}(\underline{\bm{\lambda}}, \underline{\bm{\mu}})$, $\mathrm{T}\in \mathcal{T}_{0}^{+}(\underline{\bm{\lambda}}, \underline{\bm{\nu}}),$ put $$m_{\mathrm{S}\mathrm{T}}=\sum_{\mathfrak{s}, \mathfrak{t}}q^{l(d(\mathfrak{s}))+l(d(\mathfrak{t}))}m_{\mathfrak{s}\mathfrak{t}},$$ where the sum is taken over all $\mathfrak{s}, \mathfrak{t}\in \mathrm{Std}(\underline{\bm{\lambda}})$ such that $\underline{\bm{\mu}}(\mathfrak{s})=\mathrm{S}$ and $\underline{\bm{\nu}}(\mathfrak{t})=\mathrm{T}.$ We then have the next proposition by making use of Theorem \ref{cycl-yokoschur-theo-6-8}.

\begin{proposition}\label{cycl-yokoschur-theo-6-9}
Suppose that $\underline{\bm{\mu}}, \underline{\bm{\nu}}\in \mathcal{C}_{r,n}^{d}$ with $\alpha(\underline{\bm{\mu}})=\alpha(\underline{\bm{\nu}})$. Then the set $$\{m_{\mathrm{S}\mathrm{T}}\:|\:\mathrm{S}\in \mathcal{T}_{0}^{+}(\underline{\bm{\lambda}}, \underline{\bm{\mu}})~and~\mathrm{T}\in \mathcal{T}_{0}^{+}(\underline{\bm{\lambda}}, \underline{\bm{\nu}})~for~some~\underline{\bm{\lambda}}\in \mathcal{P}_{r,n}^{d}\}$$ is an $\mathcal{R}$-basis of $M^{\underline{\bm{\nu}}\ast}\cap M^{\underline{\bm{\mu}}}.$
\end{proposition}

\begin{definition}
Suppose that $M_{n}^{r,d}=\bigoplus_{\underline{\bm{\mu}}\in \mathcal{C}_{r,n}^{d}}M^{\underline{\bm{\mu}}}.$ We define the cyclotomic Yokonuma-Schur algebra YS$_{n}^{r,d}$ as the endomorphism algebra $$\mathrm{YS}_{n}^{r,d}=\mathrm{End}_{\text{Y}_{r,n}^{d}}(M_{n}^{r,d}),$$ which is isomorphic to $\bigoplus_{\underline{\bm{\mu}}, \underline{\bm{\nu}}\in \mathcal{C}_{r,n}^{d}}\mathrm{Hom}_{\text{Y}_{r,n}^{d}}(M^{\underline{\bm{\nu}}}, M^{\underline{\bm{\mu}}}).$
\end{definition}

Let $\mathrm{S}\in \mathcal{T}_{0}^{+}(\underline{\bm{\lambda}}, \underline{\bm{\mu}})$ and $\mathrm{T}\in \mathcal{T}_{0}^{+}(\underline{\bm{\lambda}}, \underline{\bm{\nu}}).$ In view of Proposition \ref{cycl-yokoschur-theo-6-9}, we can define $\varphi_{\mathrm{S}\mathrm{T}}\in \mathrm{Hom}_{\text{Y}_{r,n}^{d}}(M^{\underline{\bm{\nu}}}, M^{\underline{\bm{\mu}}})$ by $$\varphi_{\mathrm{S}\mathrm{T}}(m_{\bm{\nu}}h)=m_{\mathrm{S}\mathrm{T}}h$$ for all $h\in \text{Y}_{r,n}^{d}.$ We extend $\varphi_{\mathrm{S}\mathrm{T}}$ to an element of YS$_{n}^{r,d}$ by defining $\varphi_{\mathrm{S}\mathrm{T}}$ to be zero on $M^{\underline{\bm{\kappa}}}$ for any $\underline{\bm{\nu}}\neq \underline{\bm{\kappa}}\in \mathcal{C}_{r,n}^{d}.$ For each $\underline{\bm{\lambda}}\in \mathcal{P}_{r,n}^{d}$, let $\mathcal{T}_{0}^{+}(\underline{\bm{\lambda}})=\bigcup_{\underline{\bm{\mu}}\in \mathcal{C}_{r,n}^{d}}\mathcal{T}_{0}^{+}(\underline{\bm{\lambda}}, \underline{\bm{\mu}}).$ We denote by YS$_{r,n}^{d,\rhd \underline{\bm{\lambda}}}$ the $\mathcal{R}$-submodule of YS$_{n}^{r,d}$ spanned by $\varphi_{\mathrm{S}\mathrm{T}}$ such that $\mathrm{S}, \mathrm{T}\in \mathcal{T}_{0}^{+}(\underline{\bm{\alpha}})$ with $\underline{\bm{\alpha}}\rhd \underline{\bm{\lambda}}.$ Then we can prove the following theorem by a similar argument as in [DJM, Theorem 6.6].

\begin{theorem}\label{cycl-yokoschur-theo-6-11}
The Yokonuma-Schur algebra $\mathrm{YS}_{n}^{r,d}$ is free as an $\mathcal{R}$-module with a basis $$\big\{\varphi_{\mathrm{S}\mathrm{T}}\:|\:\mathrm{S}, \mathrm{T}\in \mathcal{T}_{0}^{+}(\underline{\bm{\lambda}})~for~some~\underline{\bm{\lambda}}\in \mathcal{P}_{r,n}^{d}\big\}.$$ Moreover, this basis satisfies the following properties.

$\mathrm{(i)}$ The $\mathcal{R}$-linear map $\ast: \mathrm{YS}_{n}^{r,d}\rightarrow \mathrm{YS}_{n}^{r,d}$ determined by $\varphi_{\mathrm{S}\mathrm{T}}^{\ast}=\varphi_{\mathrm{T}\mathrm{S}},$ for all $\mathrm{S}, \mathrm{T}\in \mathcal{T}_{0}^{+}(\underline{\bm{\lambda}})$ and all $\underline{\bm{\lambda}}\in \mathcal{P}_{r,n}^{d},$ is an anti-automorphism of $\mathrm{YS}_{n}^{r,d}.$

$\mathrm{(ii)}$ Let $\mathrm{T}\in \mathcal{T}_{0}^{+}(\underline{\bm{\lambda}})$ and $\varphi\in \mathrm{YS}_{n}^{r,d}.$ Then for each $\mathrm{V}\in \mathcal{T}_{0}^{+}(\underline{\bm{\lambda}})$, there exists $r_{\mathrm{V}}=r_{\mathrm{V},\mathrm{T},\varphi}\in \mathcal{R}$ such that for all $\mathrm{S}\in \mathcal{T}_{0}^{+}(\underline{\bm{\lambda}}),$ we have $$\varphi_{\mathrm{S}\mathrm{T}}\varphi\equiv\sum_{\mathrm{V}\in \mathcal{T}_{0}^{+}(\underline{\bm{\lambda}})}r_{\mathrm{V}}\varphi_{\mathrm{S}\mathrm{V}}~~~~\mathrm{mod}~\mathrm{YS}_{r,n}^{d,\rhd \underline{\bm{\lambda}}}.$$
In particular, this basis $\{\varphi_{\mathrm{S}\mathrm{T}}\}$ is a cellular basis of $\mathrm{YS}_{n}^{r,d}.$
\end{theorem}

Now we can apply the general theory of cellular algebras in view of Theorem \ref{cycl-yokoschur-theo-6-11}. For example, we can easily give a complete set of non-isomorphic irreducible $\mathrm{YS}_{n}^{r,d}$-modules over an arbitrary field, and further prove that $\mathrm{YS}_{n}^{r,d}$ is a quasi-hereditary algebra. For the cyclotomic Yokonuma-Schur algebra $\mathrm{YS}_{n}^{r,d},$ we can also define the Schur functor from the category of $\mathrm{YS}_{n}^{r,d}$-modules to the category of $\text{Y}_{r,n}^{d}$-modules and the tilting modules for it in exactly the same way as in Sections 4 and 5, we skip all the details and leave them to the reader.



\end{document}